\documentclass[12pt,reqno]{amsart}

\usepackage{tikz}
\usetikzlibrary{calc}

\usepackage{pgfplots}

\usepackage{amsmath,amssymb,ifthen}
\usepackage{fullpage}
\usepackage{graphicx,psfrag,subfigure}
\usepackage{color}

\usepackage{moreverb}

\def\A{\mathbb A}

\def\P{{\mathbb P}}
\def\R{{\mathbb R}}
\def\N{{\mathbb N}}
\def\V{{\mathbb V}}

\def\EE{{\mathcal E}}

\def\KK{{\mathcal K}}
\def\LL{{\mathcal L}}
\def\MM{{\mathcal M}}
\def\OO{{\mathcal O}}
\def\PP{{\mathcal P}}
\def\RR{{\mathcal R}}
\def\SS{{\mathcal S}}
\def\TT{{\mathcal T}}
\def\UU{{\mathcal U}}

\def\sz{{\mathbb J}}

\def\diam{{\rm diam}}

\newcommand{\norm}[3][]{#1\|#2#1\|_{#3}}

\def\set#1#2{\big\{#1\,:\,#2\big\}}

\def\eps{\varepsilon}
\def\dual#1#2{\langle#1\,,\,#2\rangle}

\def\mcup{\mbox{$\bigcup$}}

\newcounter{constantsnumber}
\def\namec#1#2{%
  \ifthenelse{\equal{#1}{rho:reliable}}{C_{\rm rel}}{%
  \ifthenelse{\equal{#1}{rho:efficient}}{C_{\rm eff}}{%
  \ifthenelse{\equal{#1}{stability}}{C_{\rm stab}}{%
  \ifthenelse{\equal{#1}{equivalence1}}{C_{\rm low}}{%
  \ifthenelse{\equal{#1}{equivalence2}}{C_{\rm high}}{%
  \ifthenelse{\equal{#1}{optimal}}{C_{\rm opt}}{%
  \ifthenelse{\equal{#1}{scottzhang}}{C_{\rm sz}}{%
  \ifthenelse{\equal{#1}{dirichlet0}}{C_{\rm dir}}{%
  \ifthenelse{\equal{#1}{dirichlet}}{C_{\rm osc}}{%
  \ifthenelse{\equal{#1}{equivalence}}{C_{\rm eq}}{%
  \ifthenelse{\equal{#1}{pythagoras}}{C_{\rm pyth}}{%
  \ifthenelse{\equal{#1}{dlr}}{C_{\rm dlr}}{%
  \ifthenelse{\equal{#1}{nvb}}{C_{\rm nvb}}{%
  \ifthenelse{\equal{#1}{reduction}}{C_{\rm red}}{%
  \ifthenelse{\equal{#1}{refined}}{C_{\rm ref}}{%
  \ifthenelse{\equal{#1}{estconv}}{C_{\rm est}}{%
  \ifthenelse{\equal{#1}{cea}}{C_{\mbox{\rm\scriptsize C\'ea}}}{%
  \ifthenelse{\equal{#2}{newcounter}}{\refstepcounter{constantsnumber}\label{const#1}}{}C_{\ref{const#1}}}%
}}}}}}}}}}}}}}}}}
\def\setc#1{\namec{#1}{newcounter}}
\def\c#1{\namec{#1}{reference}}

\newcounter{contractionnumber}
\def\nameq#1#2{%
  \ifthenelse{\equal{#1}{reduction}}{q_{\rm red}}{%
  \ifthenelse{\equal{#1}{estconv}}{q_{\rm est}}{%
  \ifthenelse{\equal{#1}{cea}}{q_{\mbox{\scriptsize C\'ea}}}{%
  \ifthenelse{\equal{#2}{newcounter}}{\refstepcounter{contractionnumber}\label{contraction#1}}{}q_{\ref{contraction#1}}}%
}}}

\def\q#1{\nameq{#1}{reference}}

\def\oscT#1{{\rm osc}_{\TT,#1}}

\def\oscD#1{{\rm osc}_{D,#1}}
\def\oscN#1{{\rm osc}_{N,#1}}

\newtheorem{theorem}{Theorem}
\newtheorem{proposition}[theorem]{Proposition}
\newtheorem{lemma}[theorem]{Lemma}
\newtheorem{corollary}[theorem]{Corollary}
\newtheorem{algorithm}[theorem]{Algorithm}

\newenvironment{remark}{\medskip\noindent\textbf{Remark.}\ \it}{\qed\smallskip}

\def\subsection#1
{
 \bigskip

 \refstepcounter{subsection}
 {\noindent\bf\arabic{section}.\arabic{subsection}.~#1.~}
}

\def\revision#1{{\color{black}{#1}}}

\def\refine{{\texttt{refine}}}
\def\T{\mathbb T}

\begin{document}

\title[AFEM with inhomogeneous Dirichlet data]%
{Each $H^{1/2}$-stable projection yields convergence and quasi-optimality of adaptive FEM with inhomogeneous Dirichlet data in $\R^d$}
\date{\today}

\author{M.~Aurada}
\author{M.~Feischl}
\author{J.~Kemetm\"uller}
\author{M.~Page}
\author{D.~Praetorius}
\address{Institute for Analysis and Scientific Computing,
       Vienna University of Technology,
       Wiedner Hauptstra\ss{}e 8-10,
       A-1040 Wien, Austria}
\email{\{Markus.Aurada\,,\,Michael.Feischl\,,\,Josef.Kemetmueller\,,\,Dirk.Praetorius\}@tuwien.ac.at}
\email{Marcus.Page@tuwien.ac.at\quad\rm(corresponding author)}


\begin{abstract}
We consider the solution of second order elliptic PDEs in $\R^d$ with inhomogeneous Dirichlet data by means of an $h$-adaptive FEM with fixed polynomial order $p\in\N$. As model example serves the Poisson equation with mixed Dirichlet-Neumann boundary conditions, where the inhomogeneous Dirichlet data are discretized by use of 
\revision{an $H^{1/2}$-stable} projection, for instance, the $L^2$-projection for $p=1$
or the Scott-Zhang projection for general $p\ge1$. For error estimation, we use a residual error estimator which includes the Dirichlet data oscillations. We prove
\revision{that each $H^{1/2}$-stable projection yields}
 convergence of the adaptive algorithm even with quasi-optimal convergence rate. Numerical experiments \revision{with the $L^2$- and Scott-Zhang projection} 
conclude the work.
\end{abstract}


\maketitle

\section{Introduction}
\label{section:introduction}%

\noindent
Recently, there has been a major breakthrough in the thorough mathematical understanding of convergence and quasi-optimality of $h$-adaptive FEM for second-order elliptic PDEs. However, the focus of the numerical analysis usually lies on model problems with homogeneous Dirichlet conditions,
i.e.\ $\Delta u = f$ in $\Omega$ with $u=0$ on $\Gamma=\partial\Omega$,
see e.g.~\cite{bdd,ckns,doerfler,ks,mns,stevenson}. Instead, our model problem
\begin{align}\label{eq:strongform}
\begin{split}
 -\Delta u &= f \quad \text{in } \Omega,\\
 u &= g \quad \text{on } \Gamma_D,\\
 \partial_n u &= \phi \quad \text{on } \Gamma_N
\end{split}
\end{align}
considers \emph{inhomogeneous} mixed Dirichlet-Neumann boundary conditions.
Here, $\Omega$ is a bounded Lipschitz domain in $\R^d$ with polyhedral boundary $\Gamma = \partial\Omega$ which is split into two \revision{(possibly non-connected)} relatively open boundary parts, namely the Dirichlet boundary $\Gamma_D$ and the Neumann boundary $\Gamma_N$, i.e.\ $\Gamma_D \cap \Gamma_N = \emptyset$ and $\overline\Gamma_D \cup \overline\Gamma_N = \Gamma$. We stress that the surface measure of the Dirichlet boundary has to be positive $|\Gamma_D|>0$,
whereas $\Gamma_N$ is allowed to be empty.
The given data formally satisfy $f \in \widetilde H^{-1}(\Omega)$, $g \in H^{1/2}(\Gamma_D)$, and $\phi \in H^{-1/2}(\Gamma_N)$.
\revision{We refer to Section~\ref{section:spaces} below for the definition of these Sobolev spaces.}
As is usually required to derive (localized) a~posteriori error estimators, we assume additional regularity of the given data, namely $f \in L^2(\Omega)$, $g \in H^1(\Gamma_D)$, and $\phi \in L^2(\Gamma_N)$. \revision{Moreover, we assume that the boundary partition into $\Gamma_D$ and $\Gamma_N$ is resolved by the triangulations used.}

We stress that ---using results available in the literature--- it is easily possible to generalize the analysis from the Laplacian $L=-\Delta$ to general symmetric and uniformly elliptic differential
operators of second order. The reader is referred to the seminal work~\cite{ckns}
which treats the case of \revision{$\Gamma_D = \partial\Omega$ and} homogeneous Dirichlet data $g=0$ and
provides the \revision{analytical} tools to cover general $L$.
\revision{Therefore, we only focus on the novel techniques which are necessary
to deal with inhomogeneous Dirichlet data.}

\revision{Unlike the case $g=0$ which is well-studied in the literature, see e.g.~\cite{ao,verfuerth}, only little work has been done on a posteriori error estimation for~\eqref{eq:strongform} with $g\neq0$, cf.\ \cite{bcd,sv}.
Moreover, besides~\cite{fpp} no convergence result for AFEM with inhomogeneous Dirichlet data is found in the literature, yet.}

While the inclusion of inhomogeneous Neumann conditions $\phi$ into the
convergence analysis of \revision{e.g.~\cite{bdd,ckns,doerfler,ks,mns,stevenson} is straight forward},
incorporating inhomogeneous Dirichlet conditions $g$ is \revision{not obvious and technically much} more demanding for several reasons:
\revision{%
First, since discrete FE functions cannot satisfy general inhomogeneous Dirichlet conditions $g$, the FE scheme requires an additional discretization of $g\approx g_\ell$. Second, the error $\norm{g-g_\ell}{H^{1/2}(\Gamma_D)}$ of this data approximation has to be controlled with respect to the non-local $H^{1/2}$-norm and has to be included in the a~posteriori error analysis and the adaptive algorithm. Third, in contrast to the case $g=0$, the discrete ansatz spaces $\V_\ell$ are non-nested, i.e.\
$\V_\ell\not\subseteq \V_{\ell+1}$. We therefore loose the orthogonality in energy norm which leads to certain technicalities to construct a contraction quantity which is equivalent to the Galerkin error resp.\ error estimator.
Therefore, quasi-optimality as well as even plain convergence of AFEM
with inhomogeneous Dirichlet data is not obvious at all.}%

In an earlier work~\cite{fpp}, we considered lowest-order finite elements $p=1$ in 2D and nodal interpolation to discretize $g$. However, this situation is very special in the sense that our entire analysis in~\cite{fpp} is strictly bound to the lowest-order case and cannot be generalized to $\R^d$, since nodal interpolation \revision{of the Dirichlet data} is well-defined if and only if $d=2$.

In this work, we consider finite elements of piecewise polynomial order $p\ge1$ and dimension $d\ge 2$. We show that each uniformly $H^{1/2}(\Gamma_D)$-stable projection $\P_\ell$ onto the discrete trace space will do the job:
In this frame, we may use techniques from adaptive boundary element methods~\cite{cms,fkmp,kop} to localize the non-local $H^{1/2}$-norm in terms of a locally weighted $H^1$-seminorm. To overcome the lack of Galerkin orthogonality, the remedy is to concentrate on a quasi-Pythagoras theorem and a stronger marking criterion. The latter implies (quasi-local) equivalence of error estimators for different discretizations of the Dirichlet data.
To obtain contraction of our AFEM, we may then consider (theoretically) the $H^{1/2}(\Gamma_D)$-orthogonal projection. To obtain optimality of the marking strategy, we may consider the Scott-Zhang projection instead. Both auxiliary problems are somehow sufficiently close to the original problem with projection $\P_\ell$, which is enforced by the marking strategy.

Overall, we prove that each uniformly $H^{1/2}$-stable projection $\P_\ell$ will lead to a convergent AFEM algorithm. Under the usual restrictions on the adaptivity parameters, we even show optimal algebraic convergence behaviour with respect to the number of elements. 
\section{Adaptive Algorithm}
\label{section:algorithm}%

\noindent
It is well-known that the Poisson problem~\eqref{eq:strongform} admits a unique weak solution $u\in H^1(\Omega)$ with $u = g$ on $\Gamma_D$ in the sense of traces which solves the variational formulation
\begin{align}\label{eq:weakform}
 \dual{\nabla u}{\nabla v}_\Omega
 &= \dual{f}{v}_\Omega + \dual{\phi}{v}_{\Gamma_N}
 \quad \text{for all } v \in H^1_D(\Omega).
\end{align}
Here, the test space reads $H^1_D(\Omega) = \set{v\in H^1(\Omega)}{v = 0 \text{ on } \Gamma_D \text{ in the sense of traces}}$, and $\dual\cdot\cdot$ denotes the respective $L^2$-scalar products. The proof relies essentially on a reformulation of~\eqref{eq:strongform} as a problem with homogeneous Dirichlet data via a so-called lifting operator $\LL$, i.e.\ $\LL:H^{1/2}(\Gamma)\to H^1(\Omega)$ is a linear and continuous operator with $(\LL \widehat g)|_\Gamma = \widehat g$ for all $\widehat g\in H^{1/2}(\Gamma)$ in the sense of traces.
\revision{Again, we refer to Section~\ref{section:spaces} for the definition
of the trace space $H^{1/2}(\Gamma)$.}
However, although $\LL$ is constructed analytically, it is hardly accessible numerically in general and thus this approach is not feasible in practice.

This section provides an overview on this work and its main results.
We analyze a common adaptive mesh-refining algorithm of the type
\begin{align*}
\boxed{\texttt{ solve }}
\quad\longrightarrow\quad
\boxed{\texttt{ estimate }}
\quad\longrightarrow\quad
\boxed{\texttt{ mark }}
\quad\longrightarrow\quad
\boxed{\texttt{ refine }}
\end{align*}
which is stated in detail below in Section~\ref{section:algorithm2}. We start with a discussion of its four modules.

\subsection{The module \pmb{\texttt{solve}}}
Let $\TT_\ell$ be a regular triangulation of $\Omega$ into simplices, i.e.\ tetrahedra for 3D resp.\ triangles for 2D, which is generated from an initial triangulation $\TT_0$. Let $\EE_\ell$ be the set of facets, i.e.\ faces for 3D and edges for 2D, respectively. This set is split into interior facets $\EE_\ell^\Omega = \set{E\in\EE_\ell}{E\cap\Omega\neq\emptyset}$, i.e.\ each $E\in\EE_\ell^\Omega$ satisfies $E=T_+\cap T_-$ for $T_\pm\in\TT_\ell$, as well as boundary facets $\EE_\ell^\Gamma = \EE_\ell\backslash\EE_\ell^\Omega$. We assume that the partition of $\Gamma$ into Dirichlet boundary $\Gamma_D$ and Neumann boundary $\Gamma_N$ is 
\revision{already resolved by the initial mesh $\TT_0$},
i.e.\ $\EE_\ell^\Gamma$ is split into $\EE_\ell^D=\set{E\in\EE_\ell}{E\subseteq\overline\Gamma_D}$ and
$\EE_\ell^N=\set{E\in\EE_\ell}{E\subseteq\overline\Gamma_N}$
\revision{for all $\ell\ge0$.}
Note that
$\EE_\ell^D$ (resp.\ $\EE_\ell^N$) \revision{therefore} 
provides a regular triangulation of the
\revision{boundary} $\Gamma_D$ (resp.\ $\Gamma_N$).

We use conforming elements of fixed polynomial order $p\in\N$, where the ansatz space reads
\begin{align}\label{eq:fespace:omega}
 \SS^p(\TT_\ell)
 = \set{V_\ell\in C(\overline\Omega)}{V_\ell|_T \text{ is a polynomial of degree at most $p$ on }T\in\TT_\ell}.
\end{align}
Since a discrete function $U_\ell\in\SS^p(\TT_\ell)$ cannot satisfy general
continuous Dirichlet conditions, we have to discretize the given data
$g \in H^1(\Gamma_D)$. To this purpose, let $\P_\ell:H^{1/2}(\Gamma_D) \to \SS^p(\EE_\ell^D)$ be a projection onto the discrete trace space
\begin{align}\label{eq:fespace:gamma}
 \SS^p(\EE_\ell^D)
 = \set{V_\ell|_{\Gamma_D}}{V_\ell\in\SS^p(\TT_\ell)}
\end{align}
As in the continuous case, it is well-known that there is
a unique $U_\ell \in \SS^p(\TT_\ell)$ with $U_\ell = \P_\ell g$ on $\Gamma_D$
which solves the Galerkin formulation
\begin{align}\label{eq:galerkin}
 \dual{\nabla U_\ell}{\nabla V_\ell}_\Omega
 &= \dual{f}{V_\ell}_\Omega + \dual{\phi}{V_\ell}_{\Gamma_N}
 \quad\text{for all }V_\ell \in \SS^p_D(\TT_\ell).
\end{align}
Here, the test space is given by $\SS^p_D(\TT_\ell) = \SS^p(\TT_\ell) \cap H^1_D(\Omega) = \set{V_\ell \in \SS^p(\TT_\ell)}{V_\ell = 0 \text{ on } \Gamma_D}$. We assume that \texttt{solve} computes the exact Galerkin solution of~\eqref{eq:galerkin}. Arguing as e.g.\ in~\cite{bm,stevenson}, it is, however, possible to include an approximate solver into our analysis.

Possible choices for $\P_\ell$ include the $L^2$-orthogonal projection for the lowest-order case $p=1$, which is considered in~\cite{bcd}, or the Scott-Zhang projection from~\cite{sz} which is proposed in~\cite{sv}. 

\subsection{The module \pmb{\texttt{estimate}}}
We start with the element data oscillations
\begin{align}
 \oscT{\ell}^2:=\sum_{T\in\TT_\ell}\oscT{\ell}(T)^2,
 \text{ where }
 \oscT{\ell}(T)^2:=|T|^{2/d}\,\norm{(1-\Pi_T)(f+\Delta U_\ell)}{L^2(T)}^2
\end{align}
and where $\Pi_T:L^2(T)\to\PP^{p-1}(T)$ denotes the $L^2$-orthogonal
projection. These arise in the efficiency estimate for residual error estimators. Moreover, the efficiency involves the Neumann data oscillations
\begin{align}\label{eq:osc:neu}
 \oscN{\ell}^2:=\sum_{E\in\EE_\ell^N}\oscN\ell(E)^2,
 \text{ where }
 \oscN{\ell}(E)^2:=|T|^{1/d}\,\norm{(1-\Pi_E)\phi}{L^2(E)}^2
\end{align}
with $T\in\TT_\ell$ being the unique element with $E\subseteq\partial T$
and where $\Pi_E:L^2(E)\to\PP^{p-1}(E)$ denotes the $L^2$-orthogonal projection on the boundary. Finally, the approximation of the Dirichlet data
$\P_\ell g\approx g\in H^1(\Gamma_D)$ is controlled by the Dirichlet data oscillations
\begin{align}\label{eq:osc:dir}
 \oscD{\ell}^2:=\sum_{E\in\EE_\ell^D} \oscD{\ell}(E)^2,
 \text{ where }
 \oscD{\ell}(E)^2:=|T|^{1/d}\,\norm{(1-\Pi_E)\nabla_\Gamma g}{L^2(E)}^2,
\end{align}
where again $T\in\TT_\ell$ denotes the unique element with $E\subseteq\partial T$.
Moreover, $\nabla_\Gamma(\cdot)$ denotes the surface gradient.
We recall that up to shape regularity we have equivalence $|T|^{1/d}\simeq \diam(T)$ as well as $\diam(T)\simeq\diam(E)$ for all $T\in\TT_\ell$ and $E\in\EE_\ell$ with $E\subseteq\partial T$.

We use a residual error estimator $\eta_\ell^2 = \varrho_\ell^2 + \oscD\ell^2$ which is split into general contributions and Dirichlet oscillations, i.e.\
\begin{align}\label{eq1:estimator2:T}
 \varrho_\ell^2
 = \sum_{T\in\TT_\ell}\varrho_\ell(T)^2
\end{align}
with corresponding refinement indicators
\begin{align}\label{eq2:estimator2:T}
 \varrho_\ell(T)^2
 &:= |T|^{2/d}\,\norm{f+\Delta U_\ell}{L^2(T)}^2
 + |T|^{1/d}\big(\norm{[\partial_nU_\ell]}{L^2(\partial T\cap\Omega)}^2
 + \norm{\phi-\partial_nU_\ell}{L^2(\partial T\cap\Gamma_N)}^2\big).
\end{align}
The module \texttt{estimate} returns the elementwise contributions $\varrho_\ell(T)^2$ and $\oscD\ell(E)^2$ for all $T\in\TT_\ell$ and $E\in\EE_\ell^D$.

\subsection{The module \pmb{\texttt{mark}}}
For element marking, we use a modification of the D\"orfler marking~\cite{doerfler} proposed firstly in Stevenson~\cite{stevenson}. In each step of the adaptive loop, we mark either elements or Dirichlet edges for refinement, where the latter is only done if $\oscD\ell$ is large when compared to $\varrho_\ell$. A precise statement of the module \texttt{mark} is part of Algorithm~\ref{algorithm2} below.

\subsection{The module \pmb{\texttt{refine}}}
Locally refined meshes are obtained by use of the newest vertex bisection algorithm, see e.g.~\cite{stevenson:nvb,traxler}, where $\TT_{\ell+1} = \refine(\TT_\ell,\MM_\ell)$ for a set $\MM_\ell\subseteq\TT_\ell$ of marked elements returns the coarsest regular triangulation $\TT_{\ell+1}$ such that all marked elements $T\in\MM_\ell$ have been refined by at least one bisection. Arguing as in~\cite{ks}, one may also use variants of newest vertex bisection, where each $T\in\MM_\ell$ is refined by at least $n$ bisections with arbitrary, but fixed $n\in\N$.

\subsection{Adaptive loop}
\label{section:algorithm2}%
With the aforegoing modules, the adaptive mesh-refining algorithm takes the following form.

\begin{algorithm}\label{algorithm2}%
Let adaptivity parameters $0<\theta_1,\theta_2,\vartheta<1$ and initial triangulation $\TT_0$ be given. For each $\ell=0,1,2,\dots$ do:
\begin{itemize}
\item[(i)] Compute discrete solution $U_\ell\in\SS^p(\TT_\ell)$.
\item[(ii)] Compute refinement indicators $\varrho_\ell(T)$ and $\oscD\ell(E)$ for all $T\in\TT_\ell$ and $E\in\EE_\ell^D$.
\item[(iii)] Provided that $\oscD\ell^2 \le \vartheta\,\varrho_\ell^2$, choose $\MM_\ell \subseteq \TT_\ell$ such that
\begin{align}\label{eq:doerfler:res}
 \theta_1\,\varrho_\ell^2 \le \sum_{T\in\MM_\ell}\varrho_\ell(T)^2.
\end{align}
\item[(iv)] Provided that $\oscD\ell^2 > \vartheta\,\varrho_\ell^2$, choose $\MM_\ell^D \subseteq \EE_\ell^D$ such that
\begin{align}\label{eq:doerfler:osc}
 \theta_2\,\oscD\ell^2 \le \sum_{E\in\MM_\ell^D}\oscD\ell(E)^2
\end{align}
and let $\MM_\ell := \set{T\in\TT_\ell}{\exists E\in\MM_\ell^D\quad E\subseteq\partial T}$.
\item[(v)] Use newest vertex bisection to generate $\TT_{\ell+1}=\refine(\TT_\ell,\MM_\ell)$.
\item[(vi)] Update counter $\ell\mapsto\ell+1$ and go to {\rm(i)}.
\end{itemize}
\end{algorithm}

\subsection{Function spaces}\label{section:spaces}
This section briefly collects the function spaces and norms used in the following. 
We refer e.g.\ to the monographs~\cite{hsiao-wendland,mclean,sauterschwab} for further details.

$L^2(\Omega)$ resp.\ $H^1(\Omega)$ denote the usual Lebesgue space and Sobolev
space on $\Omega$. The dual space of $H^1(\Omega)$ with respect to the extended
$L^2(\Omega)$-scalar product is denoted by $\widetilde H^{-1}(\Omega)$.

For measurable $\gamma\subseteq\Gamma$,  e.g.\ $\gamma\in\{\Gamma_D,\Gamma_N\}$,
the Sobolev space $H^1(\gamma)$ is defined as the completion of the Lipschitz
continuous functions on $\gamma$ with respect to the norm
$\norm{v}{H^1(\gamma)}^2 = \norm{v}{L^2(\gamma)}^2 
+ \norm{\nabla_\Gamma v}{L^2(\gamma)}^2$, where $\nabla_\Gamma(\cdot)$ denotes
the surface gradient for $d=3$ resp.\ the arclength derivative for $d=2$. With the
Lebesgue space $L^2(\gamma)$, Sobolev spaces of fractional order $0\le\alpha\le1$
are defined by interpolation $H^\alpha(\gamma) = [L^2(\gamma);H^1(\gamma)]_\alpha$.
Moreover, $\widetilde H^1(\gamma$) is defined as the completion of all Lipschitz
continuous functions on $\gamma$ which vanish on $\partial\gamma$, with
respect to the $H^1(\gamma)$-norm, and $\widetilde H^\alpha(\gamma) 
= [L^2(\gamma);\widetilde H^1(\gamma)]_\alpha$ is defined by interpolation.

Sobolev spaces of negative order are defined by duality 
$\widetilde H^{-\alpha}(\gamma) = H^\alpha(\gamma)^*$ and 
$H^{-\alpha}(\gamma) = \widetilde H^\alpha(\gamma)^*$, where duality is understood
with respect to the extended $L^2(\gamma)$-scalar product.

We note that $H^{1/2}(\Gamma)$ can equivalently be defined as the trace space of
$H^1(\Omega)$, i.e.
\begin{align}\label{eq:dp:norm1}
 H^{1/2}(\Gamma) 
 = \set{\widehat w|_\Gamma}{\widehat w\in H^1(\Omega)}.
\end{align}
For our analysis, we shall use the graph norm of the restriction operator
\begin{align}
 \norm{w}{H^{1/2}(\Gamma)}
 := \inf\set{\norm{\widehat w}{H^1(\Omega)}}{\widehat w\in H^1(\Omega)
\text{ with }\widehat w|_\Gamma = w}.
\end{align}
Moreover, the graph norm and the interpolation norm are, in fact, 
equivalent norms on $H^{1/2}(\Gamma)$, and the norm equivalence constants depend
only on $\Gamma$.
A similar observation holds for the space $H^{1/2}(\gamma)$, namely
\begin{align}
 H^{1/2}(\gamma) = \set{\widehat w|_\gamma}{\widehat w\in H^{1/2}(\Gamma)},
\end{align}
and the corresponding graph norm
\begin{align}\label{eq:dp:norm2}
 \norm{w}{H^{1/2}(\gamma)}
 := \inf\set{\norm{\widehat w}{H^{1/2}(\Gamma)}}{\widehat w\in H^{1/2}(\Gamma)
 \text{ with }\widehat w|_\gamma = w}
\end{align}
is an equivalent norm on $H^{1/2}(\gamma)$. Throughout our analysis and without
loss of generalization, we shall 
equip $H^{1/2}(\Gamma)$ resp.\ $H^{1/2}(\gamma)$ with these graph 
norms~\eqref{eq:dp:norm1}--\eqref{eq:dp:norm2}.

\subsection{Main results}
Throughout, we assume that the projections $\P_\ell:H^{1/2}(\Gamma_D)\to\SS^p(\EE_\ell^D)$ are uniformly $H^{1/2}(\Gamma_D)$-stable, i.e.\ the operator norm is uniformly bounded
\begin{align}\label{eq:stability}
 \norm{\P_\ell:H^{1/2}(\Gamma_D)\to H^{1/2}(\Gamma_D)}{}
 \le \c{stability} < \infty
\end{align}
with some $\ell$-independent constant $\setc{stability}>0$. This assumption is guaranteed for the $H^{1/2}(\Gamma_D)$-orthogonal projection with $\c{stability}=1$. Moreover, the $L^2(\Gamma_D)$-orthogonal projection for the lowest-order case $p=1$ and newest vertex bisection is uniformly bounded~\cite{kp}, and so is the Scott-Zhang projection~\cite{sz} onto $\SS^p(\EE_\ell^D)$ for arbitrary $p\ge1$.

First, our discretization is quasi-optimal in the sense of the C\'ea lemma.
Note that estimate~\eqref{eq:cea} does not depend on the precise choice of $\P_\ell$, and the minimum is taken over all discrete functions.
Unlike our
observation, the result in e.g.~\cite[Theorem~6.1]{bcd} takes the minimum with
respect to the affine space $\set{W_\ell\in\SS^p(\TT_\ell)}{W_\ell|_{\Gamma_D}=\P_\ell g}$ and for first-order $p=1$ only.

\begin{proposition}[C\'ea-type estimate in $H^1$-norm]
\label{prop:cea}%
The Galerkin solution satisfies
\begin{align}\label{eq:cea}
 \norm{u-U_\ell}{H^1(\Omega)}
 \le\c{cea}\min_{W_\ell\in\SS^p(\TT_\ell)}\norm{u-W_\ell}{H^1(\Omega)}.
\end{align}
The constant $\setc{cea}>0$ depends only on $\Omega$, $\Gamma_D$,
shape regularity of $\TT_\ell$, the polynomial degree $p\ge1$, and the constant $\c{stability}>0$.
\end{proposition}

Second, the considered error estimator provides an upper bound and, up to data oscillations, also a lower bound for the Galerkin error.

\begin{proposition}[reliability and efficiency of $\eta_\ell$]
\label{prop:estimator}%
The error estimator $\eta_\ell^2 = \varrho_\ell^2+\oscD\ell^2$ is reliable
\begin{align}\label{eq:rho:reliable}
 \norm{u - U_\ell}{H^1(\Omega)}^2
 \le \c{rho:reliable}\, \eta_\ell^2
\end{align}
and efficient
\begin{align}\label{eq:rho:efficient}
 \c{rho:efficient}^{-1}\,\eta_\ell^2
 \le \norm{\nabla(u-U_\ell)}{L^2(\Omega)}^2 + \oscT{\ell}^2 +\oscN{\ell}^2 + \oscD{\ell}^2.
\end{align}
The constants $\setc{rho:reliable},\setc{rho:efficient}>0$ depend on $\Omega$ and $\Gamma_D$, on the polynomial degree $p\ge1$, stability $\c{stability}>0$, the initial triangulation $\TT_0$, and on the use of newest vertex bisection.
\end{proposition}

Note that convergence of Algorithm~\ref{algorithm2} in the sense of $\lim_\ell U_\ell = u$ in $H^1(\Omega)$ is a~priori unclear since adaptive mesh-refinement does not guarantee that the local mesh-size tends uniformly to zero. However, we have the following convergence result which is proved in the frame of the estimator reduction concept from~\cite{afp}.

\begin{theorem}[convergence of AFEM]
\label{theorem:convergence}%
{\rm(i)}
Suppose that the discretization of the Dirichlet data guarantees
some a~priori convergence
\begin{align}\label{eq:apriori:dirichlet}
 \lim_{\ell\to\infty} \norm{g_\infty-\P_\ell g}{H^{1/2}(\Gamma_D)} = 0
\end{align}
with a certain limit $g_\infty\in H^{1/2}(\Gamma_D)$.
Then, for any choice of the adaptivity parameters $0<\theta_1,\theta_2,\vartheta<1$, Algorithm~\ref{algorithm2} guarantees convergence
\begin{align}\label{eq:convergence}
 \lim_{\ell\to\infty} \norm{u-U_\ell}{H^1(\Omega)} = 0
\end{align}
and, in particular, $g_\infty = g$.\\
{\rm(ii)} Assumption~\eqref{eq:apriori:dirichlet} is satisfied for the $H^{1/2}(\Gamma_D)$-orthogonal projection, the $L^2(\Gamma_D)$-projection for $p=1$, and the Scott-Zhang projection for arbitrary $p\ge1$.
\end{theorem}

Current quasi-optimality results on AFEM rely on the fact that the estimator
$\eta_\ell^2 = \varrho_\ell^2 + \oscD\ell^2$ is equivalent to some linear convergent quasi-error quantity $\Delta_\ell$. Whereas, the convergence theorem (Theorem~\ref{theorem:convergence}) also holds for the usual D\"orfler marking, our contraction theorem relies on Stevenson's modification~\eqref{eq:doerfler:res}--\eqref{eq:doerfler:osc}. Moreover, we stress that the convergence theorem is constrained by the a~priori convergence assumption~\eqref{eq:apriori:dirichlet}, whereas the following contraction result is not.

\begin{theorem}[contraction of AFEM]
\label{theorem:contraction}%
\revision{We use Algorithm~\ref{algorithm2} with (up to the general assumptions stated above) arbitrary projection $\P_\ell$ and corresponding discrete solution $U_\ell\in\SS^1(\TT_\ell)$. In addition, let}
$P_\ell:H^{1/2}(\Gamma_D)\to\SS^p(\EE_\ell^D)$ be the $H^{1/2}(\Gamma_D)$-orthogonal projection. Let $\widetilde U_\ell\in\SS^p(\TT_\ell)$ the Galerkin solution of~\eqref{eq:galerkin} with $\widetilde U_\ell|_{\Gamma_D} = P_\ell g$ and $\widetilde\eta_\ell^{\,2} = \widetilde\varrho_\ell^{\,2} + \oscD\ell^2$ be the associated error estimator from~\eqref{eq1:estimator2:T} with $U_\ell$ replaced by $\widetilde U_\ell$. Then, for arbitrary $0<\theta_1,\theta_2<1$ and sufficiently small $0<\vartheta<1$, Algorithm~\ref{algorithm2} guarantees the existence of constants $\lambda,\mu>0$ and $0<\kappa<1$ such that the combined error quantity
\begin{align}\label{eq:delta}
 \Delta_\ell
 := \norm{\nabla(u - \widetilde U_\ell)}{L^2(\Omega)}^2
 + \lambda\,\norm{g-P_\ell g}{H^{1/2}(\Gamma_D)}^2
 + \mu\,\widetilde\eta_\ell^{\,2}
 \ge 0
\end{align}
satisfies a contraction property
\begin{align}\label{eq:contraction}
 \Delta_{\ell+1} \le \kappa\,\Delta_\ell
 \quad\text{for all }\ell\in\N_0.
\end{align}
Moreover, there are constants $\setc{equivalence1},\setc{equivalence2}>0$ such that
\begin{align}\label{eq:equivalence}
 \c{equivalence1}\,\Delta_\ell
 \le \eta_\ell^2
 \le \c{equivalence2}\,\Delta_\ell.
\end{align}
In particular, this implies convergence
$\lim_\ell\norm{u-U_\ell}{H^1(\Omega)}=0=\lim_\ell\eta_\ell$
of Algorithm~\ref{algorithm2} independently of the precise choice of the uniformly $H^{1/2}(\Gamma_D)$-stable projection $\P_\ell$.
\end{theorem}

\revision{%
\begin{remark}
The $H^{1/2}(\Gamma_D)$-orthogonal projection is \emph{not} needed for the 
implementation and can, in fact, hardly be computed explicitly. Instead, it is only
used for the numerical analysis. More precisely, we will see below that the 
modified D\"orfler marking~\eqref{eq:doerfler:res}--\eqref{eq:doerfler:osc}
for the $\P_\ell$ chosen (with corresponding discrete solution $U_\ell$ and
error estimator $\varrho_\ell$) implies
the usual D\"orfler marking for the \emph{theoretical auxiliary} problem
with the $H^{1/2}(\Gamma_D)$-orthogonal projection $P_\ell$ and corresponding 
solution $\widetilde U_\ell$ resp.\ error estimator $\widetilde\varrho_\ell$.
\end{remark}%
}

To state our quasi-optimality result for Algorithm~\ref{algorithm2}, we need to introduce further notation. Recall that, for a given triangulation $\TT_\ell$ and $\MM_\ell\subseteq\TT_\ell$,
\begin{align}
 \TT_{\ell+1}=\refine(\TT_\ell,\MM_\ell)
\end{align}
denotes the coarsest regular triangulation such that all marked elements
$T\in\MM_\ell$ have been refined by (at least one) bisection. Moreover, we write
\begin{align}
 \TT_\star = \refine(\TT_\ell)
\end{align}
if $\TT_\star$ is a finite refinement of $\TT_\ell$, i.e., there are finitely many
triangulations $\TT_{\ell+1},\dots,\TT_n$ and sets of marked elements
$\MM_\ell\subseteq\TT_\ell,\dots,\MM_{n-1}\subseteq\TT_{n-1}$ such that
$\TT_\star=\TT_n$ and $\TT_{j+1}=\refine(\TT_j,\MM_j)$ for all $j=\ell,\dots,n-1$. Finally, for a fixed initial mesh $\TT_0$, let $\T = \set{\TT_\star}{\TT_\star=\refine(\TT_0)}$ be
the set of all meshes which can be obtained by newest vertex bisection as well as the set $\T_N = \set{\TT_\star\in\T}{\#\TT_\star-\#\TT_0\le N}$ of all triangulations which have at most $N$ more elements than the initial mesh $\TT_0$.

Recall that Algorithm~\ref{algorithm2} only sees the error estimator $\eta_\ell^2 = \varrho_\ell^2 + \oscD\ell^2$, but not the error $\norm{u-U_\ell}{H^1(\Omega)}$. From this point of view, it is natural to ask for the best possible convergence rate for the error estimator. This can be characterized by means of an artificial approximation class $\A_s$: For $s\ge0$, we write
\begin{align}\label{eq:approximationclas}
 (u,f,g,\phi) \in \A_s
 \quad\stackrel{\text{def}}{\Longleftrightarrow}\quad
 \sup_{N\in\N} \inf_{\TT_\star\in\T_N} N^s\eta_\star < \infty,
\end{align}
where $\eta_\star^2 = \varrho_\star^2 +\oscD\star^2$ denotes the error estimator for the optimal mesh $\TT_\star\in\T_N$. By definition, this implies that a convergence rate $\eta_\star = \OO(N^{-s})$ is possible if the optimal meshes are chosen. The following theorem states that Algorithm~\ref{algorithm2}, in fact, guarantees $\eta_\ell = \OO(N^{-s})$ for the adaptively generated meshes $\TT_\ell$.

\begin{theorem}[quasi-optimality of AFEM]
\label{theorem:quasioptimal}%
Suppose that the sets $\MM_\ell$ resp.\ $\MM_\ell^D$ in step~(iii)--(iv) of Algorithm~\ref{algorithm2} are chosen with minimal cardinality.
Then, for sufficiently small $0<\theta_1,\vartheta<1$,
but arbitrary $0<\theta_2<1$,
Algorithm~\ref{algorithm2} guarantees the existence of a constant $\setc{optimal}>0$ such that
\begin{align}
 (u,f,g,\phi) \in \A_s
 \quad\Longleftrightarrow\quad
 \forall\ell\in\N\quad
 \eta_\ell
 \le \c{optimal}(\#\TT_\ell - \#\TT_0)^{-s},
\end{align}
i.e.\ each possible convergence rate $s>0$ is, in fact, asymptotically obtained
by AFEM.
\end{theorem}

We stress that, up to now and as far as the error estimator is concerned, only reliability~\eqref{eq:rho:reliable} is needed for the analysis. \revision{In particular,
the upper bounds on the \emph{sufficiently small} adaptivity parameters $\theta_1$
and $\vartheta$ do not depend on the efficiency constant $\c{rho:efficient}$
from~\eqref{eq:rho:efficient}. This is in contrast to the preceding works on AFEM,
e.g.~\cite{bm,ckns,fpp,ks,stevenson}, which directly ask for optimal convergence of the
error (Theorem~\ref{theorem:approximationclass}).}
Finally, the lower bound~\eqref{eq:rho:efficient} for the error estimator allows to characterize the approximation class $\A_s$ in terms of the regularity of the sought solution and the given data.
\revision{In particular, we obtain a quasi-optimality result which is analogous
to those available in the literature for homogeneous Dirichlet data, but with less
dependencies for the upper bound of the adaptivity parameters.}

\begin{theorem}[characterization of $\A_s$]
\label{theorem:approximationclass}%
It holds $(u,f,g,\phi) \in \A_s$ if and only if the following four conditions hold:
\begin{align}
 &\sup_{N \in \N}\inf_{\TT_\star\in\T_N}
 \min_{V_\star\in\SS^p(\TT_\star)} N^s\norm{u-V_\star}{H^1(\Omega)} < \infty,\label{eq:Achar1}\\
 &\sup_{N \in \N}\inf_{\TT_\star\in\T_N} N^s\oscT\star < \infty,\\
 &\sup_{N \in \N}\inf_{\TT_\star\in\T_N} N^s\oscN\star < \infty,\\
 &\sup_{N \in \N}\inf_{\TT_\star\in\T_N} N^s\oscD\star < \infty,\label{eq:Achar4}
\end{align}
i.e.\ the estimator ---and according to reliability hence the Galerkin error--- converges with the best possible rate allowed by the regularity of the sought solution and the given data.
\end{theorem}

\subsection{Outline}
Since our analysis is strongly built on properties of the Scott-Zhang projection, Section~\ref{section:scottzhang} collects the essential properties of the latter. This knowledge is used to prove Proposition~\ref{prop:cea}. Moreover, we prove that the Scott-Zhang error in a weighted $H^1(\Gamma_D)$-seminorm is 
\revision{(even locally)}
equivalent to the Dirichlet oscillations
\revision{(Proposition~\ref{lemma:oscD}) which might be of general interest.}
 This allows to prove Proposition~\ref{prop:estimator} with an estimator $\eta_\ell$ which does not explicitly contain the chosen projection $\P_\ell$. Section~\ref{section:convergence} is concerned with the proof of Theorem~\ref{theorem:convergence}. Section~\ref{section:contraction} gives the proof for the contraction result of Theorem~\ref{theorem:contraction}. Finally, the proof of the quasi-optimality results of Theorem~\ref{theorem:quasioptimal} and Theorem~\ref{theorem:approximationclass} are found in Section~\ref{section:optimality}. Some numerical experiments in Section~\ref{section:numerics} conclude the work.

In all statements, the constants involved and their dependencies are explicitly
stated. In proofs, however, we use the symbol $\lesssim$ to abbreviate $\le$ up to a multiplicative constant.
Moreover, $\simeq$ abbreviates that both estimates $\lesssim$ and $\gtrsim$ hold.

\section{Scott-Zhang Projection}
\label{section:scottzhang}%

\noindent%
The main tool of our analysis is the Scott-Zhang projection
\begin{align}
 \sz_\ell :H^1(\Omega) \to \SS^p(\TT_\ell)
\end{align}
from~\cite{sz}. A first application will be the proof of the C\'ea-type estimate for the Galerkin error (Proposition~\ref{prop:cea}). Moreover, we prove that the Scott-Zhang interpolation error
in a locally weighted $H^1$-seminorm is locally equivalent to the Dirichlet data oscillations (Proposition~\ref{lemma:oscD}). This will be the main tool to derive the bound $\norm{(1-\P_\ell)g}{H^{1/2}(\Gamma_D)} \lesssim \oscD{\ell}$.

\subsection{Scott-Zhang projection}
Analyzing the definition of $\sz_\ell$ in~\cite{sz}, one sees that $\sz_\ell$ can be defined locally in the following sense:
\begin{itemize}
\item For an element $T\in\TT_\ell$, the value $(\sz_\ell w)|_T$ on $T$ depends
only on the value of $w|_{\omega_{\ell,T}}$ on some element patch
\begin{align}
 T \subseteq \omega_{\ell,T} \subseteq \set{T'\in\TT_\ell}{T'\cap T\neq\emptyset}.
\end{align}
\item For a boundary facet $E\in\EE_\ell^\Gamma$, the trace of the Scott-Zhang projection $(\sz_\ell w)|_E$ on $E$ depends only on the trace $w|_{\omega_{\ell,E}^\Gamma}$ on some facet patch
\begin{align}
 E \subseteq \omega_{\ell,E}^\Gamma \subseteq \set{E'\in\EE_\ell}{E'\cap E\neq\emptyset}.
\end{align}
\item In case of a Dirichlet facet $E\in\EE_\ell^D$, one may choose
$\omega_{\ell,E}^\Gamma \subseteq \overline\Gamma_D$.
\end{itemize}
Moreover, $\sz_\ell$ is defined in a way that the following projection properties
hold:
\begin{itemize}
\item $\sz_\ell W_\ell = W_\ell$ for all $W_\ell\in\SS^p(\TT_\ell)$,
\item $(\sz_\ell w)|_\Gamma = w|_\Gamma$ for all $w\in H^1(\Omega)$ and
 $W_\ell\in\SS^p(\TT_\ell)$ with $w|_\Gamma = W_\ell|_\Gamma$,
\item $(\sz_\ell w)|_{\Gamma_D} = w|_{\Gamma_D}$ for all $w\in H^1(\Omega)$ and
 $W_\ell\in\SS^p(\TT_\ell)$ with $w|_{\Gamma_D} = W_\ell|_{\Gamma_D}$,
\end{itemize}
i.e.\ the projection $\sz_\ell$ preserves discrete (Dirichlet) boundary data.
Finally, $\sz_\ell$ satisfies the following (local) stability property
\begin{align}
\label{eq:scottzhangstability}
 \norm{\nabla(1-\sz_\ell) w}{L^2(T)}
 \leq \c{scottzhang}\,\norm{\nabla w}{L^2(\omega_{\ell,T})}\quad \text{for all } w \in H^1(\Omega)
\end{align}
and (local) first-order approximation property
\begin{align}
\label{eq:scottzhangapprox}
 \norm{(1-\sz_\ell) w}{L^2(T)} \leq \c{scottzhang}\,\norm{h_\ell\nabla w}{L^2(\omega_{\ell,T})}\quad \text{for all } w \in H^1(\Omega)
\end{align}
where $\setc{scottzhang}>0$ depends only on shape regularity of $\TT_\ell$, cf.~\cite{sz}.
Here, $h_\ell\in L^\infty(\Omega)$ denotes the local mesh-width function defined by $h_\ell|_T = |T|^{1/d}$ for all $T\in\TT_\ell$.
Moreover, since the overlap of the patches is controlled in terms of shape regularity, the integration domains in~\eqref{eq:scottzhangstability}--\eqref{eq:scottzhangapprox} can be replaced by $\Omega$, i.e.~\eqref{eq:scottzhangstability}--\eqref{eq:scottzhangapprox} hold also globally.

\subsection{Scott-Zhang projection onto discrete trace spaces}
We stress that $\sz_\ell$ induces operators
\begin{align}
 \sz_\ell^\Gamma:L^2(\Gamma)\to\SS^p(\EE_\ell^\Gamma)
 \quad\text{and}\quad
 \sz_\ell^D:L^2(\Gamma_D)\to\SS^p(\EE_\ell^D)
\end{align}
in the sense of $\sz_\ell^\Gamma(w|_\Gamma) = (\sz_\ell w)|_\Gamma$
and $\sz_\ell^D(w|_{\Gamma_D}) = (\sz_\ell^\Gamma(w|_\Gamma))|_{\Gamma_D}$
for all $w\in H^1(\Omega)$. We will thus not distinguish these operators
notationally.
Arguing as in~\cite{sz}, for
$\gamma\in\{\Gamma,\Gamma_D,\Gamma_N\}$, one sees that $\sz_\ell$ satisfies even (local) $L^2$-stability
\begin{align}\label{eq:scottzhang:L2(Gamma)}
 \norm{(1-\sz_\ell) w}{L^2(E)}
 \le \c{scottzhang}\,\norm{w}{L^2(\omega_{\ell,E}^\Gamma)}
 \quad\text{for all }w\in L^2(\gamma),
\end{align}
(local) $H^1$-stability
\begin{align}\label{eq:scottzhang:H1(Gamma)}
 \norm{(1-\sz_\ell) w}{H^1(E)}
 \le \c{scottzhang}\,\norm{\nabla_\Gamma w}{L^2(\omega_{\ell,E}^\Gamma)}
 \quad\text{for all }w\in H^1(\gamma),
\end{align}
as well as a (local) first-order approximation property
\begin{align}\label{eq:scottzhang:approx:H1(Gamma)}
 \norm{(1-\sz_\ell) w}{L^2(E)}
 \le \c{scottzhang}\,\norm{h_\ell\nabla_\Gamma w}{L^2(\omega_{\ell,E}^\Gamma)}
 \quad\text{for all }w\in H^1(\gamma).
\end{align}
Here, $\nabla_\Gamma(\cdot)$ denotes again the surface gradient,
and
$h_\ell\in L^\infty(\Gamma_D)$ denotes the local mesh-width function restricted
to $\Gamma_D$.
According to shape regularity of $\TT_\ell$, the integration domains in~\eqref{eq:scottzhang:L2(Gamma)}--\eqref{eq:scottzhang:approx:H1(Gamma)} can be replaced by $\gamma$, i.e.~\eqref{eq:scottzhang:L2(Gamma)}--\eqref{eq:scottzhang:approx:H1(Gamma)} hold also globally on $\gamma$.

By standard interpolation arguments applied
to~\eqref{eq:scottzhang:L2(Gamma)}--\eqref{eq:scottzhang:H1(Gamma)},
one obtains stability
\begin{align}\label{eq:scottzhang:H12(Gamma)}
 \norm{(1-\sz_\ell) w}{H^{1/2}(\gamma)}
 \le \c{scottzhang}\,\norm{w}{H^{1/2}(\gamma)}
 \quad\text{for all }w\in H^{1/2}(\gamma)
\end{align}
in the trace norm. Moreover, it is proved in~\cite[Theorem~3]{kop} that the Scott-Zhang projection satisfies
\begin{align}\label{eq:scottzhang:kp}
 \norm{(1-\sz_\ell) w}{H^{1/2}(\gamma)}
 \le \c{scottzhang}\,\min_{W_\ell\in\SS^p(\TT_\ell|_\gamma)}\norm{h_\ell^{1/2}\nabla_\Gamma (w-W_\ell)}{L^2(\gamma)}
 \quad\text{for all }w\in H^1(\gamma).
\end{align}
Throughout, the constant $\c{scottzhang}>0$ then depends only on shape regularity of $\TT_\ell$ and on $\gamma\in\{\Gamma,\Gamma_D,\Gamma_N\}$.

\subsection{Proof of C\'ea lemma (Proposition~\ref{prop:cea})}
According to weak formulation~\eqref{eq:weakform} and Galerkin
formulation~\eqref{eq:galerkin}, we have the Galerkin orthogonality relation
\begin{align*}
 \dual{\nabla(u-U_\ell)}{\nabla V_\ell}_\Omega = 0
 \quad\text{for all }V_\ell\in\SS^p_D(\TT_\ell).
\end{align*}
Let $\LL:H^{1/2}(\Gamma)\to H^1(\Omega)$ be a lifting operator.
Let $\widehat g,\widehat g_\ell\in H^{1/2}(\Gamma)$ denote arbitrary
extensions of $g=u|_{\Gamma_D}$ resp.\ $\P_\ell g = U_\ell|_{\Gamma_D}$.
Note that $(\sz_\ell\LL\sz_\ell\widehat g)|_{\Gamma_D}=(\sz_\ell u)|_{\Gamma_D}$
as well as $(\sz_\ell\LL\sz_\ell\widehat g_\ell)|_{\Gamma_D}=U_\ell|_{\Gamma_D}$.
For arbitrary $V_\ell\in\SS^p_D(\TT_\ell)$, we thus have
$U_\ell - (V_\ell + \sz_\ell\LL\sz_\ell\widehat g_\ell)\in\SS^p_D(\TT_\ell)$, whence
\begin{align*}
 \norm{\nabla(u-U_\ell)}{L^2(\Omega)}^2
 = \dual{\nabla(u-U_\ell)}{\nabla(u-(V_\ell + \sz_\ell\LL\sz_\ell\widehat g_\ell))}_\Omega
\end{align*}
according to the Galerkin orthogonality. Therefore, the Cauchy inequality
proves
\begin{align*}
 \norm{\nabla(u-U_\ell)}{L^2(\Omega)}
 \le\min_{V_\ell\in\SS^p_D(\TT_\ell)}\norm{\nabla(u-(V_\ell + \sz_\ell\LL\sz_\ell\widehat g_\ell))}{L^2(\Omega)}.
\end{align*}
We now plug-in $V_\ell = \sz_\ell u-\sz_\ell\LL\sz_\ell\widehat g\in\SS^p_D(\TT_\ell)$
and use stability of $\sz_\ell$ and $\LL$ to see
\begin{align*}
 \norm{\nabla(u-U_\ell)}{L^2(\Omega)}
 &\le \norm{\nabla(u-\sz_\ell u + \sz_\ell\LL\sz_\ell(\widehat g-\widehat g_\ell))}{L^2(\Omega)}
 \\&
 \lesssim \norm{\nabla(u-\sz_\ell u)}{L^2(\Omega)}
 + \norm{\widehat g-\widehat g_\ell}{H^{1/2}(\Gamma)}.
\end{align*}
Since the extensions $\widehat g,\widehat g_\ell$ of $g$ and $\P_\ell g$ were arbitrary, we obtain
\begin{align*}
 \norm{\nabla(u-U_\ell)}{L^2(\Omega)}
 &\lesssim \norm{\nabla(u-\sz_\ell u)}{L^2(\Omega)}
 + \norm{(1-\P_\ell)g}{H^{1/2}(\Gamma_D)}.
\end{align*}
According to the projection property $\sz_\ell W_\ell = W_\ell$ for
$W_\ell\in\SS^p(\TT_\ell)$ and $H^1$-stability~\eqref{eq:scottzhangstability},
it holds that
\begin{align*}
 \norm{\nabla(u-\sz_\ell u)}{L^2(\Omega)}
 = \min_{W_\ell\in\SS^p(\TT_\ell)}
 \norm{\nabla(1-\sz_\ell)(u-W_\ell)}{L^2(\Omega)}
 \lesssim \min_{W_\ell\in\SS^p(\TT_\ell)}\norm{\nabla(u-W_\ell)}{L^2(\Omega)}.
\end{align*}
The same argument for $\P_\ell$ with stability on $H^{1/2}(\Gamma)$
gives
\begin{align*}
 \norm{(1-\P_\ell)g}{H^{1/2}(\Gamma_D)}
 = \min_{W_\ell\in\SS^p(\TT_\ell)} \norm{(1-\P_\ell)(g-W_\ell|_{\Gamma_D})}{H^{1/2}(\Gamma_D)}
 \lesssim \min_{W_\ell\in\SS^p(\TT_\ell)}
 \norm{g-W_\ell|_{\Gamma_D}}{H^{1/2}(\Gamma_D)}.
\end{align*}
Combining the last three estimates, we infer
\begin{align*}
 \norm{\nabla(u-U_\ell)}{L^2(\Omega)}
 + \norm{(1-\P_\ell)g}{H^{1/2}(\Gamma_D)}
 \lesssim \min_{W_\ell\in\SS^p(\TT_\ell)}\big(
 \norm{\nabla(u-W_\ell)}{L^2(\Omega)}
 + \norm{g-W_\ell|_{\Gamma_D}}{H^{1/2}(\Gamma_D)}
 \big).
\end{align*}
Finally, the Rellich compactness theorem implies norm equivalence
$\norm{\cdot}{H^1(\Omega)}\simeq\norm{\nabla(\cdot)}{L^2(\Omega)}
 + \norm{(\cdot)|_{\Gamma_D}}{H^{1/2}(\Gamma_D)}$ on $H^1(\Omega)$.
This concludes the proof.\qed

\subsection{Scott-Zhang projection and Dirichlet data oscillations}
We stress that the newest vertex bisection algorithm guarantees that only
finitely many shapes of elements $T\in\set{T\in\TT_\star}{\TT_\star\in\T}$ can occur.
In particular, only finitely many shapes of patches occur. Further details are found in~\cite[Chapter~4]{verfuerth}
as well as in~\cite{stevenson:nvb,traxler}.
This observation will be used in the proof of the following lemma.

\begin{proposition}\label{lemma:oscD}%
Let $\Pi_\ell:L^2(\Gamma_D)\to\PP^{p-1}(\EE_\ell^D)$ denote the $L^2(\Gamma_D)$-projection.
Then,
\begin{align}\label{eq:dirichlet2}
 \norm{(1\!-\!\Pi_\ell)\nabla_\Gamma g}{L^2(E)}
 \le \norm{\nabla_\Gamma(1\!-\!\sz_\ell)g}{L^2(E)}
 \le \c{dirichlet0}\norm{(1\!-\!\Pi_\ell)\nabla_\Gamma g}{L^2(\omega_{\ell,E}^\Gamma)}
 \text{ for all }E\in\EE_\ell^D
\end{align}
and, in particular,
\begin{align}\label{eq:dirichlet}
 \oscD{\ell} \le \norm{h_\ell^{1/2}\nabla_\Gamma(1-\sz_\ell)g}{L^2(\Gamma_D)}
 \le \c{dirichlet0}\,\oscD{\ell}
\end{align}
The constant $\setc{dirichlet0}\ge1$ depends only on $\Gamma_D$, the polynomial degree $p$, the initial triangulation $\TT_0$, and the use of newest
vertex bisection to obtain $\TT_\ell\in\T$, but \emph{not} on $g$.
\end{proposition}

\begin{proof}%
\def\Q{\mathbb Q}%
Since $\Pi_\ell$ is the piecewise $L^2$-projection, the lower bound
in~\eqref{eq:dirichlet2}--\eqref{eq:dirichlet} is obvious. To verify the upper bound, we argue by
contradiction and assume that the upper bound in~\eqref{eq:dirichlet2} is wrong
for each constant $C>0$. For $n\in\N$, we thus find some
$\widetilde g_n\in H^1(\Gamma)$ such that
\begin{align}\label{eq:dirichlet2prime}
 \norm{\nabla_\Gamma(1-\sz_\ell)\widetilde g_n}{L^2(E)}
 > n\,\norm{(1-\Pi_\ell)\nabla_\Gamma \widetilde g_n}{L^2(\omega_{\ell,E}^\Gamma)}.
\end{align}
Let $\Q_{\ell,E}:H^1(\omega_{\ell,E}^\Gamma)\to\SS^p(\EE_\ell|_{\omega_{\ell,E}^\Gamma})$
denote the $H^1$-orthogonal projection on the patch $\omega_{\ell,E}^\Gamma$ and define
$\overline g_n = (1-\Q_{\ell,E})\widetilde g_n$. Since the value of $\sz_\ell v$ on $E$
depends only on the values of $v$ on $\omega_{\ell,E}^\Gamma$, the projection property
of $\sz_\ell$ reveals $(1-\sz_\ell)\Q_{\ell,E}\widetilde g_n=0$ on $E$. Moreover,
$\nabla_\Gamma\Q_{\ell,E}\widetilde g_n\in\PP^{p-1}(\EE_\ell|_{\omega_{\ell,E}^\Gamma})$
so that $(1-\Pi_\ell)\nabla_\Gamma\Q_{\ell,E}\widetilde g_n=0$ on $\omega_{\ell,E}^\Gamma$. From the orthogonal
decomposition $\widetilde g_n = \Q_{\ell,E}\widetilde g_n + \overline g_n$, we thus
see
$\norm{\nabla_\Gamma(1-\sz_\ell)\overline g_n}{L^2(E)}
 = \norm{\nabla_\Gamma(1-\sz_\ell)\widetilde g_n}{L^2(E)}$
and
$\norm{(1-\Pi_\ell)\nabla_\Gamma \overline g_n}{L^2(\omega_{\ell,E}^\Gamma)}
 = \norm{(1-\Pi_\ell)\nabla_\Gamma \widetilde g_n}{L^2(\omega_{\ell,E}^\Gamma)}$.
In particular, we observe $\overline g_n\neq0$ from~\eqref{eq:dirichlet2prime} so that we may define
$g_n:=\overline g_n/\norm{\overline g_n}{H^1(\omega_{\ell,E}^\Gamma)}$. This definition
guarantees
\begin{align}\label{eq1:dirichlet}
 \norm{g_n}{H^1(\omega_{\ell,E}^\Gamma)}=1
 \quad\text{and}\quad
 g_n\in\SS^p(\EE_\ell|_{\omega_{\ell,E}^\Gamma})^\perp,
\end{align}
where orthogonality is understood with respect to the $H^1(\omega_{\ell,E}^\Gamma)$-scalar
product. Moreover, it holds that
\begin{align}\label{eq2:dirichlet}
 \norm{(1-\Pi_\ell)\nabla_\Gamma g_n}{L^2(\omega_{\ell,E}^\Gamma)}
 <\frac{1}{n}\,\norm{\nabla_\Gamma(1-\sz_\ell) g_n}{L^2(E)}
 \lesssim \frac{1}{n}\,\norm{\nabla_\Gamma g_n}{L^2(\omega_{\ell,E}^\Gamma)}
 \xrightarrow{n\to\infty}0
\end{align}
due to the construction of $g_n$ and local $H^1$-stability of $\sz_\ell:H^1(\Gamma_D)\to H^1(\Gamma_D)$.

First,~\eqref{eq2:dirichlet} implies that
$\norm{\Pi_\ell\nabla_\Gamma g_n}{L^2(\omega_{\ell,E}^\Gamma)}\le C<\infty$ is uniformly
bounded as $n\to\infty$. Since $\Pi_\ell\nabla_\Gamma g_n\in\PP^{p-1}(\EE_\ell|_{\omega_{\ell,E}^\Gamma})$ belongs to a finite dimensional space, we may apply the
Bolzano-Weierstrass theorem to extract a convergent subsequence. Without loss
of generality, we may thus assume
\begin{align}\label{eq3:dirichlet}
 \Pi_\ell\nabla_\Gamma g_n
 \xrightarrow{n\to\infty}\Phi_\ell\in \PP^{p-1}(\EE_\ell|_{\omega_{\ell,E}^\Gamma})
 \quad\text{in strong $L^2$-sense}.
\end{align}

Second, this and~\eqref{eq2:dirichlet} prove $L^2$-convergence of $\nabla_\Gamma g_n$ to $\Phi_\ell$,
\begin{align}\label{eq4:dirichlet}
 \norm{\nabla_\Gamma g_n-\Phi_\ell}{L^2(\omega_{\ell,E}^\Gamma)}
 \le \norm{(1-\Pi_\ell)\nabla_\Gamma g_n}{L^2(\omega_{\ell,E}^\Gamma)}
 + \norm{\Pi_\ell\nabla_\Gamma g_n-\Phi_\ell}{L^2(\omega_{\ell,E}^\Gamma)}
 \xrightarrow{n\to\infty}0.
\end{align}

Third, orthogonality~\eqref{eq1:dirichlet} implies $\int_{\omega_{\ell,E}^\Gamma}g_n\,d\Gamma=0$ if we consider the constant function $1\in \SS^p(\EE_\ell|_{\omega_{\ell,E}^\Gamma})$. Therefore, the Friedrichs inequality and~\eqref{eq4:dirichlet} predict
uniform boundedness $\norm{g_n}{H^1(\omega_{\ell,E}^\Gamma)}\lesssim\norm{\nabla_\Gamma g_n}{L^2(\omega_{\ell,E}^\Gamma)}\le C<\infty$ as $n\to\infty$. According to weak compactness
in Hilbert spaces, we may thus extract a weakly convergent subsequence.
Without loss of generality, we may thus assume
\begin{align}\label{eq5:dirichlet}
 g_n \xrightarrow{n\to\infty} g_\infty\in H^1(\omega_{\ell,E}^\Gamma)
 \quad\text{in weak $H^1$-sense}.
\end{align}

Fourth, the combination of~\eqref{eq4:dirichlet} and~\eqref{eq5:dirichlet}
implies $\nabla_\Gamma g_\infty=\Phi_\ell$.
This follows from the fact that
$\norm{\Phi_\ell - \nabla_\Gamma(\cdot)}{L^2(\omega_{\ell,E}^\Gamma)}$
is convex and continuous, whence weakly lower semicontinuous on $H^1(\omega_{\ell,E}^\Gamma)$,
i.e.\ $\norm{\Phi_\ell - \nabla_\Gamma g_\infty}{L^2(\omega_{\ell,E}^\Gamma)}
\le\liminf_n\norm{\Phi_\ell - \nabla_\Gamma g_n}{L^2(\omega_{\ell,E}^\Gamma)}
= 0$.

Fifth, the Rellich compactness theorem proves that the convergence
in~\eqref{eq5:dirichlet} does also hold in strong $L^2$-sense. Together
with~\eqref{eq4:dirichlet} and $\Phi_\ell = \nabla_\Gamma g_\infty$, we now observe strong $H^1$-convergence
\begin{align*}
 \norm{g_\infty-g_n}{H^1(\omega_{\ell,E}^\Gamma)}^2
 = \norm{g_\infty-g_n}{L^2(\omega_{\ell,E}^\Gamma)}^2
 + \norm{\Phi_\ell-\nabla_\Gamma g_n}{L^2(\omega_{\ell,E}^\Gamma)}^2
 \xrightarrow{n\to\infty}0,
\end{align*}
whence $\norm{g_\infty}{H^1(\omega_{\ell,E}^\Gamma)}=1$ as well as $g_\infty\in\SS^p(\EE_\ell|_{\omega_{\ell,E}^\Gamma})^\perp$ according
to~\eqref{eq1:dirichlet}.

On the other hand,
$\nabla_\Gamma g_\infty = \Phi_\ell\in\PP^{p-1}(\EE_\ell|_{\omega_{\ell,E}^\Gamma})$ implies
$g_\infty\in\SS^p(\EE_\ell|_{\omega_{\ell,E}^\Gamma})$. This yields
$g_\infty\in\SS^p(\EE_\ell|_{\omega_{\ell,E}^\Gamma})\cap\SS^p(\EE_\ell|_{\omega_{\ell,E}^\Gamma})^\perp=\{0\}$ and contradicts $\norm{g_\infty}{H^1(\omega_{\ell,E}^\Gamma)}^2=1$.

This contradiction proves the upper bound in~\eqref{eq:dirichlet2}. A standard
scaling argument verifies that the constant $\c{dirichlet0}>0$ does only
depend on the shape of $\omega_{\ell,E}^\Gamma$ but not on the diameter. As stated
above, newest vertex bisection guarantees that only finitely many shapes of patches
$\omega_{\ell,E}^\Gamma$ may occur, i.e.\ $\c{dirichlet0}>0$ depends only on $\TT_0$ and
the use of newest vertex bisection. Summing~\eqref{eq:dirichlet2} over all
Dirichlet facets, we see
\begin{align*}
  \oscD{\ell}^2
 = \norm{h_\ell^{1/2}(1-\Pi_\ell)\nabla_\Gamma g}{L^2(\Gamma_D)}^2
 &\le \norm{h_\ell^{1/2}\nabla_\Gamma(1-\sz_\ell)g}{L^2(\Gamma_D)}^2
\\
 &\lesssim\sum_{E\in\EE_\ell^D}h_\ell|_E\,
 \norm{(1-\Pi_\ell)\nabla_\Gamma g}{L^2(\omega_{\ell,E}^\Gamma)}^2\\
 &\lesssim \norm{h_\ell^{1/2}(1-\Pi_\ell)\nabla_\Gamma g}{L^2(\Gamma_D)}^2,
\end{align*}
where the final estimate holds due to uniform shape regularity.
\end{proof}

\begin{corollary}\label{cor:oscD}
It holds $\norm{(1-\P_\ell)g}{H^{1/2}(\Gamma_D)} \le \c{dirichlet}\,\oscD\ell$, where $\c{dirichlet}>0$ depends on $\Gamma_D$, the polynomial degree $p\ge1$, stability $\setc{stability}>0$, the initial mesh $\TT_0$, and the use of newest vertex bisection.
\end{corollary}

\begin{proof}
By use of the projection property and stability of $\P_\ell$, one sees $\norm{(1-\P_\ell)g}{H^{1/2}(\Gamma_D)}
=\norm{(1-\P_\ell)(1-\sz_\ell)g}{H^{1/2}(\Gamma_D)}
\lesssim \norm{(1-\sz_\ell)g}{H^{1/2}(\Gamma_D)}$. The approximation estimate~\eqref{eq:scottzhang:kp} and Proposition~\ref{lemma:oscD} conclude
$\norm{(1-\sz_\ell)g}{H^{1/2}(\Gamma_D)}\lesssim\norm{h_\ell^{1/2}\nabla_\Gamma(1-\sz_\ell)g}{L^2(\Gamma_D)}
\simeq\oscD\ell$.
\end{proof}

\subsection{Proof of reliability and efficiency (Proposition~\ref{prop:estimator})}
We consider a continuous auxiliary problem
\begin{align}
\begin{split}
 -\Delta w &= 0 \hspace*{18.6mm} \text{in } \Omega,\\
 w &= (1-\P_\ell)g \quad \text{on } \Gamma_D,\\
 \partial_n w &= 0 \hspace*{18.6mm} \text{on } \Gamma_N,
\end{split}
\end{align}
with unique solution $w\in H^1(\Omega)$. We then have norm equivalence $\norm{w}{H^1(\Omega)}\simeq\norm{(1-\P_\ell)g}{H^{1/2}(\Gamma_D)}$ as well as $u-U_\ell-w\in H^1_D(\Omega)$. From this, we obtain
\begin{align*}
 \norm{u-U_\ell}{H^1(\Omega)}
 \lesssim \norm{\nabla(u-U_\ell-w)}{L^2(\Omega)}
 + \norm{(1-\P_\ell)g}{H^{1/2}(\Gamma_D)}.
\end{align*}
The first term on the right-hand side can be handled as for homogeneous
Dirichlet data, i.e.\ use of the Galerkin orthogonality combined with
approximation estimates for a Cl\'ement-type quasi-interpolation operator
(e.g.\ the Scott-Zhang projection). This leads to
\begin{align*}
 &\norm{\nabla(u-U_\ell-w)}{L^2(\Omega)}
 \lesssim \varrho_\ell
\end{align*}
Details are found e.g.\ in~\cite{bcd}. The $H^{1/2}(\Gamma_D)$-norm is
dominated by the Dirichlet data oscillations $\oscD\ell$, see Corollary~\ref{cor:oscD}.

By use of bubble functions and local scaling arguments, one obtains the estimates
\begin{align*}
 |T|^{2/d}\,\norm{f+\Delta U_\ell}{L^2(T)}^2
 &\lesssim \norm{\nabla(u-U_\ell)}{L^2(T)}^2 + \oscT{\ell}(T)^2+\oscN{\ell}(\partial T\cap \Gamma_N),\\
 |T|^{1/d}\,\norm{[\partial_nU_\ell]}{L^2(E\cap\Omega)}^2
 &\lesssim \norm{\nabla(u-U_\ell)}{L^2(\Omega_{\ell,E})}^2 + \oscT{\ell}(\omega_{\ell,E})^2,\\
 |T|^{1/d}\,\norm{\phi-\partial_nU_\ell}{L^2(E\cap\Gamma_N)}^2
 &\lesssim \norm{\nabla(u-U_\ell)}{L^2(\Omega_{\ell,E})}^2 + \oscT{\ell}(\omega_{\ell,E})^2+\oscN{\ell}(E\cap\Gamma_N)^2,
\end{align*}
where $\Omega_{\ell,E}=T^+\cup T^-$ denotes the facet patch of
$T_+\cap T_-=E\in\EE_\ell$.
Details are found e.g.\ in~\cite{ao,verfuerth}.
Summing these estimates over all elements, one obtains the efficiency estimate~\eqref{eq:rho:efficient}.
\qed

\section{Convergence}
\label{section:convergence}%

\noindent%
In this section, we aim to prove Theorem~\ref{theorem:convergence}.
Our proof of the convergence theorem relies on the estimator reduction principle from~\cite{afp}, i.e.\ we verify that the error estimator is contractive up to some zero sequence.

\subsection{Estimator reduction estimate}
Note that the estimator $\eta_\ell^2 = \varrho_\ell^2 + \oscD\ell^2$ can be localized over elements via
\begin{align}\label{eq:eta:local}
 \eta_\ell^2 = \sum_{T\in\TT_\ell}\eta_\ell(T)^2
 \quad\text{with}\quad
 \eta_\ell(T)^2 = \varrho_\ell(T)^2 + |T|^{1/d}\,\norm{(1-\Pi_\ell)\nabla_\Gamma g}{L^2(\partial T\cap\Gamma_D)}^2
\end{align}
with $\Pi_\ell:L^2(\Gamma_D)\to\PP^{p-1}(\EE_\ell^D)$
the (even $\EE_\ell^D$-piecewise) $L^2(\Gamma_D)$-orthogonal projection.

\begin{lemma}[modified marking implies D\"orfler marking]
\label{lemma:doerfler1}%
For $0<\theta_1,\theta_2,\vartheta<1$ in Algorithm~\ref{algorithm2}, there is some parameter $0<\theta<1$ such that the error estimator $\eta_\ell^2 = \varrho_\ell^2 + \oscD\ell^2$ satisfies
\begin{align}\label{eq:doerfler}
 \theta\,\eta_\ell^2
 \le \sum_{T\in\MM_\ell}\eta_\ell(T)^2,
\end{align}
and all elements $T\in\MM_\ell$ are refined by at least one bisection.
\end{lemma}

\begin{proof}
First, assume $\oscD\ell^2 \le \vartheta\,\varrho_\ell^2$ and let $\MM_\ell\subseteq\TT_\ell$ satisfy~\eqref{eq:doerfler:res}. Then,
\begin{align*}
 \theta_1(\varrho_\ell^2+ \oscD\ell^2)
 \le \theta_1(1+\vartheta)\varrho_\ell^2
 \le (1+\vartheta)\sum_{T\in\MM_\ell}\varrho_\ell(T)^2.
\end{align*}
Therefore, the D\"orfler marking~\eqref{eq:doerfler} holds with
$\theta\le \theta_1(1+\vartheta)^{-1}$.

Second, assume $\oscD\ell^2 > \vartheta\,\varrho_\ell^2$ and let $\MM_\ell^D\subseteq\EE_\ell^D$ satisfy~\eqref{eq:doerfler:osc}.
Then,
\begin{align*}
 \theta_2(\varrho_\ell^2+ \oscD\ell^2)
 \le \theta_2(1+\vartheta^{-1})\oscD\ell^2
 \le (1+\vartheta^{-1})\sum_{E\in\MM_\ell^D}\oscD\ell(E)^2.
\end{align*}
Therefore, the D\"orfler marking~\eqref{eq:doerfler} holds with
$\theta \le \theta_2(1+\vartheta^{-1})^{-1}$, and all elements which
have some facet $E\in\MM_\ell^D$ are refined.
\end{proof}

\begin{proposition}[estimator reduction]
\label{prop:reduction}%
Let $\TT_\star = \refine(\TT_\ell)$ be an arbitrary refinement of $\TT_\ell$
and $\MM_\ell\subseteq\TT_\ell\backslash\TT_\star$ a subset of the refined elements
which satisfies the D\"orfler marking~\eqref{eq:doerfler} for some $0<\theta<1$.
Then,
\begin{align}\label{eq:reduction}
 \eta_\star^2
 \le \q{reduction}\,\eta_\ell^2
 + \c{reduction}\,\norm{\nabla(U_{\ell+1}-U_\ell)}{L^2(\Omega)}^2
\end{align}%
with certain constants $0<\q{reduction}<1$ and $\c{reduction}>0$ which depend only on the parameter $0<\theta<1$, shape regularity of $\TT_\star$, and the polynomial degree $p\ge1$.
\end{proposition}

\begin{proof}[Sketch of proof]
For the sake of completeness, we include the idea of the proof of~\eqref{eq:reduction} although our proof is only a minor extension of the proof from~\cite[Cor.\ 3.4, Proof of Theorem 4.1]{ckns},
where all details are found.
First, we employ a triangle inequality and the Young inequality to
see for arbitrary $\delta>0$
\begin{align*}
 \eta_\star^2
 &\le(1+\delta)\,\Big(\sum_{T'\in\TT_\star}
 |T'|^{2/d}\,\norm{f+\Delta U_\ell}{L^2(T')}^2
 +|T'|^{1/d}\,\norm{[\partial_nU_\ell]}{L^2(\partial T'\cap\Omega)}^2
 \\&\qquad\qquad\qquad\qquad
+|T'|^{1/d}\,\norm{\phi-\partial_n U_\ell}{L^2(\partial T'\cap\Gamma_N)}^2
 +|T'|^{1/d}\,\norm{(1-\Pi_\star)\nabla_\Gamma g}{L^2(\partial T'\cap\Gamma_D)}^2
 \Big)
 \\&\quad
 + (1+\delta^{-1})\,\Big(\sum_{T'\in\TT_\star}
  |T'|^{2/d}\norm{\Delta(U_\star-U_\ell)}{L^2(T')}^2
  + |T'|^{1/d}\norm{[\partial_n(U_\star-U_\ell)]}{L^2(\partial T'\cap\Omega)}^2
 \\&\quad\qquad\qquad\qquad\qquad
  + |T'|^{1/d}\norm{\partial_n(U_\star-U_\ell)}{L^2(\partial T'\cap\Gamma_N)}^2
\Big).
\end{align*}
A scaling argument proves that the second bracket is bounded by
$C\,\norm{\nabla(U_\star-U_\ell)}{L^2(\Omega)}^2$, where $C>0$ depends only
on shape regularity of $\TT_\star$. The first
contribution of the first bracket is estimated as follows
\begin{align*}
 \sum_{T'\in\TT_\star}
 |T'|^{2/d}\,\norm{f+\Delta U_\ell}{L^2(T')}^2
 &= \sum_{T'\in\TT_\star\cap\TT_\ell}
 |T'|^{2/d}\,\norm{f+\Delta U_\ell}{L^2(T')}^2
 \\&\quad
 + \sum_{T'\in\TT_\star\backslash\TT_\ell}
 |T'|^{2/d}\,\norm{f+\Delta U_\ell}{L^2(T')}^2
\end{align*}
Since a refined element $T\in\TT_\ell\backslash\TT_\star$ is the
(essentially disjoint) union of its sons $T'\in\TT_\star\backslash\TT_\ell$ and
$|T'|\le|T|/2$, the second sum is estimated by
\begin{align*}
 \sum_{T'\in\TT_\star\backslash\TT_\ell}
 |T'|^{2/d}\,\norm{f+\Delta U_\ell}{L^2(T')}^2
 \le 2^{-2/d} \sum_{T\in\TT_\ell\backslash\TT_\star}
 |T|^{2/d}\,\norm{f+\Delta U_\ell}{L^2(T)}^2.
\end{align*}
This yields
\begin{align*}
 \sum_{T'\in\TT_\star}
 |T'|^{2/d}\,\norm{f+\Delta U_\ell}{L^2(T')}^2
 &\le\sum_{T\in\TT_\ell}|T|^{2/d}\,\norm{f+\Delta U_\ell}{L^2(T)}^2
 \\&\quad
 - (1-2^{-2/d})\,\sum_{T\in\TT_\ell\backslash\TT_{\star}}|T|^{2/d}\,\norm{f+\Delta U_\ell}{L^2(T)}^2,
\end{align*}
where the sums on the right-hand side only involve contributions of $\eta_\ell(T)^2$.
We employ the same type of argument for the other contributions.
Together with the estimate
\begin{align*}
 \norm{(1-\Pi_\star)\nabla_\Gamma g}{L^2(\partial T'\cap\Gamma_D)}
 \le \norm{(1-\Pi_\ell)\nabla_\Gamma g}{L^2(\partial T'\cap\Gamma_D)}
 \quad \text{for all }T'\in\TT_\star
\end{align*}
which follows from the fact that the $L^2$-projection onto $\PP^{p-1}(\EE_\star^D)$
is even piecewise orthogonal, we are led to
\begin{align*}
 &\Big(\sum_{T'\in\TT_\star}
 |T'|^{2/d}\,\norm{f+\Delta U_\ell}{L^2(T')}^2
 +|T'|^{1/d}\,\norm{[\partial_nU_\ell]}{L^2(\partial T'\cap\Omega)}^2
 \\&\qquad\qquad\qquad\qquad
+|T'|^{1/d}\,\norm{\phi-\partial_n U_\ell}{L^2(\partial T'\cap\Gamma_N)}^2
 +|T'|^{1/d}\,\norm{(1-\Pi_\star)\nabla_\Gamma g}{L^2(\partial T'\cap\Gamma_D)}^2
 \Big)
 \\&\quad
 \le\sum_{T\in\TT_\ell}\eta_\ell(T)^2
 - (1-2^{-1/d})\,\sum_{T\in\TT_\ell\backslash\TT_\star}\eta_\ell(T)^2
 \\&\quad
 \le\big(1-(1-2^{-1/d})\theta\big)\,\eta_\ell^2.
\end{align*}
The final estimate follows from $\MM_\ell\subseteq\TT_\ell\backslash\TT_\star$
and the D\"orfler marking~\eqref{eq:doerfler}, i.e.\ we subtract less.
With $0<q:=1-(1-2^{-1/d})\theta$, we have thus proved
\begin{align*}
 \eta_\star^2 \le (1+\delta)\,q\,\eta_\ell^2
 + (1+\delta^{-1})\,C\,\norm{\nabla(U_\star - U_\ell)}{L^2(\Omega)}^2.
\end{align*}
Finally, we choose $\delta>0$ sufficiently small such that
$0<\q{reduction}:=(1+\delta)\,q<1$ and define $\c{reduction} = (1+\delta^{-1})\,C$.
\end{proof}

\subsection{A~priori convergence of Scott-Zhang projection}
\label{section:apriori:scottzhang}%
We assume that $(\sz_{\ell+1}v)|_T = (\sz_{\ell}v)|_T$
for all $T\in\TT_\ell\cap\TT_{\ell+1}$ with $\omega_{\ell,T}\subseteq\bigcup(\TT_\ell\cap\TT_{\ell+1})$ 
which can always be achieved by an appropriate choice of the dual basis functions in the definition of $\sz_{\ell+1}$. In this section, we prove that under the aforegoing assumptions and for arbitrary refinement, i.e.\ $\TT_\ell = \refine(\TT_{\ell-1})$ for all $\ell\in\N$, the limit of the Scott-Zhang interpolants $\sz_\ell v$ exists in $H^1(\Omega)$ as $\ell\to\infty$.
In particular, this provides the essential ingredient to prove that, under
the same assumptions, the limit of Galerkin solutions $U_\ell$ exists in $H^1(\Omega)$. For 2D and first-order elements $p=1$, this result has first been proved in~\cite{fpp}. Although the proof transfers directly to the present setting, we include it for the sake of completeness.

\begin{proposition}
\label{prop:apriori:scottzhang}%
Let $v\in H^1(\Omega)$. Then, the limit $\sz_\infty v:=\lim_\ell\sz_\ell v$ exists in $H^1(\Omega)$ and defines a continuous linear operator $\sz_\infty:H^1(\Omega)\to H^1(\Omega)$.
\end{proposition}

\begin{proof}
If the limit $\sz_\infty v$ exists for all $v \in H^1(\Omega)$, it is a consequence of the Banach-Steinhaus theorem that $\sz_\infty$ defines, in fact, a linear and continuous operator.
Hence, it remains to prove the existence of $\sz_\infty v$ in $H^1(\Omega)$ for fixed $v \in H^1(\Omega)$. To that end, we follow the ideas from~\cite{msv} and define the following subsets of $\Omega$:
\begin{align*}
\Omega_\ell^0 &:= \textstyle{\bigcup\{T \in \TT_\ell\,:\, \omega_{\ell}(T) \subseteq \bigcup\big(\bigcup_{i=0}^\infty \bigcap_{j=i}^\infty \TT_j}\big)\},\\
\Omega_{\ell} &:= \textstyle{\bigcup\{T \in \TT_\ell\,:\, \text{There exists }k \geq 0\text{ s.t. } \omega_{\ell}(T)\text{ is at least uniformly refined in }\TT_{\ell+k} \}},\\
\Omega_{\ell}^* &:= \Omega\setminus(\Omega_{\ell} \cup \Omega_\ell^0),
\end{align*}
where $\omega_{\ell}(\omega):=\bigcup\{T \in \TT_\ell \,:\, T\cap \omega \neq \emptyset\}$ denotes the patch of $\omega\subseteq\Omega$ with respect to $\TT_\ell$.
In other words, $\Omega_\ell^0$ is the set of all elements whose patch is not refined anymore and thus stays the same in $\TT_k$ for all $k\ge\ell$, whereas $\Omega_\ell$ denotes the set of elements whose patch is uniformly refined at least once after $k$ steps.
According to~\cite[Corollary 4.1]{msv} and uniform shape regularity, it holds that
\begin{align}
\label{eq1:apriorisz}
 \lim_{\ell\to\infty}\norm{\chi_{\Omega_\ell}h_\ell}{L^\infty(\Omega)}=0 = \lim_{\ell \to \infty}\norm{\chi_{\omega_\ell(\Omega_\ell)}h_\ell}{L^\infty(\Omega)},
\end{align}
where $\chi_{\Omega_\ell}:\Omega \to \{0,1\}$ denotes the characteristic function with respect to $\Omega_\ell$ and $\chi_{\omega_\ell(\Omega_\ell)}$ the characteristic function of the patch of $\Omega_\ell$.
Now, let $\eps>0$ be arbitrary. Since the space $H^2(\Omega)$ is dense in $H^1(\Omega)$, we find $v_\eps \in H^2(\Omega)$ such that
$\norm{v-v_\eps}{H^1(\Omega)} \leq \eps$. Due to local approximation and stability properties of $\sz_\ell$ (see~\eqref{eq:scottzhangstability} and~\eqref{eq:scottzhangapprox}), we obtain
\begin{align*}
\norm{(1-\sz_\ell)v}{H^1(\Omega_{\ell})}\lesssim \norm{(1-\sz_\ell)v_\eps}{H^1(\Omega_{\ell})}+\eps\leq \norm{h_\ell\, D^2 v_\eps}{L^2(\omega_\ell(\Omega_{\ell}))}
+\eps,
\end{align*}
where the last estimate is a 
consequence of the Bramble-Hilbert lemma.
By use of~\eqref{eq1:apriorisz}, we may choose $\ell_0 \in \N$ sufficiently large
to guarantee the estimate
$\norm{h_\ell\,D^2 v_\eps}{L^2(\omega_\ell(\Omega_{\ell}))}\leq
\norm{h_\ell}{L^\infty(\omega_\ell(\Omega_{\ell}))}\norm{D^2 v_\eps}{L^2(\Omega)}\leq \eps$ for all $\ell \geq \ell_0$. Then, there holds
\begin{align}
 \label{eq:szunif}
 \norm{(1-\sz_\ell)v}{H^1(\Omega_{\ell})}\lesssim \eps \quad \text{for all }\ell \geq \ell_0.
\end{align}
According to~\cite[Proposition 4.2]{msv}, it holds $\lim_{\ell}|\Omega_\ell^*|=0$. This provides the existence of $\ell_1 \in \N$ such that
\begin{align}
\label{eq:omstar}
 \norm{v}{H^1(\omega_\ell(\Omega_{\ell}^*))}\leq \eps\quad \text{for all } \ell \geq \ell_1
\end{align}
due to the non-concentration of Lebesgue functions and uniform shape regularity, i.e.\ $|\omega_\ell(\Omega_\ell^*)|\lesssim |\Omega_\ell^*|$.
With these preparations, we finally aim at proving that $\sz_\ell v$ is a Cauchy sequence in $H^1(\Omega)$: Let $\ell \geq \max\{\ell_0,\ell_1\}$ and $k\geq 0$ be arbitrary.
First, we use that for any $T \in \TT_\ell$, $(\sz_\ell v)|_T$ depends only on $v|_{\omega_\ell(T)}$ because $\omega_{\ell,T}\subseteq \omega_\ell(T)$. Then, by definition of $\Omega_\ell^0$ and
our assumption on the definition of $\sz_\ell$ and $\sz_{\ell+k}$ on $\TT_\ell \cap \TT_{\ell+k}$, we obtain
\begin{align}
 \label{eq1:szapriori}
\norm{\sz_\ell v- \sz_{\ell+k} v }{H^1(\Omega_\ell^0)} = 0.
\end{align}
Second, due to the local stability of $\sz_\ell$ and~\eqref{eq:omstar}, there holds
\begin{align}
 \label{eq2:szapriori}
\begin{array}{rcl}
\norm{\sz_\ell v - \sz_{\ell+k} v}{H^1(\Omega_{\ell}^*)} &\leq& \norm{\sz_\ell v}{H^1(\Omega_{\ell}^*)}+\norm{\sz_{\ell+k} v}{H^1(\Omega_{\ell}^*)}\\
&\lesssim& \norm{v}{H^1(\omega_\ell(\Omega_{\ell}^*))}+\norm{v}{H^1(\omega_{\ell+k}(\Omega_{\ell}^*))}\\
&\leq& 2 \norm{v}{H^1(\omega_\ell(\Omega_{\ell}^*))} \leq 2 \eps.
\end{array}
\end{align}
Third, we proceed by exploiting~\eqref{eq:szunif}. We have
\begin{align}
 \label{eq3:szapriori}
\norm{\sz_\ell v - \sz_{\ell+k} v}{H^1(\Omega_{\ell})} \leq \norm{\sz_\ell v - v}{H^1(\Omega_{\ell})}+\norm{v-\sz_{\ell+k}v}{H^1(\Omega_{\ell})}
 \lesssim \eps.
\end{align}
Combining the estimates from~\eqref{eq1:szapriori}--\eqref{eq3:szapriori}, we conclude
 $\norm{\sz_\ell v - \sz_{\ell+k} v}{H^1(\Omega)} \lesssim \eps$,
i.e.\ $(\sz_\ell v)$ is a Cauchy sequence in $H^1(\Omega)$ and hence convergent.
\end{proof}

\begin{corollary}
\label{cor:apriori:scottzhang}%
Under the assumptions of Proposition~\ref{prop:apriori:scottzhang}, the limit
$g_\infty := \lim_\ell \sz_\ell g$ exists in $H^{1/2}(\Gamma_D)$.
\end{corollary}

\begin{proof}
Let $\widehat g\in H^{1/2}(\Gamma)$ denote an arbitrary extension of $g$.
With some lifting operator $\LL$, we define $v:=\LL \widehat g$ and note that
$(\sz_\ell v)|_{\Gamma_D} = (\sz_\ell\widehat g)|_{\Gamma_D} = \sz_\ell g$. Since $\sz_\infty v = \lim_\ell \sz_\ell v$ exists in $H^1(\Omega)$, we obtain
\begin{align*}
 \norm{(\sz_\infty v)|_{\Gamma_D} - \sz_\ell g}{H^{1/2}(\Gamma_D)}
 \le \norm{(\sz_\infty v)|_{\Gamma} - \sz_\ell\widehat g}{H^{1/2}(\Gamma)}
 \le \norm{\sz_\infty v - \sz_\ell v}{H^{1}(\Omega)}
 \xrightarrow{\ell\to\infty}0.
\end{align*}
This concludes the proof with $(\sz_\infty v)|_{\Gamma_D} =: g_\infty$.
\end{proof}

\subsection{A~priori convergence of orthogonal projections}
In this subsection, we recall an early observation from~\cite[Lemma~6.1]{bv} which will be applied several times.
We stress that the original proof of~\cite{bv} is
based on the orthogonal projection. However, the argument also works for
(possibly nonlinear) projections with $P_\ell P_k = P_\ell$ for $\ell\le k$ which
satisfy a C\'ea-type quasi-optimality. Since the Scott-Zhang projection satisfies
$\sz_\ell\sz_k \neq \sz_\ell$, in general, Proposition~\ref{prop:apriori:scottzhang} is \emph{not} a consequence of such an abstract result.

\begin{lemma}
\label{lemma:apriori:orthogonal}%
Let $H$ be a Hilbert space and $X_\ell$ be a sequence of closed subspaces with $X_\ell\subseteq X_{\ell+1}$ for all $\ell\ge0$. Let $P_\ell:H\to X_\ell$ denote the $H$-orthogonal projection onto $X_\ell$. Then, for each $x\in H$, the limit $x_\infty:=\lim\limits_{\ell\to\infty} P_\ell x$ exists in $H$.\qed
\end{lemma}

Since the discrete trace spaces $\SS^p(\EE_\ell^D)$ are finite dimensional and hence closed subspaces of $H^{1/2}(\Gamma_D)$, the lemma immediately applies to the $H^{1/2}(\Gamma_D)$-orthogonal projection.

\begin{corollary}
\label{cor:apriori:H12}%
Let $P_\ell:H^{1/2}(\Gamma_D)\to\SS^p(\EE_\ell^D)$ denote the $H^{1/2}(\Gamma_D)$-orthogonal projection. Then, the limit $g_\infty := \lim\limits_{\ell\to\infty} P_\ell g$ exists in $H^{1/2}(\Gamma_D)$.\qed
\end{corollary}

%

\begin{corollary}
\label{cor:apriori:L2}%
Let $\pi_\ell:L^2(\Gamma_D)\to\SS^1(\EE_\ell^D)$ denote the $L^2(\Gamma_D)$-orthogonal projection. Then, the limit $g_\infty := \lim\limits_{\ell\to\infty}\pi_\ell g$ exists weakly in $H^1(\Gamma_D)$ and strongly in $H^\alpha(\Gamma_D)$ for all $0\le \alpha<1$.
\end{corollary}

\begin{proof}
According to Lemma~\ref{lemma:apriori:orthogonal}, the limit $g_\infty = \lim_\ell\pi_\ell g$ exists strongly in $L^2(\Gamma_D)$.
Moreover and according to~\cite[Theorem~8]{kp}, the $\pi_\ell$ are uniformly stable in $H^1(\Gamma_D)$, since we use newest vertex bisection. Hence, the sequence $(\pi_\ell g)$ is uniformly bounded in $H^1(\Gamma_D)$ and thus admits a weakly convergent subsequence $(\pi_{\ell_k}g)$ with weak limit $\widetilde g_\infty\in H^1(\Gamma_D)$, where weak convergence is understood in $H^1(\Gamma_D)$. Since the inclusion $H^1(\Gamma_D)\subset L^2(\Gamma_D)$ is compact, the sequence $(\pi_{\ell_k} g)$ converges strongly to $\widetilde g_\infty$ in $L^2(\Gamma_D)$. From uniqueness of limits, we conclude $\widetilde g_\infty = g_\infty$. Iterating this argument, we see that each subsequence of $(\pi_\ell g)$ contains a subsequence which converges weakly to $g_\infty$ in $H^1(\Gamma_D)$. This proves that the entire sequence converges weakly to $g_\infty$ in $H^1(\Gamma_D)$. Strong convergence in $H^\alpha(\Gamma_D)$ follows by compact inclusion $H^1(\Gamma_D)\subset H^\alpha(\Gamma_D)$ for all $0\le \alpha<1$.
\end{proof}

\subsection{A~priori convergence of Galerkin solutions}
We now show that the limit of Galerkin solutions $U_\ell$ exists as $\ell\to\infty$
provided that the meshes are nested, i.e.\ $\TT_{\ell+1} = \refine(\TT_\ell)$.

\begin{proposition}
\label{prop:apriori:galerkin}%
Under Assumption~\eqref{eq:apriori:dirichlet} that $g_\infty := \lim_\ell\P_\ell g$ exists in $H^{1/2}(\Gamma)$, also the limit
$U_\infty := \lim_\ell U_\ell$ of Galerkin solutions exists in $H^1(\Omega)$.
\end{proposition}

\begin{proof}
We consider the continuous auxiliary problem
\begin{align*}
 -\Delta w_\ell&=0 \quad\text{in }\Omega,\\
 w_\ell &= \P_\ell g\quad\text{on }\Gamma_D,\\
 \partial_nw_\ell &=0 \quad\text{on }\Gamma_N.
\end{align*}
Let $w_\ell\in H^1(\Omega)$ be the unique (weak) solution and note that the trace $\widehat g_\ell:=w_\ell|_\Gamma \in H^{1/2}(\Gamma)$ provides an extension of $\P_\ell g$ with
\begin{align*}
 \norm{\widehat g_\ell}{H^{1/2}(\Gamma)}
 \le \norm{w_\ell}{H^1(\Omega)}
 \lesssim \norm{\P_\ell g}{H^{1/2}(\Gamma_D)}
 \le \norm{\widehat g_\ell}{H^{1/2}(\Gamma)}.
\end{align*}
For arbitrary $k,\ell\in\N$, the same type of arguments proves
\begin{align*}
 \norm{\widehat g_\ell-\widehat g_k}{H^{1/2}(\Gamma)}
 \simeq \norm{(\P_\ell-\P_k)g}{H^{1/2}(\Gamma_D)}.
\end{align*}
According to Assumption~\eqref{eq:apriori:dirichlet},
$(\P_\ell g)$ is a Cauchy sequence in $H^{1/2}(\Gamma_D)$. Therefore, $(\widehat g_\ell)$ is a Cauchy sequence in $H^{1/2}(\Gamma)$, whence convergent with limit $\widehat g_\infty\in H^{1/2}(\Gamma)$.
Next, note that $(\sz_\ell\LL \widehat g_\ell)|_{\Gamma_D} = \P_\ell g$,
where $\LL:H^{1/2}(\Gamma)\to H^1(\Omega)$ denotes
some lifting operator. Therefore, $\widetilde U_\ell := U_\ell - \sz_\ell\LL \widehat g_\ell \in \SS^p_D(\TT_\ell)$ is the unique solution of the variational formulation
\begin{align}
\label{eq1:weak3d}
 \dual{\nabla\widetilde U_\ell}{\nabla V_\ell}_\Omega=\dual{\nabla u}{\nabla V_\ell}_\Omega-\dual{\nabla\sz_\ell\LL \widehat g_\ell}{\nabla V_\ell}_\Omega \quad \text{ for all } V_\ell \in \SS^p_D(\TT_\ell).
\end{align}
Finally, we need to show that $\widetilde U_\ell$ and $\sz_\ell\LL \widehat g_\ell$ are convergent to conclude convergence of $U_\ell = \widetilde U_\ell + \sz_\ell\LL \widehat g_\ell$.

With convergence of $(\widehat g_\ell)$ to $\widehat g_\infty$ and Proposition~\ref{prop:apriori:scottzhang}, we obtain
\begin{align*}
 \norm{\sz_\ell\LL \widehat g_\ell - \sz_\infty \LL \widehat g_\infty}{H^1(\Omega)}
 &\le \norm{\sz_\ell(\LL \widehat g_\ell - \LL \widehat g_\infty)}{H^1(\Omega)}
 + \norm{\sz_\ell \LL \widehat g_\infty - \sz_\infty \LL \widehat g_\infty}{H^1(\Omega)}\\
 &\lesssim
 \norm{\widehat g_\ell - \widehat g_\infty}{H^{1/2}(\Gamma)}
 +\norm{\sz_\ell \LL \widehat g_\infty - \sz_\infty \LL \widehat g_\infty}{H^1(\Omega)}
 \xrightarrow{\ell\to\infty} 0.
\end{align*}
This proves convergence of $\sz_\ell\LL \widehat g_\ell$ to $\sz_\infty\LL\widehat g_\infty$ as $\ell\to\infty$. To see convergence of $\widetilde U_\ell$, let $\widetilde U_{\ell,\infty} \in \SS^p_D(\TT_\ell)$ be the unique solution of the discrete auxiliary problem
\begin{align}
\label{eq2:weak3d}
\dual{\nabla\widetilde U_{\ell,\infty}}{\nabla V_\ell}_\Omega=\dual{\nabla u}{\nabla V_\ell}_\Omega-\dual{\nabla \sz_\infty\LL \widehat g_\infty}{\nabla V_\ell}_\Omega
\quad \text{ for all } V_\ell \in \SS^p_D(\TT_\ell).
\end{align}
Due to the nestedness of the ansatz spaces $\SS^p_D(\TT_\ell)$, Lemma~\ref{lemma:apriori:orthogonal} predicts a priori convergence $\widetilde U_{\ell,\infty} \xrightarrow{\ell\to\infty} \widetilde U_\infty\in H^1(\Omega_D)$.
With the stability of~\eqref{eq1:weak3d} and~\eqref{eq2:weak3d}, we obtain
\begin{align*}
 \norm{\nabla(\widetilde U_{\ell,\infty} - \widetilde U_{\ell})}{L^2(\Omega)} \lesssim \norm{\sz_\ell\LL \widehat g_\ell- \sz_\infty \LL \widehat g_\infty}{H^1(\Omega)} \xrightarrow{\ell\to\infty} 0,
\end{align*}
and therefore $\widetilde U_\ell \xrightarrow{\ell\to\infty} \widetilde U_\infty$ in $H^1(\Omega_D)$.

Finally, we now conclude
\begin{align*}
 U_\ell = \widetilde U_\ell + \sz_\ell\LL \widehat g_\ell \xrightarrow{\ell\to\infty} \widetilde U_\infty + \sz_\infty \LL \widehat g_\infty =:U_\infty \in H^1(\Omega),
\end{align*}
which concludes the proof.
\end{proof}

\subsection{Proof of convergence theorem (Theorem~\ref{theorem:convergence})}
(i) Since the limit $U_\infty = \lim_\ell U_\ell$ exists in $H^1(\Omega)$, we infer $\lim_\ell\norm{\nabla(U_{\ell+1}-U_\ell)}{L^2(\Omega)}=0$.
In view of this and Lemma~\ref{lemma:doerfler1}, the estimator reduction estimate~\eqref{eq:reduction} takes the form
\begin{align*}
 \eta_{\ell+1}^2
 \le \q{reduction}\,\eta_{\ell}^2
 + \alpha_\ell
 \quad\text{for all }\ell\ge0
\end{align*}
with some non-negative $\alpha_\ell\ge0$ such that $\lim_\ell\alpha_\ell=0$, i.e.\ the estimator is contractive up to a non-negative zero sequence. It is a consequence of elementary calculus that $\lim_\ell\eta_\ell=0$, see e.g.\ \cite[Lemma~2.3]{afp}. Finally, reliability $\norm{u-U_\ell}{H^1(\Omega)} \lesssim \eta_\ell$ thus concludes the proof.

(ii) The verification of Assumption~\eqref{eq:apriori:dirichlet} is done in
Corollary~\ref{cor:apriori:scottzhang} for the Scott-Zhang projection,
Corollary~\ref{cor:apriori:H12} for the $H^{1/2}(\Gamma_D)$-orthogonal projection,
and Corollary~\ref{cor:apriori:L2} for the $L^2(\Gamma_D)$-orthogonal projection.
\qed

\section{Contraction}
\label{section:contraction}%

\noindent%
In principle, the convergence rate of $\lim_\ell U_\ell = u$ from Theorem~\ref{theorem:convergence} could be slow. Moreover, Theorem~\ref{theorem:convergence} restricts the Dirichlet projection $\P_\ell$ by Assumption~\eqref{eq:apriori:dirichlet}. In this section, we aim to show linear convergence for some quasi-error quantity $\Delta_\ell \simeq \eta_\ell^2 = \varrho_\ell^2 + \oscD\ell^2$ with respect to the step $\ell$ of Algorithm~\ref{algorithm2} and independently of the projection $\P_\ell$ chosen. The essential observation is that the marking step in Algorithm~\ref{algorithm2} is in some sense independent of the $\P_\ell$ chosen.

\subsection{Implicit D\"orfler marking}
Let $\widetilde U_\ell\in\SS^p(\TT_\ell)$ be a Galerkin solution of~\eqref{eq:galerkin} with different Dirichlet data $\widetilde U_\ell = P_\ell g$ on $\Gamma_D$, where $P_\ell:H^{1/2}(\Gamma_D)\to\SS^p(\EE_\ell^D)$ is a uniformly stable projection onto $\SS^p(\EE_\ell^D)$ in the sense of~\eqref{eq:stability}. Let $\widetilde\eta_\ell^{\,2} = \widetilde\varrho_\ell^{\,2} + \oscD\ell^2$ be the associated error estimator. In the following, we prove that marking in Algorithm~\ref{algorithm2} with $\eta_\ell^2 = \varrho_\ell^2 + \oscD\ell^2$ and sufficiently small $0<\vartheta<1$ implicitly implies the simple D\"orfler marking~\eqref{eq:doerfler} for $\widetilde\eta_\ell$.

\begin{lemma}[local equivalence of error estimators for different projections]
\label{lemma:equivalence}%
For arbitrary $\UU_\ell\subseteq\TT_\ell$, it holds that
\begin{align}\label{eq:equivalence2}
 \c{equivalence}^{-1}\sum_{T\in\UU_\ell}\varrho_\ell(T)^2
 \le \sum_{T\in\UU_\ell}\widetilde\varrho_\ell(T)^2 + \oscD\ell^2
 \quad\text{and}\quad
 \c{equivalence}^{-1}\sum_{T\in\UU_\ell}\widetilde\varrho_\ell(T)^2
 \le \sum_{T\in\UU_\ell}\varrho_\ell(T)^2 + \oscD\ell^2.
\end{align}
The constant $\setc{equivalence}>1$ depends only on shape regularity of $\TT_\ell$ and on $\c{stability}>0$. In particular, this implies equivalence \begin{align}
 (\c{equivalence}+1)^{-1}\,\eta_\ell^2 \le \widetilde\eta_\ell^{\,2}
 \le (\c{equivalence}+1)\,\eta_\ell^2.
\end{align}
\end{lemma}

\begin{proof}
Arguing as for the estimator reduction, it follows from the triangle inequality and scaling arguments that
\begin{align*}
 \varrho_\ell(T)^2
 \lesssim \widetilde\varrho_\ell(T)^2
 + \norm{\nabla(U_\ell-\widetilde U_\ell)}{L^2(\omega_\ell(T))}^2
 \quad\text{for all }T\in\TT_\ell,
\end{align*}
where $\omega_\ell(T) = \bigcup\set{T'\in\TT_\ell}{T'\cap T\neq\emptyset}$ denotes the element patch of $T$. Consequently, a rough estimate gives
\begin{align*}
 \sum_{T\in\UU_\ell}\varrho_\ell(T)^2
 \lesssim \sum_{T\in\UU_\ell}\widetilde\varrho_\ell(T)^2
 + \norm{\nabla(U_\ell-\widetilde U_\ell)}{L^2(\Omega)}^2
 \quad\text{for all }\UU_\ell\subseteq\TT_\ell.
\end{align*}
Recall the Galerkin orthogonality
\begin{align*}
 \dual{\nabla(U_\ell-\widetilde U_\ell)}{\nabla V_\ell}
 = \dual{\nabla(u-\widetilde U_\ell)}{\nabla V_\ell}
 - \dual{\nabla(u-U_\ell)}{\nabla V_\ell} = 0
 \quad\text{for all }V_\ell\in\SS^p_D(\TT_\ell).
\end{align*}
Let $\widehat g\in H^{1/2}(\Gamma)$ be an arbitrary extension of $(U_\ell-\widetilde U_\ell)|_{\Gamma_D}=(\P_\ell-P_\ell)g \in H^{1/2}(\Gamma_D)$. We choose the test function $V_\ell = (U_\ell-\widetilde U_\ell) - \sz_\ell\LL\widehat g\in\SS^1_D(\TT_\ell)$ to see
\begin{align*}
 \norm{\nabla(U_\ell-\widetilde U_\ell)}{L^2(\Omega)}^2
 = \dual{\nabla(U_\ell-\widetilde U_\ell)}{\nabla\sz_\ell\LL\widehat g}_\Omega.
\end{align*}
Stability of Scott-Zhang projection $\sz_\ell$ and lifting operator $\LL$ thus give
\begin{align*}
 \norm{\nabla(U_\ell-\widetilde U_\ell)}{L^2(\Omega)}
 \le \norm{\nabla\sz_\ell\LL\widehat g}{L^2(\Omega)}
 \lesssim \norm{\widehat g}{H^{1/2}(\Gamma)}.
\end{align*}
Since $\widehat g$ was an arbitrary extension of $(\P_\ell-P_\ell)g$, we end up with
\begin{align*}
 \norm{\nabla(U_\ell-\widetilde U_\ell)}{L^2(\Omega)}
 \lesssim \norm{(\P_\ell-P_\ell)g}{H^{1/2}(\Gamma_D)}
 &\le \norm{(\P_\ell-1)g}{H^{1/2}(\Gamma_D)}
 + \norm{(1-P_\ell)g}{H^{1/2}(\Gamma_D)}
 \\&
 \lesssim\oscD\ell,
\end{align*}
where we have used Corollary~\ref{cor:oscD}.
This proves the first estimate in~\eqref{eq:equivalence2}, and the second follows with the same arguments.
\end{proof}

The following lemma is the main reason, why we stick with Stevenson's modified D\"orfler marking~\eqref{eq:doerfler:res}--\eqref{eq:doerfler:osc} instead of simple D\"orfler marking~\eqref{eq:doerfler}.

\begin{lemma}[modified D\"orfler marking implies D\"orfler marking for different projection]
\label{lemma:doerfler2}%
For arbitrary $0<\theta_1,\theta_2<1$ and sufficiently small $0<\vartheta<1$, there is some $0<\theta<1$ such that the marking criterion~\eqref{eq:doerfler:res}--\eqref{eq:doerfler:osc} for $\eta_\ell^2 = \varrho_\ell^2 + \oscD\ell^2$ implies the D\"orfler marking
\begin{align}\label{eq:doerfler2}
 \theta\,\widetilde\eta_\ell^{\,2}
 \le \sum_{T\in\MM_\ell}\widetilde\eta_\ell(T)^2
\end{align}
for $\widetilde\eta_\ell^{\,2} = \widetilde\varrho_\ell^{\,2}+\oscD\ell^2$. The parameter $0<\theta<1$ depends on $0<\theta_1,\theta_2,\vartheta<1$ and on $\c{equivalence}>0$.
\end{lemma}

\begin{proof}
We argue as in the proof of Lemma~\ref{lemma:doerfler1}.
First, assume $\oscD\ell^2 \le \vartheta\,\varrho_\ell^2$ and let $\MM_\ell\subseteq\TT_\ell$ satisfy~\eqref{eq:doerfler:res}. According to
Lemma~\ref{lemma:equivalence}, we see
\begin{align*}
 \theta_1\,\eta_\ell^2
 \le \theta_1(1+\vartheta)\varrho_\ell^2
 \le (1+\vartheta)\sum_{T\in\MM_\ell}\varrho_\ell(T)^2
 &\le\c{equivalence} (1+\vartheta)\Big(
 \sum_{T\in\MM_\ell}\widetilde\varrho_\ell(T)^2 + \oscD\ell^2\Big)
 \\&\le\c{equivalence} (1+\vartheta)\Big(\sum_{T\in\MM_\ell}\widetilde\varrho_\ell(T)^2 + \vartheta\,\varrho_\ell^2\Big).
\end{align*}
This proves
\begin{align*}
 \big(\theta_1\c{equivalence}^{-1}(1+\vartheta)^{-1}-\vartheta\big)
 \,\eta_\ell^2
 \le \sum_{T\in\MM_\ell}\widetilde\varrho_\ell(T)^2.
\end{align*}
Together with $(\c{equivalence}+1)^{-1}\,\widetilde\eta_\ell^{\,2}\le \eta_\ell^2$, we thus obtain the D\"orfler marking~\eqref{eq:doerfler2} with
$0<\theta\le (\c{equivalence}+1)^{-1}\big(\theta_1\c{equivalence}^{-1}(1+\vartheta)^{-1} -\vartheta\big)<1$, provided that $0<\vartheta<1$ is sufficiently small
compared to $0<\theta_1<1$.

Second, assume $\oscD\ell^2 > \vartheta\,\varrho_\ell^2$ and let $\MM_\ell^D\subseteq\EE_\ell^D$ satisfy~\eqref{eq:doerfler:osc}.
Then,
\begin{align*}
 \theta_2\,\eta_\ell^2
 \le \theta_2(1+\vartheta^{-1})\oscD\ell^2
 \le (1+\vartheta^{-1})\sum_{E\in\MM_\ell^D}\oscD\ell(E)^2
 \le (1+\vartheta^{-1})\sum_{T\in\MM_\ell}\widetilde\eta_\ell(T)^2,
\end{align*}
where $\MM_\ell = \set{T\in\TT_\ell}{\exists E\in\MM_\ell^D\quad
E\subset\partial T}$ is defined in step~(iv) of Algorithm~\ref{algorithm2}.
As before $(\c{equivalence}+1)^{-1}\,\widetilde\eta_\ell^{\,2} \le \eta_\ell^2$ thus proves~\eqref{eq:doerfler} with
$0<\theta \le \c{equivalence}^{-1}\theta_2(1+\vartheta^{-1})^{-1}<1$.
\end{proof}

\subsection{Quasi-Pythagoras theorem}
To prove Theorem~\ref{theorem:contraction}, we consider a \emph{theoretical} auxiliary problem: Throughout the remainder of Section~\ref{section:contraction}, $\widetilde U_\ell\in\SS^p(\TT_\ell)$ denotes the Galerkin solution of~\eqref{eq:galerkin} with Dirichlet data $\widetilde U_\ell = P_\ell g$ on $\Gamma_D$, where $P_\ell:H^{1/2}(\Gamma_D)\to\SS^p(\EE_\ell^D)$ denotes the $H^{1/2}(\Gamma_D)$-orthogonal projection. Associated with $\widetilde U_\ell$ is the error estimator $\widetilde\eta_\ell^{\,2} = \widetilde\varrho_\ell^{\,2}+\oscD\ell^2$, where $\widetilde\varrho_\ell$ is defined in~\eqref{eq2:estimator2:T} with $U_\ell$ replaced by $\widetilde U_\ell$.

Recall that the aforegoing statements of Section~\ref{section:scottzhang} and Section~\ref{section:convergence} hold for \emph{any} uniformly $H^{1/2}(\Gamma_D)$-stable projection $\P_\ell$ and thus apply to $\widetilde\eta_\ell^{\,2} = \widetilde\varrho_\ell^{\,2}+\oscD\ell^2$. We shall need reliability $\norm{u-\widetilde U_\ell}{H^1(\Omega)}^2 \lesssim \widetilde\eta_\ell^{\,2}$ as well as the estimator reduction~\eqref{eq:reduction} from Proposition~\ref{prop:reduction} for $\widetilde\eta_\ell^{\,2}$, which is a consequence of Lemma~\ref{lemma:doerfler2}. Our concept of proof of Theorem~\ref{theorem:contraction} goes back to~\cite[Proof of Theorem~4.1]{ckns}. Therein, however, the proof relies on the Pythagoras theorem $\norm{\nabla(u-U_\ell)}{L^2(\Omega)}^2 = \norm{\nabla(u-U_{\ell+1})}{L^2(\Omega)}^2 + \norm{\nabla(U_{\ell+1}-U_\ell)}{L^2(\Omega)}^2$ which does \emph{not} hold in case of inhomogeneous Dirichlet data and $\P_\ell g\neq\P_{\ell+1}g$, in general. Instead, we rely on a perturbed Pythagoras theorem which will be used for the auxiliary problem.

\begin{lemma}[quasi-Pythagoras theorem]
\label{lemma:pythagoras}
Let $\TT_\star = \refine(\TT_\ell)$ be an arbitrary refinement of $\TT_\ell$ with the associated auxiliary solution $\widetilde U_\star \in \SS^p(\TT_\star)$, where $\widetilde U_\star = P_\star g$ on $\Gamma_D$. Then, \begin{align}\label{eq:pythagoras2}
\begin{split}
 (1-\alpha)\,&\norm{\nabla(u-\widetilde U_\star)}{L^2(\Omega)}^2
 \le \norm{\nabla(u-\widetilde U_\ell)}{L^2(\Omega)}^2
 - \norm{\nabla(\widetilde U_\star - \widetilde U_\ell)}{L^2(\Omega)}^2
 \\&\qquad
 + \alpha^{-1}\,\c{pythagoras}\,\norm{(P_\star-P_\ell)g}{H^{1/2}(\Gamma_D)}^2
\end{split}
\end{align}
for all $\alpha>0$.
%
The constant $\setc{pythagoras}>0$ depends only on the shape regularity of $\sigma(\TT_\ell)$ and $\sigma(\TT_\star)$ and on $\Omega$ and $\Gamma_D$.
\end{lemma}

\begin{proof}
We recall the Galerkin orthogonality
\begin{align*}
 \dual{\nabla(u-\widetilde U_\star)}{\nabla V_\star}_\Omega = 0
 \quad\text{for all }V_\star\in\SS^p_D(\TT_\star).
\end{align*}
Now, let $\widetilde U_\star^\ell\in\SS^p(\TT_\star)$ be the unique Galerkin solution of~\eqref{eq:galerkin} with $\widetilde U_\star^\ell|_{\Gamma_D}=P_\ell g$. We use the Galerkin orthogonality with $V_\star = \widetilde U_\star^\ell-\widetilde U_\ell\in\SS^p_D(\TT_\star)$. This and the Young inequality allow to estimate the $L^2$-scalar product
\begin{align*}
 2|\dual{\nabla&(u-\widetilde U_\star)}{\nabla(\widetilde U_\star - \widetilde U_\ell)}_\Omega|
 = 2 |\dual{\nabla(u-\widetilde U_\star)}{\nabla(\widetilde U_\star - \widetilde U_\star^\ell)}_\Omega|
 \\&
 \le \alpha\norm{\nabla(u-\widetilde U_\star)}{L^2(\Omega)}^2
 + \alpha^{-1}\norm{\nabla(\widetilde U_\star - \widetilde U_\star^\ell)}{L^2(\Omega)}^2
\end{align*}
for all $\alpha>0$. To estimate the second contribution on the right-hand side, we choose an arbitrary extension $\widehat g\in H^{1/2}(\Gamma)$ of $(P_\star-P_\ell)g\in H^{1/2}(\Gamma_D)$. Then, we use the test function $V_\star = (\widetilde U_\star-\widetilde U_\star^\ell)-\sz_\star\LL\widehat g\in\SS^p_D(\TT_\star)$, where $\LL:H^{1/2}(\Gamma)\to H^1(\Omega)$ again denotes a lifting operator. Recall that the choice of the Scott-Zhang projection $\sz_\star$ guarantees that this function has zero Dirichlet data on $\Gamma_D$ since $\widehat g|_{\Gamma_D} = (P_\star-P_\ell)g\in\SS^p(\EE_\star^D)$. Now, the Galerkin orthogonalities
for $\widetilde U_\star,\widetilde U_\star^\ell\in\SS^p(\TT_\star)$ yield
\begin{align*}
 0 = \dual{\nabla(u-\widetilde U_\star^\ell)}{\nabla V_\star}_\Omega
 - \dual{\nabla(u-\widetilde U_\star)}{\nabla V_\star}_\Omega
 = \dual{\nabla(\widetilde U_\star-\widetilde U_\star^\ell)}{\nabla V_\star}_\Omega.
\end{align*}
By the above choice of $V_\star\in\SS^p_D(\TT_\star)$ and stability of $\sz_\star$ and $\LL$, this yields
\begin{align*}
 \norm{\nabla(\widetilde U_\star-\widetilde U_\star^\ell)}{L^2(\Omega)}^2
 = \dual{\nabla(\widetilde U_\star-\widetilde U_\star^\ell)}{\nabla \sz_\star\LL\widehat g}_\Omega
 \lesssim \norm{\nabla(\widetilde U_\star-\widetilde U_\star^\ell)}{L^2(\Omega)}
 \norm{\widehat g}{H^{1/2}(\Gamma)}.
\end{align*}
Since $\widehat g$ was an arbitrary extension of $(P_\star-P_\ell)g\in H^{1/2}(\Gamma_D)$ to $H^{1/2}(\Gamma)$, this yields
\begin{align*}
 \norm{\nabla(\widetilde U_\star-\widetilde U_\star^\ell)}{L^2(\Omega)}
 \lesssim \norm{(P_\star-P_\ell)g}{H^{1/2}(\Gamma_D)}.
\end{align*}
So far, we have thus derived
\begin{align*}
 2|\dual{\nabla(u-\widetilde U_\star)}{\nabla(\widetilde U_\star-\widetilde U_\ell)}_\Omega|
 \le \alpha \norm{\nabla(u-\widetilde U_\star)}{L^2(\Omega)}^2
 + \alpha^{-1}\,\c{pythagoras}\,\norm{(P_\star-P_\ell)g}{H^{1/2}(\Gamma_D)}^2,
\end{align*}
To verify~\eqref{eq:pythagoras2}, 
we use the identity
\begin{align*}
 &\norm{\nabla(u-\widetilde U_\ell)}{L^2(\Omega)}^2
 = \norm{\nabla\big((u-\widetilde U_\star)+(\widetilde U_\star - \widetilde U_\ell)\big)}{L^2(\Omega)}^2
 \\&\qquad
 = \norm{\nabla(u-\widetilde U_\star)}{L^2(\Omega)}^2
 + 2\dual{\nabla(u-\widetilde U_\star)}{\nabla(\widetilde U_\star - \widetilde U_\ell)}_\Omega
 + \norm{\nabla(\widetilde U_\star - \widetilde U_\ell)\big)}{L^2(\Omega)}^2.
\end{align*}
Rearranging the terms accordingly and use of the estimate for the scalar product, we conclude the proof.
\end{proof}

\subsection{Proof of contraction theorem (Theorem~\ref{theorem:contraction})}
Using the quasi-Pythagoras theorem~\eqref{eq:pythagoras2} with $\TT_\star=\TT_{\ell+1}$, we see
\begin{align*}
 (1-\alpha)\,\norm{\nabla(u-\widetilde U_{\ell+1})}{L^2(\Omega)}^2
 &\le \norm{\nabla(u-\widetilde U_{\ell})}{L^2(\Omega)}^2
 - \norm{\nabla(\widetilde U_{\ell+1}-\widetilde U_{\ell})}{L^2(\Omega)}^2
 \\&\qquad
 + \alpha^{-1}\c{pythagoras}\,\norm{(P_{\ell+1}-P_\ell)g}{H^{1/2}(\Gamma_D)}^2.
\end{align*}
The use of the $H^{1/2}(\Gamma_D)$-orthogonal projection provides the orthogonality relation
\begin{align*}
\norm{(1-P_{\ell+1})g}{H^{1/2}(\Gamma_D)}^2
+ \norm{(P_{\ell+1}-P_\ell)g}{H^{1/2}(\Gamma_D)}^2
= \norm{(1-P_\ell)g}{H^{1/2}(\Gamma_D)}^2.
\end{align*}
Combining the last two estimates, we obtain
\begin{align*}
 (1-\alpha)\,&\norm{\nabla(u-\widetilde U_{\ell+1})}{L^2(\Omega)}^2
 + \alpha^{-1}\c{pythagoras}\norm{(1-P_{\ell+1})g}{H^{1/2}(\Gamma_D)}^2
 \\&
 \le \norm{\nabla(u-\widetilde U_{\ell})}{L^2(\Omega)}^2
 + \alpha^{-1}\c{pythagoras}\norm{(1-P_{\ell})g}{H^{1/2}(\Gamma_D)}^2
 - \norm{\nabla(\widetilde U_{\ell+1}-\widetilde U_{\ell})}{L^2(\Omega)}^2.
\end{align*}
Applying Lemma~\ref{lemma:doerfler2}, we see that Algorithm~\ref{algorithm2} for $\eta_\ell^2 = \varrho_\ell^2 + \oscD\ell^2$ implicitly implies the D\"orfler marking~\eqref{eq:doerfler2} (resp.~\eqref{eq:doerfler}) for $\widetilde\eta_\ell^{\,2} = \widetilde\varrho_\ell^{\,2}+\oscD\ell^2$. Therefore, the estimator reduction~\eqref{eq:reduction} of Proposition~\ref{prop:reduction} applies to the auxiliary problem and provides
\begin{align*}
 \widetilde\eta_{\ell+1}^{\,2}
 \le\q{reduction}\,\widetilde\eta_{\ell}^{\,2}
 +\c{reduction}\norm{\nabla(\widetilde U_{\ell+1}-\widetilde U_\ell)}{L^2(\Omega)}^2
 \quad\text{for all }\ell\ge0.
\end{align*}
Now, we add the last two estimates to see, for $\beta>0$,
\begin{align*}
 (1-\alpha)\,&\norm{\nabla(u-\widetilde U_{\ell+1})}{L^2(\Omega)}^2
 + \alpha^{-1}\c{pythagoras}\norm{(1-P_{\ell+1})g}{H^{1/2}(\Gamma_D)}^2
 + \beta\,\widetilde\eta_{\ell+1}^{\,2}\\
 &\le \norm{\nabla(u-\widetilde U_{\ell})}{L^2(\Omega)}^2
 + \alpha^{-1}\c{pythagoras}\norm{(1-P_{\ell})g}{H^{1/2}(\Gamma_D)}^2
 + \beta\q{reduction}\,\widetilde\eta_{\ell}^{\,2}\\
 &\qquad
 + (\beta\c{reduction}-1)
 \,\norm{\nabla(\widetilde U_{\ell+1}-\widetilde U_{\ell})}{L^2(\Omega)}^2.
\end{align*}
We choose $\beta>0$ sufficiently small to guarantee $\beta\c{reduction}-1\le0$, i.e.\ the last term on the right-hand side of the last estimate can be omitted. Then, we use the reliability
$\norm{u-\widetilde U_\ell}{H^1(\Omega)}^2\lesssim\widetilde\eta_\ell^{\,2}$ and the estimate $\norm{(1-P_\ell)g}{H^{1/2}(\Gamma_D)}^2\lesssim\oscD\ell^2\le\widetilde\eta_\ell^{\,2}$ from Corollary~\ref{cor:oscD} in the form
\begin{align*}
 \norm{\nabla(u-\widetilde U_\ell)}{L^2(\Omega)}^2 
 + \norm{(1-P_\ell)g}{H^{1/2}(\Gamma_D)}^2
 \le C\,\widetilde\eta_\ell^{\,2}
\end{align*}
to see, for arbitrary $\gamma,\delta>0$
\begin{align*}
 (1-&\alpha)\,\norm{\nabla(u-\widetilde U_{\ell+1})}{L^2(\Omega)}^2
 + \alpha^{-1}\c{pythagoras}\norm{(1-P_{\ell+1})g}{H^{1/2}(\Gamma_D)}^2
 + \beta\,\widetilde\eta_{\ell+1}^{\,2}\\
 &\le \norm{\nabla(u-\widetilde U_{\ell})}{L^2(\Omega)}^2
 + \alpha^{-1}\c{pythagoras}\norm{(1-P_{\ell})g}{H^{1/2}(\Gamma_D)}^2
 + \beta\q{reduction}\,\widetilde\eta_{\ell}^{\,2}\\
 &\le (1-\gamma\beta C^{-1})\,\norm{\nabla(u-\widetilde U_{\ell})}{L^2(\Omega)}^2
 + (1-\delta\beta C^{-1})\alpha^{-1}\c{pythagoras}\,\norm{(1-P_{\ell})g}{H^{1/2}(\Gamma_D)}^2\\
 &\qquad+\beta(\q{reduction}+\gamma+\delta\alpha^{-1}\c{pythagoras})\,
 \widetilde\eta_\ell^{\,2}.
\end{align*}
For $0<\alpha<1$, we may now rearrange this estimate to end up with
\begin{align*}
 &\norm{\nabla(u-\widetilde U_{\ell+1})}{L^2(\Omega)}^2
 + \frac{\c{pythagoras}}{\alpha(1-\alpha)}\,
 \norm{(1-P_{\ell+1})g}{H^{1/2}(\Gamma_D)}^2
 + \frac{\beta}{1-\alpha}\,\widetilde\eta_{\ell+1}^{\,2}
 \\&\qquad
 \le \frac{1-\gamma\beta C^{-1}}{1-\alpha}\,
 \norm{\nabla(u-\widetilde U_{\ell})}{L^2(\Omega)}^2
 + (1-\delta\beta C^{-1})\,\frac{\c{pythagoras}}{\alpha(1-\alpha)}\,
 \norm{(1-P_{\ell})g}{H^{1/2}(\Gamma_D)}^2
 \\&\qquad\qquad
 + (\q{reduction}+\gamma+\delta\alpha^{-1}\c{pythagoras})\,\frac{\beta}{1-\alpha}\,
 \widetilde\eta_\ell^{\,2}.
\end{align*}
It remains to choose the free constants $0<\alpha,\gamma,\delta<1$, whereas
$\beta>0$ has already been fixed:
\begin{itemize}
\item First, choose $0<\gamma<1$ sufficiently small to guarantee
$0<\q{reduction}+\gamma<1$ and $0<\gamma\beta C^{-1}<1$.
\item Second, choose $0<\alpha<1$ sufficiently small such that
$0<(1-\gamma\beta C^{-1})/(1-\alpha)<1$.
\item Third, choose $\delta>0$ sufficiently small with $\q{reduction}+\gamma+\delta\alpha^{-1}\c{pythagoras}<1$.
\end{itemize}
With $\mu:=\beta/(1-\alpha)$, $\lambda:=\alpha^{-1}\c{pythagoras}/(1-\alpha)$,
and $0<\kappa<1$ the maximal contraction constant of the three contributions,
we end up with the contraction estimate~\eqref{eq:contraction}.

It thus only remains to prove equivalence~\eqref{eq:equivalence}: According to the
definition of $\Delta_\ell$ in~\eqref{eq:contraction}, we have equivalence
$\Delta_\ell \simeq \widetilde\eta_\ell^{\,2}$.
Finally, Lemma~\ref{lemma:equivalence} implies
$\widetilde\eta_\ell^{\,2} \simeq \eta_\ell^2$
and concludes the proof.
\qed

\bigskip

\begin{remark}
For $\P_\ell = P_\ell$ the $H^{1/2}(\Gamma_D)$-orthogonal projection, the
proof reveals that Theorem~\ref{theorem:contraction} holds without any restriction on the adaptivity parameters and even for simple D\"orfler marking~\eqref{eq:doerfler}.
\end{remark}

%

\section{Quasi-optimality}
\label{section:optimality}%

\noindent
In this section, we aim to prove Theorem~\ref{theorem:quasioptimal}--\ref{theorem:approximationclass}. In some sense, the heart of the matter of the quasi-optimality analysis is the discrete local reliability
of Proposition~\ref{prop:dlr}. This is, however, only proved for discrete Dirichlet data obtained
by the Scott-Zhang projection. We therefore consider this as an auxiliary
problem: Let $\widetilde U_\ell \in \SS^p(\TT_\ell)$ denote the Galerkin solution
of~\eqref{eq:galerkin} with respect to the Scott-Zhang projection, i.e.\
$\widetilde U_\ell = \sz_\ell g$ on $\Gamma_D$. Finally and as above,
$\widetilde\eta_\ell^{\,2} = \widetilde\varrho_\ell^{\,2} + \oscD\ell^2$ denotes the error estimator for this auxiliary problem.
Although the discrete local reliability of $\widetilde\eta_\ell$ does not imply discrete local reliability of the error estimator $\eta_\ell$ for the primal problem, we will see that nevertheless discrete local reliability of an equivalent error estimator is sufficient for quasi-optimality.

\subsection{Optimality of D\"orfler marking}
Throughout, we assume that the Scott-Zhang projections are chosen with respect
to the assumptions of Section~\ref{section:apriori:scottzhang}.

\begin{proposition}[discrete local reliability for Scott-Zhang projection]
\label{prop:dlr}
Let $\TT_\star = \refine(\TT_\ell)$ be an arbitrary refinement of $\TT_\ell$ and $\widetilde U_\star\in\SS^p(\TT_\star)$ the corresponding Galerkin solution~\eqref{eq:galerkin} with $\widetilde U_\star = \sz_\star g$ on $\Gamma_D$.
Then, there is a set $\RR_\ell\subseteq\TT_\ell$ which contains the refined elements, $\TT_\ell\backslash\TT_\star\subseteq\RR_\ell$ such
that
\begin{align}\label{eq:dlr}
 \norm{\widetilde U_\star - \widetilde U_\ell}{H^1(\Omega)}
 \le \c{dlr}\,\sum_{T\in\RR_\ell}\widetilde\eta_\ell(T)^2
 \quad\text{and}\quad \#\RR_\ell \le \c{refined}\,\#(\TT_\ell\backslash\TT_\star).
\end{align}
The constants $\c{dlr},\c{refined}>0$ depend only on $\TT_0$ and the use of newest vertex bisection.
\end{proposition}

\begin{proof}
We consider a discrete auxiliary problem
\begin{align*}
 \dual{\nabla W_\star}{\nabla V_\star}_\Omega = 0
 \quad\text{for all }V_\star \in \SS^p_D(\TT_\star)
\end{align*}
with unique solution $W_\star \in \SS^p(\TT_\star)$ with $W_\star|_{\Gamma_D} = (\sz_\star-\sz_\ell)g$. Then, $(\widetilde U_\star - \widetilde U_\ell - W_\star)\in\SS^p_D(\TT_\star)$, and the $H^1$-norm is bounded by the $H^1$-seminorm. Moreover, arguing as in~\cite[Lemma~3.6]{ckns}, we see
\begin{align*}
 \norm{\widetilde U_\star - \widetilde U_\ell - W_\star}{H^1(\Omega)}^2
 \lesssim \norm{\nabla(\widetilde U_\star - \widetilde U_\ell - W_\star)}{L^2(\Omega)}^2
 \lesssim \sum_{T\in\TT_\ell\backslash\TT_\star}\widetilde\varrho_\ell(T)^2
 \le \sum_{T\in\TT_\ell\backslash\TT_\star}\widetilde\eta_\ell(T)^2.
\end{align*}
According to the triangle inequality, it thus only remains to bound
$\norm{W_\star}{H^1(\Omega)}$ by $\sum_{T\in\RR_\ell}\widetilde\eta_\ell(T)^2$ with some appropriate $\RR_\ell\supseteq\TT_\ell\backslash\TT_\star$. To that end,
let $\LL:H^{1/2}(\Gamma)\to H^1(\Omega)$ be a lifting operator and $\widehat g\in H^{1/2}(\Gamma)$ an arbitrary extension of $(\sz_\star-\sz_\ell)g\in H^{1/2}(\Gamma_D)$. With $V_\star := W_\star - \sz_\star\LL\widehat g\in\SS^p_D(\TT_\star)$, we obtain
\begin{align*}
 \norm{W_\star}{L^2(\Omega)}
 \le \norm{V_\star}{L^2(\Omega)}
  + \norm{\sz_\star\LL\widehat g}{L^2(\Omega)}
 &\lesssim \norm{\nabla V_\star}{L^2(\Omega)}
  + \norm{\sz_\star\LL\widehat g}{L^2(\Omega)}
 \\&
 \lesssim \norm{\nabla W_\star}{L^2(\Omega)}
  + \norm{\sz_\star\LL\widehat g}{H^1(\Omega)}.
\end{align*}
Moreover, the variational formulation for $W_\star\in\SS^p(\TT_\star)$ yields
\begin{align*}
 0 \!=\! \dual{\nabla W_\star}{\nabla V_\star}_\Omega
 \!=\! \norm{\nabla W_\star}{L^2(\Omega)}^2
 \!-\! \dual{\nabla W_\star}{\nabla\sz_\star\LL\widehat g}_\Omega,
\text{ whence }
 \norm{\nabla W_\star}{L^2(\Omega)}
 \le \norm{\nabla\sz_\star\LL\widehat g}{L^2(\Omega)}.
\end{align*}
Combining the last two estimates,
we obtain
\begin{align*}
 \norm{W_\star}{H^1(\Omega)}
 \lesssim \norm{\sz_\star\LL\widehat g}{H^1(\Omega)}
 \lesssim \norm{\widehat g}{H^{1/2}(\Gamma)}.
\end{align*}
Since $\widehat g$ was an arbitrary extension, this proves
\begin{align*}
 \norm{W_\star}{H^1(\Omega)} \lesssim \norm{(\sz_\star-\sz_\ell)g}{H^{1/2}(\Gamma_D)}.
\end{align*}
To abbreviate the notation in the remainder of the proof, let $\RR_\ell^D := \EE_\ell^D\backslash\EE_\ell^\star$ denote the refined Dirichlet facets. We define inductively
\begin{align*}
 \omega_\ell^0 = \mcup\RR_\ell^D,
 \quad
 \omega_\ell^n = \mcup\set{E\in\EE_\ell^D}{E\cap\omega_\ell^{n-1}\neq\emptyset}
 \quad\text{for }n\ge1,
\end{align*}
i.e.\ $\omega_\ell^n$ denotes the region of the refined Dirichlet facets plus $n$ layers of (non-refined) Dirichlet facets with respect to $\EE_\ell^D$.
Note that $\omega_{\ell}^1$ is nothing but the usual patch of $\RR_\ell^D$. Due to the local definition of $\sz_\ell$ and $\sz_\star$, we observe
\begin{align}\label{eq:szdlr}
(\sz_\star-\sz_\ell)g = 0 \quad \text{on }\Gamma_D\setminus \omega_\ell^1.
\end{align}
Let $\zeta_{\ell,z}\in \SS^1(\EE_\ell^D)$ denote the hat function associated with some node $z\in\KK_\ell^D$ of $\EE_\ell^D$. Clearly, the hat functions $\set{\zeta_{\ell,z}}{z\in\KK_\ell^D}$ provide a partition of unity $\sum_{z\in\KK_\ell^D} \zeta_{\ell,z} = 1 $ on $\Gamma_D$ resp.
$\sum_{z\in\KK_\ell^D\cap\omega_\ell^1} \zeta_{\ell,z} = 1 $ on $\omega_\ell^1$.
Exploiting~\eqref{eq:szdlr}, we see
\begin{align}\label{eq1:szdlr}
\begin{split}
\norm{(\sz_\star-\sz_\ell)g}{H^{1/2}(\Gamma_D)}& 
=\norm[\Big]{\sum_{z\in\KK_\ell^D\cap\omega_\ell^1} \zeta_{\ell,z}(\sz_\star-\sz_\ell)g}{H^{1/2}(\Gamma_D)} .
\end{split}
\end{align}
We now adapt the arguments of~\cite{cms,fkmp} to our setting.
Analogously to the proof of~\cite[Theorem 3.2]{cms}
resp.~\cite[Proposition~4.3]{fkmp}, we obtain
\begin{align*}
\begin{split}
\norm[\Big]{\sum_{z\in\KK_\ell^D\cap\omega_\ell^1} &\zeta_{\ell,z}(\sz_\star-\sz_\ell)g}{H^{1/2}(\Gamma_D)}^2
 \lesssim
\sum_{z\in\KK_\ell^D\cap\omega_\ell^1} \norm{\zeta_{\ell,z}(\sz_\star-\sz_\ell)g}{H^{1/2}(\Gamma_D)}^2\\
&\lesssim
\sum_{z\in\KK_\ell^D\cap\omega_\ell^1} \norm{\zeta_{\ell,z}(\sz_\star-\sz_\ell)g}{L^2(\Gamma_D)} \norm{\zeta_{\ell,z}(\sz_\star-\sz_\ell)g}{H^1(\Gamma_D)},
\end{split}
\end{align*}
where the final estimate is just the interpolation estimate.
As above, let
\begin{align*}
 \omega_{\ell,z}^0 := \{z\},
 \quad
 \omega_{\ell,z}^n := \mcup\set{E\in\EE_\ell^D}{E\cap\omega_{\ell,z}^{n-1}\neq\emptyset}
 \quad\text{for }n\ge1,
\end{align*}
i.e.\ $\omega_{\ell,z}^1 = \mcup\set{E\in\EE_\ell^D}{z\in E}$ denotes the
node patch of $z\in\KK_\ell^D$ which is just the support of the hat function $\zeta_{\ell,z}$ on $\Gamma_D$.
To proceed, we apply the Friedrichs inequality to the summands on the right-hand side of the estimate above and derive
 \begin{align}\label{eq3:szdlr}
\begin{split}
\norm[\Big]{\sum_{z\in\KK_\ell^D\cap\omega_\ell^1} \zeta_{\ell,z}(\sz_\star-\sz_\ell)g}{H^{1/2}(\Gamma_D)}^2
 &\lesssim
\sum_{z\in\KK_\ell^D\cap\omega_\ell^1} \diam(\omega_{\ell,z}^1)\norm{\nabla_\Gamma\big(\zeta_{\ell,z}(\sz_\star-\sz_\ell)g\big)}{L^2(\omega_{\ell,z}^1)}^2\\
&\simeq
\sum_{z\in\KK_\ell^D\cap\omega_\ell^1} \norm{h_\ell^{1/2}\nabla_\Gamma\big(\zeta_{\ell,z}(\sz_\star-\sz_\ell)g\big)}{L^2(\omega_{\ell,z}^1)}^2.
\end{split}
\end{align}
Here, $h_\ell\in L^\infty(\Gamma_D)$ denotes the local mesh-width function $h_\ell|_E = |T|^{1/d}$ for $E\in\EE_\ell^D$ and $T\in\TT_\ell$ the unique element with $E\subset\partial T$.
Formally, the constants in the Friedrichs inequality depend on the shape of
$\omega_{\ell,z}^1$.
Note, however, that there are only finitely many shapes of patches due to the use of newest vertex bisection. Next, we use the estimate
$|\nabla_\Gamma\zeta_{\ell,z}|_E| \simeq \diam(E)^{-1} \simeq h_\ell^{-1}|_E$ for $E\in\EE_\ell^D$. This and the product rule yield
\begin{align*}
\begin{split}
\norm{h_\ell^{1/2}&\nabla_\Gamma\big(\zeta_{\ell,z}(\sz_\star-\sz_\ell)g\big)}{L^2(\omega_{\ell,z}^1)}^2
\\&
\leq
\norm{h_\ell^{1/2}(\nabla_\Gamma\zeta_{\ell,z})(\sz_\star-\sz_\ell)g}{L^2(\omega_{\ell,z}^1)}^2
+\norm{h_\ell^{1/2}\nabla_\Gamma(\sz_\star-\sz_\ell)g}{L^2(\omega_{\ell,z}^1)}^2\\
&\lesssim
\norm{h_\ell^{-1/2}(\sz_\star-\sz_\ell)g}{L^2(\omega_{\ell,z}^1)}^2+
\norm{h_\ell^{1/2}\nabla_\Gamma(\sz_\star-\sz_\ell)g}{L^2(\omega_{\ell,z}^1)}^2
\\&
= \norm{h_\ell^{-1/2}\sz_\star(1-\sz_\ell)g}{L^2(\omega_{\ell,z}^1)}^2+
\norm{h_\ell^{1/2}\nabla_\Gamma\sz_\star(1-\sz_\ell)g}{L^2(\omega_{\ell,z}^1)}^2.
\end{split}
\end{align*}
Finally, the local stability of $\sz_\star$ and the local approximation property of $\sz_\ell$ yield
\begin{align}\label{eq5:szdlr}
\begin{split}
\norm{h_\ell^{1/2}\nabla_\Gamma\big(\zeta_{\ell,z}(\sz_\star-\sz_\ell)g\big)}{L^2(\omega_{\ell,z}^1)}^2
&\lesssim
\norm{h_\ell^{-1/2}(1-\sz_\ell)g}{L^2(\omega_{\ell,z}^2)}^2+
\norm{h_\ell^{1/2}\nabla_\Gamma(1-\sz_\ell)g}{L^2(\omega_{\ell,z}^2)}^2\\
&\lesssim
\norm{h_\ell^{1/2}\nabla_\Gamma(1-\sz_\ell)g}{L^2(\omega_{\ell,z}^3)}^2\\
&\lesssim
\norm{h_\ell^{1/2}(1-\Pi_\ell)\nabla_\Gamma g}{L^2(\omega_{\ell,z}^4)}^2,
\end{split}
\end{align}
where we have finally used Estimate~\eqref{eq:dirichlet2} of Proposition~\ref{lemma:oscD}. Now, let $\widetilde\RR_\ell^D := \set{E\in\EE_\ell^D}{E\subseteq\omega_\ell^5}$ denote the set of Dirichlet facets which lie in $\omega_\ell^5$ and note that $\#\widetilde\RR_\ell^D \simeq\#\RR_\ell^D\lesssim\#(\TT_\ell\backslash\TT_\star)$ up to shape regularity. The combination of~\eqref{eq1:szdlr}--\eqref{eq5:szdlr} yields
\begin{align*}
 \norm{W_\star}{H^1(\Omega)}^2
 \lesssim \norm{(\sz_\star-\sz_\ell)g}{H^{1/2}(\Gamma_D)}^2
 &\lesssim \sum_{z\in\KK_\ell^D\cap\omega_\ell^1}
 \norm{h_\ell^{1/2}(1-\Pi_\ell)\nabla_\Gamma g}{L^2(\omega_{\ell,z}^4)}^2
 \\&
 \simeq \norm{h_\ell^{1/2}(1-\Pi_\ell)\nabla_\Gamma g}{L^2(\omega_{\ell}^5)}^2
 \\&
 \lesssim \sum_{E\in\widetilde\RR_\ell^D}\oscD\ell(E)^2,
\end{align*}
due to~\eqref{eq:dirichlet2}--\eqref{eq:dirichlet} in Proposition~\ref{lemma:oscD}
Defining the set
\begin{align*}
\RR_\ell := \TT_\ell\backslash\TT_\star \cup \set{T\in\TT_\ell}{\exists E\in\widetilde\RR_\ell^D\quad E\subset\partial T},
\end{align*}
we observe $\TT_\ell\backslash\TT_\star \subseteq \RR_\ell$ as well as $\#\RR_\ell \lesssim \#(\TT_\ell\backslash\TT_\star)$. Moreover, the definition of the local contributions of $\widetilde\eta_\ell$ in~\eqref{eq:eta:local} shows
\begin{align*}
 \sum_{E\in\widetilde\RR_\ell^D}\oscD\ell(E)^2
 \le \sum_{T\in\RR_\ell}\widetilde\eta_\ell(T)^2.
\end{align*}
This concludes the proof.
\end{proof}

\begin{corollary}[optimality of D\"orfler marking for Scott-Zhang projection]
\label{prop:doerfler}%
For arbitrary $0<\kappa_\star<1$, there is a constant $0<\theta_\star<1$
 such that for all $\ell\in\N_0$ and all meshes $\TT_\star = \refine(\TT_\ell)$ with
$\widetilde\eta_\star^{\,2} \le \kappa_\star\,\widetilde\eta_\ell^{\,2}$, the set $\RR_\ell\subseteq\TT_\ell$ from Proposition~\ref{prop:dlr} satisfies 
the D\"orfler marking
\begin{align}
 \theta\,\widetilde\eta_\ell^{\,2}
 \le \sum_{T\in\RR_\ell}\widetilde\eta_\ell(T)^2
\end{align}
for all $0<\theta\le\theta_\star$.
\end{corollary}

\begin{proof}
We split the estimator into the contributions on the non-refined resp.\ refined elements
\begin{align*}
  \widetilde\eta_\ell^{\,2}
  = \sum_{T\in\TT_\ell}\widetilde\eta_\ell(T)^2
  = \sum_{T\in\TT_\ell\cap\TT_\star}\widetilde\eta_\ell(T)^2
  + \sum_{T\in\TT_\ell\backslash\TT_\star}\widetilde\eta_\ell(T)^2.
\end{align*}
Arguing as for the estimator reduction
in Proposition~\ref{prop:reduction} with $\delta = 1$, we see
\begin{align*}
 \sum_{T\in\TT_\ell\cap\TT_\star}\widetilde\eta_\ell(T)^2
 \le 2\sum_{T\in\TT_\ell\cap\TT_\star}\widetilde\eta_\star(T)^2
 + \c{reduction}\,\norm{\nabla(\widetilde U_\star - \widetilde U_\ell)}{L^2(\Omega)}^2
\le 2\widetilde\eta_\star^{\,2} + \c{reduction}\,\norm{\nabla(\widetilde U_\star - \widetilde U_\ell)}{L^2(\Omega)}^2.
\end{align*}
We now combine both estimates and use $\widetilde\eta_\star^{\,2} \le \kappa_\star\,\widetilde\eta_\ell^{\,2}$ as well as the discrete local reliability with $\RR_\ell\supseteq\TT_\ell\backslash\TT_\star$ to see
\begin{align*}
 \widetilde\eta_\ell^{\,2}
 \le 2\widetilde\eta_\star^{\,2} + \c{reduction}\,\norm{\nabla(\widetilde U_\star - \widetilde U_\ell)}{L^2(\Omega)}^2
 + \sum_{T\in\TT_\ell\backslash\TT_\star}\widetilde\eta_\ell(T)^2
 \le 2\,\kappa_\star\,\widetilde\eta_\ell^{\,2} + (\c{reduction}\c{dlr}+1)\,\sum_{T\in\RR_\ell}
 \widetilde\eta_\ell(T)^2.
\end{align*}
Rearranging this estimate, we obtain
\begin{align*}
 (\c{reduction}\c{dlr}+1)^{-1}(1-2\kappa_\star)\, \widetilde\eta_\ell^{\,2}
 \le \sum_{T\in\RR_\ell}\widetilde\eta_\ell(T)^2,
\end{align*}
so that $0<\theta_\star := (\c{reduction}\c{dlr}+1)^{-1}(1-2\kappa_\star)<1$ concludes the proof.
\end{proof}

\subsection{Optimality of newest vertex bisection}
The quasi-optimality analysis of AFEM requires two properties of the mesh-refinement which are satisfied for newest vertex bisection: First, for two triangulations $\TT',\TT''\in\T$, let $\TT'\oplus\TT''\in\T$ be the coarsest common refinement of both. Since newest vertex bisection is a binary refinement rule, it can be proved that $\TT'\oplus\TT''$ is just the overlay of both meshes, see~\cite[Proof of Lemma~5.2]{stevenson} for 2D and the generalization to arbitrary dimension in~\cite[Lemma~3.7]{ckns}. Moreover, the number of elements of the overlay is controlled by
\begin{align}\label{eq:overlay}
 \#(\TT'\oplus\TT'') \le \#\TT' + \#\TT'' - \#\TT_0,
\end{align}
since both meshes are generated from the initial mesh $\TT_0$.

Second, we need the optimality of the mesh-closure, i.e.\ the definition $\TT_{\ell+1} = \refine(\TT_\ell,\MM_\ell)$ leads at least to refinement of all marked elements $T\in\MM_\ell$. In addition, further elements $T\in\TT_\ell\backslash\MM_\ell$ have to be refined to ensure conformity of the mesh. It has been proved in~\cite[Theorem~2.4]{bdd} for 2D that
\begin{align}\label{eq:closure}
 \#\TT_{\ell+1} - \#\TT_0 \le \c{nvb}\,\sum_{j=0}^{\ell}\#\MM_j
 \quad\text{for all }\ell\ge0,
\end{align}
i.e.\ the number of elements in $\TT_{\ell+1}$ is bounded by the number of marked elements.
The constant $\setc{nvb}>0$ depends only on $\TT_0$ in the sense that the
initial reference edge distribution had to satisfy a certain assumption. Very recently~\cite{kp}, it could be proved that~\eqref{eq:closure} holds without any further assumptions on $\TT_0$. For arbitrary dimension,~\eqref{eq:closure} has been proved in~\cite[Theorem~6.1]{stevenson:nvb} and $\TT_0$ has to satisfy a certain assumption on the initial reference edge distribution.

\subsection{Proof of quasi-optimality of AFEM (Theorem~\ref{theorem:quasioptimal})}
In a first step, we prove that $\#\MM_\ell \lesssim \Delta_\ell^{-1/(2s)}$. To that end, let $\eps>0$ be a free parameter which is determined later. According to the definition of the approximation class $\A_s$, there is some triangulation $\TT_\eps\in\T$ with
\begin{align*}
 \eta_\eps \le \eps
 \quad\text{and}\quad \#\TT_\eps - \#\TT_0 \lesssim \eps^{-1/s},
\end{align*}
where the hidden constant depends only on $\A_s$. We consider the overlay
$\TT_\star := \TT_\eps \oplus \TT_\ell$. Arguing as for the estimator
reduction~\eqref{eq:reduction} and use of the discrete local reliability for
$\widetilde\eta_\ell$, we obtain
\begin{align*}
 \widetilde\eta_\star \lesssim
 \widetilde\eta_\ell
 + \norm{\nabla(\widetilde U_\star - \widetilde U_\ell)}{L^2(\Omega)}
 \lesssim \widetilde \eta_\ell \simeq \eta_\ell \le \eps,
\end{align*}
where we have finally used the equivalence of both error estimators provided
by Lemma~\ref{lemma:equivalence}. Choosing $\eps = \delta\,\eta_\ell \simeq \delta\,\widetilde\eta_\ell$ with
 sufficiently small $\delta>0$, we thus infer
\begin{align*}
 \widetilde\eta_\star \le \widetilde\kappa_\star\widetilde\eta_\ell
\end{align*}
with some appropriate $0<\widetilde\kappa_\star\le\kappa_\star$, where arbitrary $0<\kappa_\star<1$ in Proposition~\ref{prop:doerfler} fixes $0<\theta_\star<1$. The constant $\widetilde\kappa_\star$ will be determined later. Together with the
overlay estimate~\eqref{eq:overlay}, we infer
\begin{align*}
 \#\RR_\ell \simeq \#(\TT_\ell\backslash\TT_\star)
 \le \#\TT_\star - \#\TT_\ell \le \#\TT_\eps - \#\TT_0 \lesssim \eps^{-1/s}
\end{align*}
as well as the D\"orfler estimate
\begin{align*}
 \theta_\star \widetilde\eta_\ell^{\,2}
 \le \sum_{T\in\RR_\ell} \widetilde\eta_\ell(T)^2.
\end{align*}
We now need to show that this implies Stevenson's modified D\"orfler marking.
To that end, we again employ Lemma~\ref{lemma:equivalence}:

$\bullet$ In case of $\oscD\ell^2 \le \vartheta \, \varrho_\ell^2$, we employ Lemma~\ref{lemma:equivalence} twice to see
\begin{align*}
 \theta_\star \varrho_\ell^2
 \lesssim \theta_\star\widetilde\eta_\ell^{\,2}
 \le \sum_{T\in\RR_\ell}\widetilde\eta_\ell(T)^2 + \oscD\ell^2
 \lesssim \sum_{T\in\RR_\ell}\varrho_\ell(T)^2 + \oscD\ell^2
 \le \sum_{T\in\RR_\ell}\varrho_\ell(T)^2 + \vartheta\varrho_\ell^2.
\end{align*}
Put differently, we obtain
\begin{align*}
 \big((\c{equivalence}+1)^{-2}\theta_\star-\vartheta\big)\,\varrho_\ell^2
 \le\sum_{T\in\RR_\ell}\varrho_\ell(T)^2,
\end{align*}
i.e.\ for $0<\vartheta,\theta_1<1$ sufficiently small, the set $\RR_\ell\subseteq\TT_\ell$ satisfies
the marking criterion~\eqref{eq:doerfler:res}.

$\bullet$ In case of $\oscD\ell^2 > \vartheta\,\varrho_\ell^2$, we use that the Dirichlet oscillations are locally determined, i.e.
\begin{align*}
 \sum_{E\in\EE_\ell^D\cap\EE_\star^D}\!\! \oscD\ell(E)^2
 = \!\!\sum_{E\in\EE_\ell^D\cap\EE_\star^D}\!\! \oscD\star(E)^2
 \le \oscD\star^2
 \le \widetilde\eta_\star^{\,2}
 \le \widetilde\kappa_\star \widetilde\eta_\ell^{\,2}
 \simeq \widetilde\kappa_\star \eta_\ell^2
 \le \widetilde\kappa_\star(1+\vartheta^{-1})\, \oscD\ell^2.
\end{align*}
This estimate yields
\begin{align*}
 \big(1-\widetilde \kappa_\star(1+\vartheta^{-1})(\c{equivalence} + 1)\big)
 \,\oscD\ell^2
 \le \sum_{E\in\EE_\ell^D\backslash\EE_\star^D}\oscD\ell(E)^2
\end{align*}
For arbitrary $0<\theta_2<1$ and sufficiently small $0<\widetilde\kappa_\star<1$, we infer that $\EE_\ell^D\backslash\EE_\star^D$ satisfies the marking criterion~\eqref{eq:doerfler:osc}.

In the first case, minimal cardinality of $\MM_\ell\subseteq\TT_\ell$ in step~(iii) of Algorithm~\ref{algorithm2} implies $\#\MM_\ell\le\#\RR_\ell\simeq\#(\TT_\ell\backslash\TT_\star)$. In the second case, minimal cardinality of $\MM_\ell^D\subseteq\EE_\ell^D$ and the definition of $\MM_\ell\subseteq\TT_\ell$ in step~(iv) of Algorithm~\ref{algorithm2} imply $\#\MM_\ell\le\#\MM_\ell^D\le\#(\EE_\ell^D\backslash\EE_\star^D) \lesssim \#(\TT_\ell\backslash\TT_\star)$. In either case, we thus conclude
\begin{align*}
 \#\MM_\ell \lesssim \#(\TT_\ell\backslash\TT_\star)
 \lesssim\eps^{-1/s} \simeq \eta_\ell^{-1/s} \simeq \Delta_\ell^{-1/(2s)}.
 \quad\text{for all }\ell\ge0.
\end{align*}

We now conclude the proof as e.g.\ in~\cite{stevenson,ckns}: By use of the
closure estimate~\eqref{eq:closure}, we obtain
\begin{align*}
 \#\TT_\ell - \#\TT_0
 \lesssim \sum_{j=0}^{\ell-1}\#\MM_j
 \lesssim \sum_{j=0}^{\ell-1}\Delta_j^{-1/(2s)}.
\end{align*}
Note that the contraction property~\eqref{eq:contraction} of $\Delta_j$ implies $\Delta_\ell \le \kappa^{\ell-j}\,\Delta_j$, whence $\Delta_j^{-1/(2s)}\le\kappa^{(\ell-j)/(2s)}\,\Delta_\ell^{-1/(2s)}$.
According to $0<\kappa<1$ and the geometric series, this gives
\begin{align*}
 \#\TT_\ell - \#\TT_0
 \lesssim \Delta_\ell^{-1/(2s)}\,\sum_{j=0}^{\ell-1}\kappa^{(\ell-j)/(2s)}
 \lesssim \Delta_\ell^{-1/(2s)}
 \simeq \eta_\ell^{-1/s}.
\end{align*}
Altogether, we may therefore conclude that $(u,f,g,\phi)\in\A_s$ implies
$\eta_\ell \lesssim (\#\TT_\ell - \#\TT_0)^{-s}$ for all $\ell\ge0$. The converse implication is obvious by definition of $\A_s$.
\qed

\subsection{Characterization of approximation class (Theorem~\ref{theorem:approximationclass})}
First, note that for a given mesh $\TT_\star \in \T$ the estimator $\eta_\star$ dominates all oscillation terms, i.e.
\begin{align*}
\oscT\star \leq \eta_\star,\quad \oscD\star \leq \eta_\star , \quad \oscN\star \leq \eta_\star.
\end{align*}
We assume $(u,f,g,\phi) \in \A_s$ for some $s>0$. For each $N\in\N$ it exists $\TT_\star \in \T_N$ such that
\begin{align}\label{eq1:Acharacterization}
N^s\oscT\star \leq N^s\eta_\star \leq C:= \sup_{N\in\N} \inf_{\TT_\star\in\T_N} N^s\eta_\star <\infty.
\end{align}
Analogously, we have
\begin{align}\label{eq2:Acharacterization}
N^s\oscN\star \leq C<\infty \quad\text{and} \quad N^s\oscD\star \leq C<\infty.
\end{align}
The reliability result in Proposition~\ref{prop:estimator} yields
\begin{align}\label{eq3:Acharacterization}
 \min_{V_\star\in\SS^p(\TT_\star)} N^s\norm{u-V_\star}{H^1(\Omega)} \leq N^s\norm{u-U_\star}{H^1(\Omega)}\leq \c{rho:reliable}N^s\eta_\star \leq \c{rho:reliable} C <\infty.
\end{align}
Because $N\in\N$ was arbitrary, the estimates~\eqref{eq1:Acharacterization}--\eqref{eq3:Acharacterization} prove~\eqref{eq:Achar1}--\eqref{eq:Achar4}.

Now, we assume that~\eqref{eq:Achar1}--\eqref{eq:Achar4} hold for $(u,f,g,\phi)$. We aim to prove $(u,f,g,\phi)\in \A_s$. By use of the efficiency estimate in
Proposition~\ref{prop:estimator} and the C\'ea-type estimate in Proposition~\ref{prop:cea}, we derive
\begin{align}\label{eq4:Acharacterization}
\begin{split}
 \sup_{N\in\N} \inf_{\TT_\star\in\T_N} N^s\eta_\star \leq \c{rho:efficient}\sup_{N\in\N} \inf_{\TT_\star\in\T_N}
N^s \Big( \c{cea}&\min_{V_\star\in\SS^p(\TT_\star)}\norm{u-V_\star}{H^1(\Omega)}^2\\
& + \oscT{\star}^2 +\oscN{\star}^2 + \oscD{\star}^2\Big).
\end{split}
\end{align}
For $N\in\N$, the assumption~\eqref{eq:Achar1}--\eqref{eq:Achar4} guarantee meshes $\TT_{\star_u}$, $\TT_{\star_\Omega}$, $\TT_{\star_N}$, $\TT_{\star_D} \in \T_{N/4}$ such that
\begin{align*}
(N/4)^s \min_{V_{\star_u}\in\SS^p(\TT_{\star_u})}\norm{u-V_{\star_u}}{H^1(\Omega)} &\leq \sup_{N>0}\inf_{\TT_\star\in\T_N}
 \min_{V_\star\in\SS^p(\TT_\star)} N^s\norm{u-V_\star}{H^1(\Omega)} =:C_u< \infty,\\
 (N/4)^s\,\oscT{\star_\Omega}&\leq\sup_{N>0}\inf_{\TT_\star\in\T_N} N^s\oscT\star =:C_{{\rm osc}_\TT} < \infty,\\
 (N/4)^s\,\oscN{\star_N}&\leq\sup_{N>0}\inf_{\TT_\star\in\T_N} N^s\oscN\star=:C_{{\rm osc}_N} < \infty,\\
 (N/4)^s\,\oscD{\star_D}&\leq\sup_{N>0}\inf_{\TT_\star\in\T_N} N^s\oscD\star =:C_{{\rm osc}_D}< \infty.
\end{align*}
Now, we consider the overlay $\TT_{*} := \TT_{\star_u} \oplus \TT_{\star_\Omega} \oplus \TT_{\star_N} \oplus \TT_{\star_D}$. The overlay estimate~\eqref{eq:overlay} gives $\#\TT_* \le N - 3\#\TT_0$, whence $\#\TT_*-\#\TT_0\leq N$. Due to the fact that $\Pi_E$ and $\Pi_T$ are projections, we get immediately by definition of the oscillation terms and $\SS^p(\TT_*)\supseteq \SS^p(\TT_{\star_u})$
\begin{align*}
 \min_{V_*\in\SS^p(\TT_*)}\norm{u-V_*}{H^1(\Omega)} \leq \min_{V_{\star_u}\in\SS^p(\TT_{\star_u})}\norm{u-V_{\star_u}}{H^1(\Omega)}, \\
 \oscT{*}\leq \oscT{\star_\Omega},\quad
 \oscN{*} \leq \oscN{\star_N}, \quad\text{and}\quad
 \oscD{*} \leq \oscD{\star_D}.
\end{align*}
Together with~\eqref{eq4:Acharacterization}, we prove
\begin{align*}
\inf_{\TT_\star\in\T_N} N^s\eta_\star& 
\leq \c{rho:efficient}
N^s \Big( \c{cea}\min_{V_*\in\SS^p(\TT_*)}\norm{u-V_*}{H^1(\Omega)}^2
+ \oscT{*}^2 +\oscN{*}^2 + \oscD{*}^2\Big)\\
&\leq \c{rho:efficient}
N^s \Big( \c{cea}\min_{V_{\star_u}\in\SS^p(\TT_{\star_u})}\norm{u-V_{\star_u}}{H^1(\Omega)}^2
+ \oscT{\star_T}^2 +\oscN{\star_N}^2 + \oscD{\star_D}^2\Big)\\
&\leq \c{rho:efficient}4^{-s}(\c{cea}C_u + C_{{\rm osc}_\TT} + C_{{\rm osc}_N} + C_{{\rm osc}_D} ) <\infty,
\end{align*}%
where the constants are independent of $N\in\N$. Taking the supremum over $N\in\N$, we conclude $(u,f,g,\phi)\in \A_s$.
\qed 

\section{Numerical Experiment}
\label{section:numerics}%

\begin{figure}
\begin{center}
%
%
%

	\pgfplotsset{width=14cm,height=9cm,compat=1.3} 
	\tikzstyle{every pin}=[fill=white,draw=black,font=\footnotesize,rounded corners=5pt]

\begin{tikzpicture}[trim axis left, trim axis right]
\begin{axis}[
	title = {Error estimator $\varrho_\ell$ for $L^2$-orthogonal projection in 2D},
	xmode=log,ymode=log,
	xtick={1e1,1e2,1e3,1e4,1e5,1e6},
	xlabel={number of elements $N = \#\mathcal{T}_\ell$},
	thick,
	mark size = 3,
	legend style={
		cells={anchor=west},
		rounded corners=7pt, 
	},
	]

	
	\addplot[forget plot, domain=2e3:2e6,color=black!50,mark=none,densely dashed] {x^(-2/7)*1.3e0};
 	\node[coordinate,pin=right:{$\mathcal{O}\big(N^{-2/7}\big)$}] at (axis cs:6e5,3e-2) {};

	\addplot[forget plot, domain=2e3:2e6,color=black!50,mark=none,densely dashed] {x^(-1/2)*3.5e0};
 	\node[coordinate,pin=left:{$\mathcal{O}\big(N^{-1/2}\big)$}] at (axis cs:1e5,1.1e-2) {};



	\addplot
		table {./figures/compare_orthogonal_2D/unif_L2.dat}; 
	\addlegendentry{uniform}

	\addplot
		table {./figures/compare_orthogonal_2D/025_L2.dat}; 
	\addlegendentry{$\theta = 1/4$}
			
	\addplot
		table {./figures/compare_orthogonal_2D/18_L2.dat}; 
	\addlegendentry{$\theta = 1/8$}	

	\addplot
		table {./figures/compare_orthogonal_2D/116_L2.dat}; 
	\addlegendentry{$\theta = 1/16$}

\end{axis}
\end{tikzpicture}

%
%
%
%
%
%
%
%
%
%
%
%
%
%
\vspace*{10mm}
%
%
%
	\pgfplotsset{width=14cm,height=9cm,compat=1.3} 
	\tikzstyle{every pin}=[fill=white,draw=black,font=\footnotesize,rounded corners=5pt]

\begin{tikzpicture}[trim axis left, trim axis right]
\begin{axis}[
	title={Error estimator $\varrho_\ell$ for Scott-Zhang projection in 2D},
	xmode=log,ymode=log,
	xtick={1e1,1e2,1e3,1e4,1e5,1e6},
	xlabel={number of elements $N = \#\mathcal{T}_\ell$},
		thick,
	mark size = 3,
	legend style={
		cells={anchor=west},
		rounded corners=7pt, 
	},
	]
	\addplot[forget plot, domain=2e3:2e6,color=black!50,mark=none,densely dashed] {x^(-2/7)*1.3e0};
 	\node[coordinate,pin=right:{$\mathcal{O}\big(N^{-2/7}\big)$}] at (axis cs:6e5,3e-2) {};

	\addplot[forget plot, domain=2e3:2e6,color=black!50,mark=none,densely dashed] {x^(-1/2)*3.5e0};
 	\node[coordinate,pin=left:{$\mathcal{O}\big(N^{-1/2}\big)$}] at (axis cs:1e5,1.1e-2) {};



	\addplot
		table {./figures/compare_scottzhang_2D/unif_SZ.dat}; 
	\addlegendentry{uniform}

	\addplot
		table {./figures/compare_scottzhang_2D/025_SZ.dat}; 
	\addlegendentry{$\theta = 1/4$}

	\addplot
		table {./figures/compare_scottzhang_2D/18_SZ.dat}; 
	\addlegendentry{$\theta = 1/8$}
		
	\addplot
		table {./figures/compare_scottzhang_2D/116_SZ.dat}; 
	\addlegendentry{$\theta = 1/16$}

\end{axis}
\end{tikzpicture}
%
%
%
%
%
%
%
%
%
%
%
%
%
%
%
\end{center}
\caption{Numerical results for $\varrho_\ell$ for uniform and adaptive mesh-refinement using $\theta \in \{1/4, 1/8, 1/16\}$ and $L^2$-orthogonal projection and Scott-Zhang projection in 2D, respectively, plotted over the number of elements $N = \# \TT_\ell$.}
\label{fig:compare_orthogonal_sz_2D}
\end{figure}

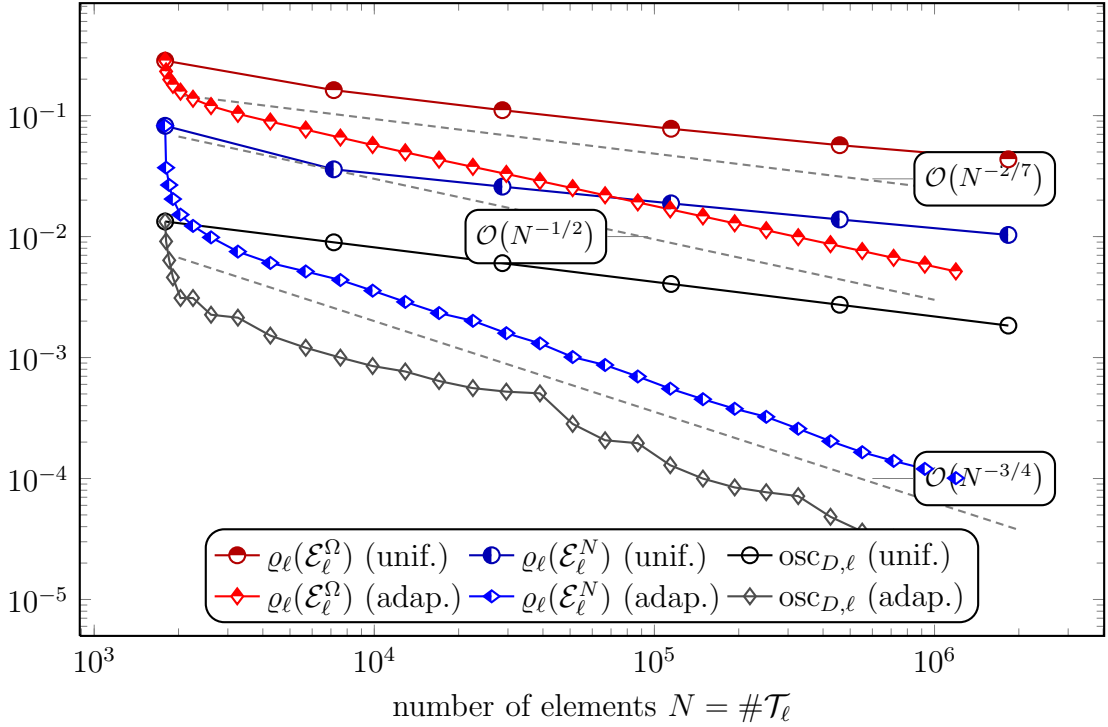
\begin{figure}
\begin{center}
%
%
%
%

	\pgfplotsset{width=15.2cm,height=10cm,compat=1.3} 
	\tikzstyle{every pin}=[fill=white,draw=black,font=\footnotesize,rounded corners=5pt]

\begin{tikzpicture}[trim axis left, trim axis right]
\begin{axis}[
	title=Individual contributions of $\varrho_\ell$ for Scott-Zhang projection in 2D,
	xmode=log,ymode=log,
	xtick={1e1,1e2,1e3,1e4,1e5,1e6},
	xlabel={number of elements $N = \#\mathcal{T}_\ell$},
	ymin = 5e-6,
	legend style={
		cells={anchor=west},
		anchor=north west, 
		legend columns=3, 
		rounded corners=7pt, 
	},
	legend style={at={(0.5,+0.02)},anchor=south}, 
	mark size= 3,
	thick,
	]

	
	\addplot[color=red!70!black,mark=halfcircle*] 
		table {./figures/compare_sep_2D/jumpsUnif.dat}; 
	\addlegendentry{$\varrho_{\ell}(\mathcal{E}_\ell^\Omega)$  (unif.)}

	\addplot[color=blue!70!black,mark=halfcircle*,every mark/.append style = {rotate=90}] 
		table {./figures/compare_sep_2D/neuUnif.dat}; 
	\addlegendentry{$\varrho_{\ell}(\mathcal{E}_\ell^N)$  (unif.)}
	
	\addplot[color=black,mark=o] 
		table {./figures/compare_sep_2D/UDellUnif.dat}; 
	\addlegendentry{$\textrm{osc}_{D,\ell}$ (unif.)}


	\addplot[color=red,mark=halfdiamond*,every mark/.append style = {rotate=180}] 
		table {./figures/compare_sep_2D/jumps_SZ.dat}; 
	\addlegendentry{$\varrho_{\ell}(\mathcal{E}_\ell^\Omega)$  (adap.)}
	
	\addplot[color=blue,mark=halfdiamond*,every mark/.append style = {rotate=-90}] 
		table {./figures/compare_sep_2D/neumann_SZ.dat}; 
	\addlegendentry{$\varrho_{\ell}(\mathcal{E}_\ell^N)$  (adap.)}

	\addplot[color=black!70,mark=diamond] 
		table {./figures/compare_sep_2D/UDell_SZ.dat}; 
	\addlegendentry{$\textrm{osc}_{D,\ell}$ (adap.)}

	\addplot[forget plot, domain=2e3:1e6,color=black!50,mark=none,densely dashed] {x^(-1/2)*30e-1};
 	\node[coordinate,pin=left:{$\mathcal{O}\big(N^{-1/2}\big)$}] at (axis cs:9.5e4,1e-2) {};
	
	\addplot[forget plot, domain=2e3:2e6,color=black!50,mark=none,densely dashed] {x^(-3/4)*2e0};
 	\node[coordinate,pin=right:{$\mathcal{O}\big(N^{-3/4}\big)$}] at (axis cs:6e5,1e-4) {};

	\addplot[forget plot, domain=2e3:2e6,color=black!50,mark=none,densely dashed] {x^(-2/7)*1.3e0};
 	\node[coordinate,pin=right:{$\mathcal{O}\big(N^{-2/7}\big)$}] at (axis cs:6e5,3e-2) {};

	
%

\end{axis}
\end{tikzpicture}
%
%
%
%
%
%
%
%
%
%
%
%
\end{center}
\caption{Numerical results for 
$\varrho_{\ell}(\EE_\ell^\Omega)$, $\varrho_{\ell}(\EE_\ell^N)$ and $\oscD{\ell}$
 for uniform and adaptive mesh-refinement using $\theta = 0.25$ and Scott-Zhang projection in 2D, plotted over the number of elements $N = \# \TT_\ell$.}
\label{fig:compare_sep_2D}
\end{figure}

\noindent In this section, we provide numerical results for mixed boundary value problems in two and three space dimensions for the lowest-order case $p = 1$. In both examples we choose $\vartheta = \theta_1 = \theta_2$ in Algorithm~\ref{algorithm2}. For comparison of the individual contributions $\varrho_\ell$, we further define the jump terms
\begin{align}
\varrho_{\ell}(\EE_\ell^\Omega)^2 := \sum_{E \in \EE_\ell^\Omega} |T|^{1/d}\norm{[\partial_n U_\ell]}{L^2(E)}^2,
\end{align}
the volume terms
\begin{align}
\varrho_{\ell}(\Omega)^2 := \sum_{T \in \TT_\ell} |T|^{2/d}\norm{f}{L^2(T)}^2
\end{align}
and the Neumann terms
\begin{align}
\varrho_{\ell}(\EE_\ell^N)^2 :=\sum_{E \in \EE_\ell^N}|T|^{1/d}\norm{\phi - \partial_n U_\ell}{L^2(E)}^2.
\end{align}
for the respective space dimension $d \in {2, 3}$.

\subsection{2D example on Z-shape}
\noindent
In our first example, we consider the Z-shaped domain $\Omega = (-1, 1)^2 \backslash \rm{conv}\{(0,0), (-1,-1), (0,-1)\}$, see Figure~\ref{fig:meshes_2D}, where also the partition of the boundary $\Gamma = \partial \Omega$ into Dirichlet boundary $\Gamma_D$ and Neumann boundary $\Gamma_N$ as well as the initial mesh is shown. We prescribe the exact solution
\begin{align}
 u(x) = r^{4/7}\cos(4\varphi/7)
\end{align}
of problem~\ref{eq:strongform} in polar coordinates $x = r(\cos \varphi, \sin \varphi)$ and compute the Neumann and Dirichlet data thereof. Note, that $f$ is harmonic so that
\begin{align*}
-\Delta u = f = 0.
\end{align*}
The solution $u$ as well as the Dirichlet data $g = u|_\Gamma$ show a generic singularity at the reentrant corner $r = 0$. 

Figure~\ref{fig:compare_orthogonal_sz_2D} shows a comparison between uniform and adaptive mesh refinement, where the adaptivity parameter $\theta$
varies between $1/4$ and $1/16$ and where the Dirichlet data are discretized by means of the $L^2$-projection and the Scott-Zhang projection, respectively. It is easily seen that both discretizations lead to the optimal convergence rate $\mathcal{O}(N^{-1/2})$
for all parameters $\theta$, whereas uniform refinement leads only to suboptimal convergence behaviour of
approximately $\mathcal{O}(N^{-2/7})$. Note that due to $f \equiv 0$, we have no volume contributions in this example.

In Figure~\ref{fig:compare_sep_2D}, we compare the jump terms, the Neumann terms, as well as the Dirichlet oscillations $\oscD\ell$ for uniform and adaptive refinement, where we have chosen the Scott-Zhang projection to discretize the boundary data.
Even here, we observe better convergence rates with adaptive refinement. Due to the corner singularity of the exact solution at $r = 0$, uniform refinement leads to a suboptimal convergence behaviour, even for the oscillations.
\begin{figure}
	\begin{center}
		\begin{minipage}{.45\textwidth}
			\includegraphics[width=1\textwidth]{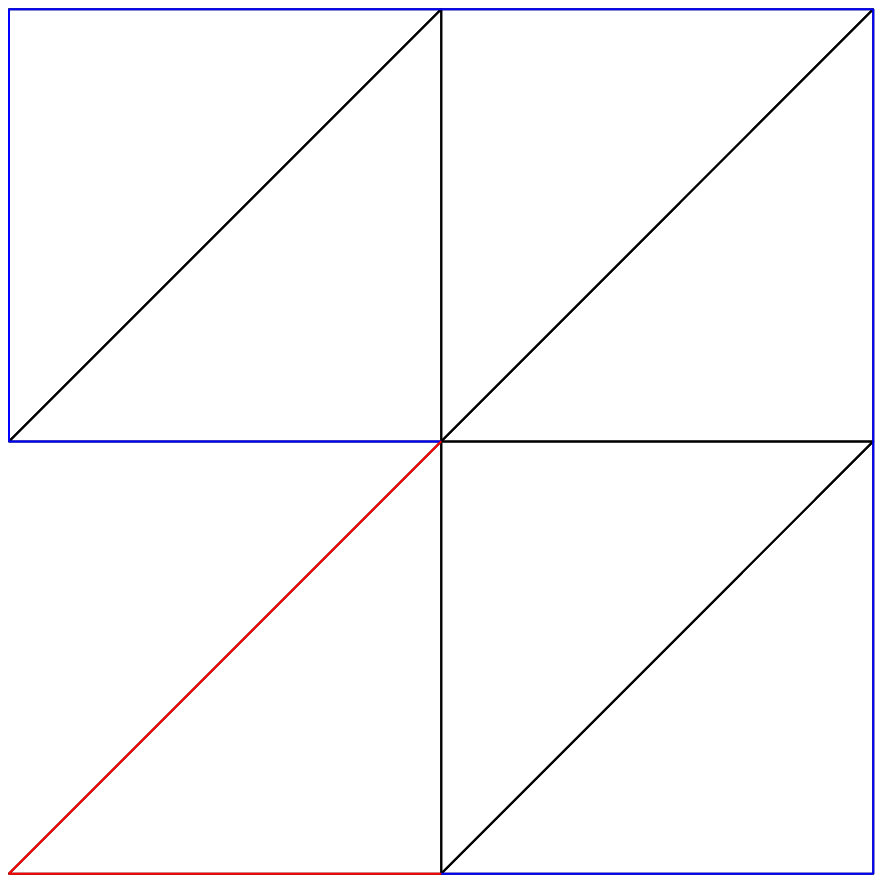}
		\end{minipage}
		\begin{minipage}{.45\textwidth}
			\includegraphics[width=1\textwidth]{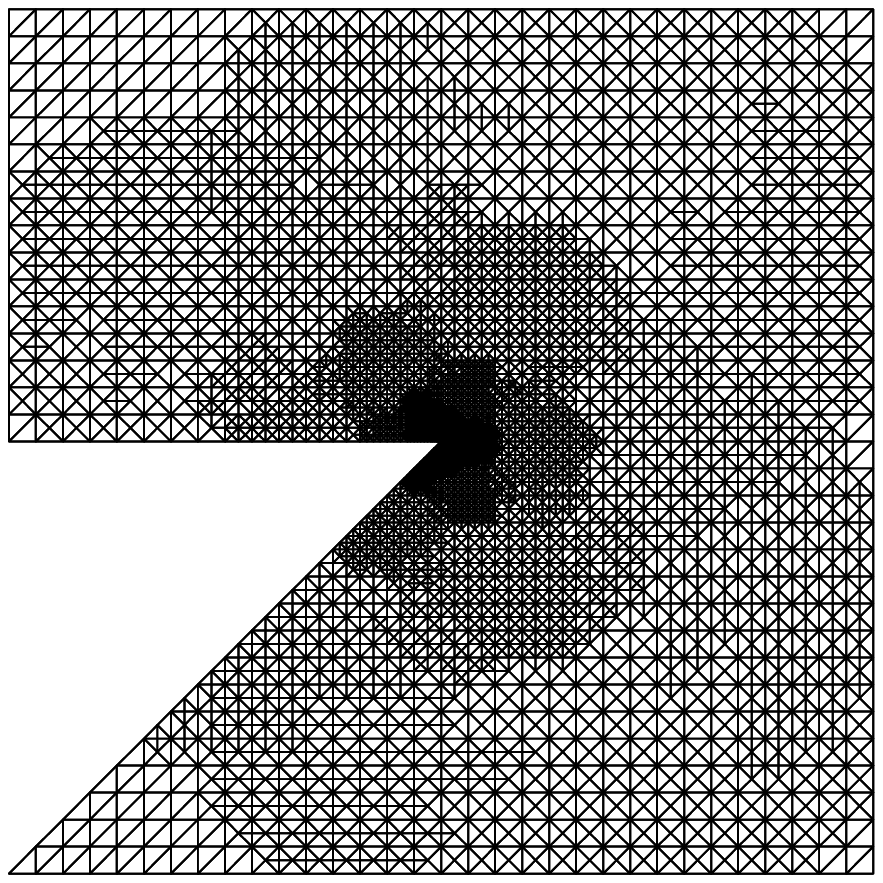}
		\end{minipage}
	\end{center}
	\caption{Z-shaped domain with initial mesh $\TT_0$ and adaptively generated mesh $\TT_{11}$ with $N=9.864$ for $\theta = 0.25$. The Dirichlet boundary $\Gamma_D$ is marked red, whereas the blue parts denote the Neumann boundary $\Gamma\backslash \Gamma_D$.
	}
	\label{fig:meshes_2D}
\end{figure}
Finally, in Figure~\ref{fig:meshes_2D}, the initial mesh $\TT_0$ and the adaptively generated mesh $\TT_{11}$ with $N=9.864$ Elements are visualized. As expected, adaptive refinement is essentially concentrated around the reentrant corner $r = 0$.

\begin{figure}
\begin{center}
%
%
%

	\pgfplotsset{width=14cm,height=9cm,compat=1.3} 
	\tikzstyle{every pin}=[fill=white,draw=black,font=\footnotesize,rounded corners=5pt]

\begin{tikzpicture}[trim axis left, trim axis right]
\begin{axis}[
	title = {Error estimator $\varrho_\ell$ for $L^2$-orthogonal projection in 3D},
	xmode=log,ymode=log,
	xmax = 5e5,
	xtick={1e1,1e2,1e3,1e4,1e5,1e6},
	xlabel={number of elements $N = \#\mathcal{T}_\ell$},
	thick,
	mark size = 3,
	legend style={
		cells={anchor=west},
		rounded corners=7pt, 
	},
	]

	\addplot[forget plot, domain=4e2:3e5,color=black!50,mark=none,densely dashed] {x^(-1/3)*9.3e0};
 	\node[coordinate,pin=right:{$\mathcal{O}\big(N^{-1/3}\big)$}] at (axis cs:1e5,2.0e-1) {};


	\addplot[forget plot, domain=4e2:3e5,color=black!50,mark=none,densely dashed] {x^(-2/9)*5e0};
 	\node[coordinate,pin=right:{$\mathcal{O}\big(N^{-2/9}\big)$}] at (axis cs:2e4,5.5e-1) {};

	\addplot
		table {./figures/compare_orthogonal/orthogonal_1.00.dat}; 
	\addlegendentry{uniform}

%

		
	\addplot
		table {./figures/compare_orthogonal/orthogonal_0.25.dat}; 
	\addlegendentry{$\theta = 1/4$}

%

	\addplot
		table {./figures/compare_orthogonal/orthogonal_0.12.dat}; 
	\addlegendentry{$\theta = 1/8$}

	\addplot
		table {./figures/compare_orthogonal/orthogonal_0.06.dat}; 
	\addlegendentry{$\theta = 1/16$}	

	

\end{axis}
\end{tikzpicture}

%
%
%
%
%
%
%
%
%
%
%
%
%
%
\vspace*{10mm}
%
%
%
	\pgfplotsset{width=14cm,height=9cm,compat=1.3} 
	\tikzstyle{every pin}=[fill=white,draw=black,font=\footnotesize,rounded corners=5pt]

\begin{tikzpicture}[trim axis left, trim axis right]
\begin{axis}[
	title={Error estimator $\varrho_\ell$ for Scott-Zhang projection in 3D},
	xmode=log,ymode=log,
	xtick={1e1,1e2,1e3,1e4,1e5,1e6},
	xlabel={number of elements $N = \#\mathcal{T}_\ell$},
		thick,
	xmax = 5e5,
	mark size = 3,
	legend style={
		cells={anchor=west},
		rounded corners=7pt, 
	},
	]
	\addplot[forget plot, domain=4e2:3e5,color=black!50,mark=none,densely dashed] {x^(-1/3)*9.3e0};
 	\node[coordinate,pin=right:{$\mathcal{O}\big(N^{-1/3}\big)$}] at (axis cs:1e5,2.0e-1) {};

%

	\addplot[forget plot, domain=4e2:3e5,color=black!50,mark=none,densely dashed] {x^(-2/9)*5e0};
 	\node[coordinate,pin=right:{$\mathcal{O}\big(N^{-2/9}\big)$}] at (axis cs:2e4,5.5e-1) {};

	\addplot
		table {./figures/compare_scottzhang/scottzhang_1.00.dat}; 
	\addlegendentry{uniform}

%


	\addplot
		table {./figures/compare_scottzhang/scottzhang_0.25.dat}; 
	\addlegendentry{$\theta = 1/4$}

%

	\addplot
		table {./figures/compare_scottzhang/scottzhang_0.12.dat}; 
	\addlegendentry{$\theta = 1/8$}
		
	\addplot
		table {./figures/compare_scottzhang/scottzhang_0.06.dat}; 
	\addlegendentry{$\theta = 1/16$}



\end{axis}
\end{tikzpicture}
%
%
%
%
%
%
%
%
%
%
%
%
%
%
%
\end{center}
\caption{Numerical results for $\varrho_\ell$ for uniform and adaptive mesh-refinement using $\theta \in \{1/4, 1/8, 1/16\}$ and $L^2$-orthogonal projection and Scott-Zhang projection in 3D, respectively, plotted over the number of elements $N = \# \TT_\ell$.}
\label{fig:compare_orthogonal_sz}
\end{figure}

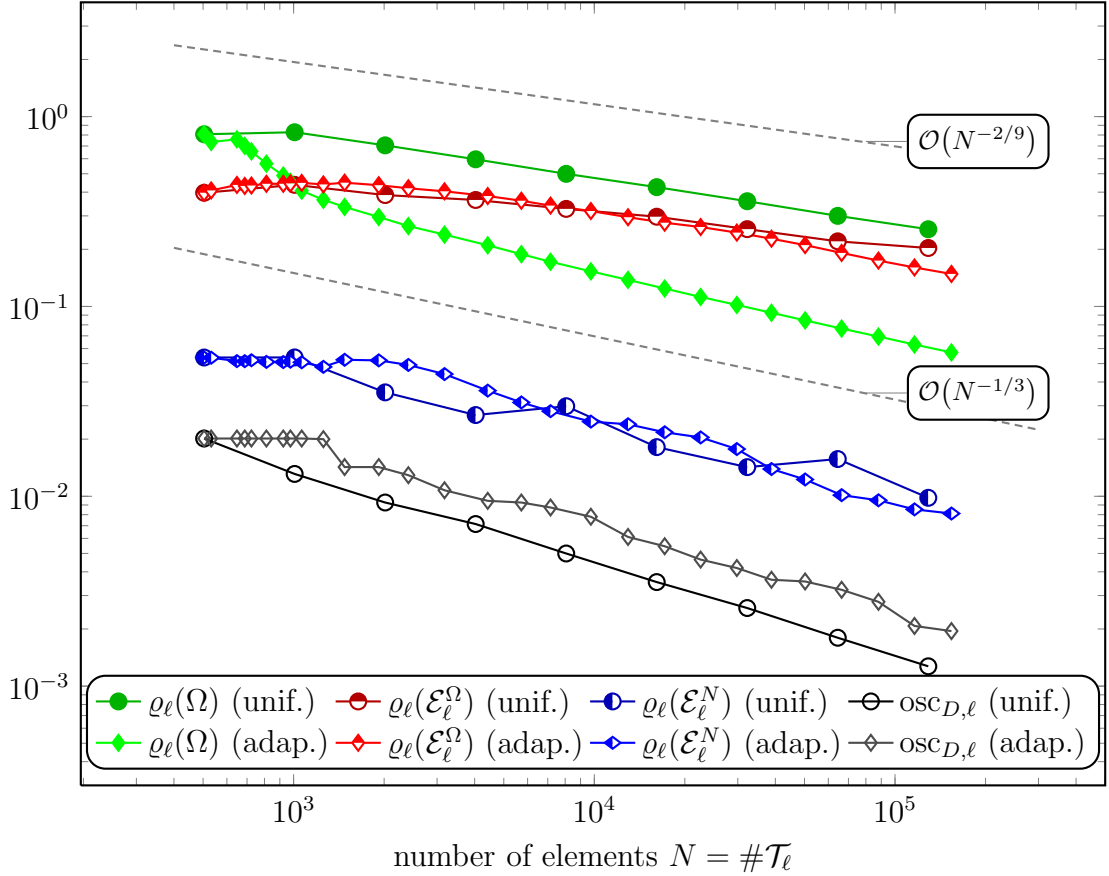
\begin{figure}
\begin{center}
%
%
%
%

	\pgfplotsset{width=15.2cm,height=12cm,compat=1.3} 
	\tikzstyle{every pin}=[fill=white,draw=black,font=\footnotesize,rounded corners=5pt]

\begin{tikzpicture}[trim axis left, trim axis right]
\begin{axis}[
	title=Individual contributions of $\varrho_\ell$ for $L^2$-orthogonal projection in 3D,
	xmode=log,ymode=log,
	xtick={1e1,1e2,1e3,1e4,1e5,1e6},
	xlabel={number of elements $N = \#\mathcal{T}_\ell$},
	ymin = 3e-4,
	ymax = 4e0,
	xmax = 5e5,
	legend style={
		cells={anchor=west},
		anchor=north west, 
		legend columns=4, 
		rounded corners=7pt, 
	},
	legend style={at={(0.5,+0.02)},anchor=south}, 
	mark size= 3,
	thick,
	]

	\addplot[color=green!70!black!100,mark=*] 
		table {./figures/compare_sep/L2_1.00_volume.dat}; 
	\addlegendentry{$\varrho_{\ell}(\Omega)$ (unif.)}
	
	\addplot[color=red!70!black,mark=halfcircle*] 
		table {./figures/compare_sep/L2_1.00_jumps.dat}; 
	\addlegendentry{$\varrho_{\ell}(\mathcal{E}_\ell^\Omega)$  (unif.)}

	\addplot[color=blue!70!black,mark=halfcircle*,every mark/.append style = {rotate=90}] 
		table {./figures/compare_sep/L2_1.00_neumann.dat}; 
	\addlegendentry{$\varrho_{\ell}(\mathcal{E}_\ell^N)$  (unif.)}
	
	\addplot[color=black,mark=o] 
		table {./figures/compare_sep/L2_1.00_dirichlet.dat}; 
	\addlegendentry{$\textrm{osc}_{D,\ell}$ (unif.)}

	\addplot[color=green,mark=diamond*] 
		table {./figures/compare_sep/L2_0.25_volume.dat}; 
	\addlegendentry{$\varrho_{\ell}(\Omega)$ (adap.)}

	\addplot[color=red,mark=halfdiamond*,every mark/.append style = {rotate=180}] 
		table {./figures/compare_sep/L2_0.25_jumps.dat}; 
	\addlegendentry{$\varrho_{\ell}(\mathcal{E}_\ell^\Omega)$  (adap.)}
	
	\addplot[color=blue,mark=halfdiamond*,every mark/.append style = {rotate=-90}] 
		table {./figures/compare_sep/L2_0.25_neumann.dat}; 
	\addlegendentry{$\varrho_{\ell}(\mathcal{E}_\ell^N)$  (adap.)}

	\addplot[color=black!70,mark=diamond] 
		table {./figures/compare_sep/L2_0.25_dirichlet.dat}; 
	\addlegendentry{$\textrm{osc}_{D,\ell}$ (adap.)}

	\addplot[domain=4e2:3e5,color=black!50,mark=none,densely dashed] {x^(-2/9)*0.9e1};
 	\node[coordinate,pin=right:{$\mathcal{O}\big(N^{-2/9}\big)$}] at (axis cs:8e4,7.4e-1) {};
	
	\addplot[domain=4e2:3e5,color=black!50,mark=none,densely dashed] {x^(-1/3)*1.5e0};
 	\node[coordinate,pin=right:{$\mathcal{O}\big(N^{-1/3}\big)$}] at (axis cs:8e4,3.5e-2) {};

\end{axis}
\end{tikzpicture}
%
%
%
%
%
%
%
%
%
%
%
%
\end{center}
\caption{Numerical results for 
$\varrho_{\ell}(\Omega)$, $\varrho_{\ell}(\EE_\ell^\Omega)$, $\varrho_{\ell}(\EE_\ell^N)$ and $\oscD{\ell}$
 for uniform and adaptive mesh-refinement using $\theta = 1/4$ and $L^2$-orthogonal projection in 3D, plotted over the number of elements $N = \# \TT_\ell$.}
\label{fig:compare_sep}
\end{figure}

\subsection{3D example on the Fichera cube}
\noindent
As computational domain serves the Fichera cube $\Omega = (-1,1)^3\backslash[0,1]^3$ which has a concave corner and three reentrant edges. The partition of the boundary $\Gamma = \partial \Omega$ into Dirichlet boundary $\Gamma_D$ and Neumann boundary $\Gamma_N$, as well as the initial surface mesh is shown in Figure~\ref{fig:meshes}. We solve problem~\eqref{eq:strongform} with right-hand side
\begin{align*}
f(x,y,z) := - \frac{5}{16}\,(x^2+y^2+z^2)^{-7/8}.
\end{align*}
The boundary data are prescribed by the trace resp.\ normal derivative of the exact solution
\begin{align*}
u(x,y,z) = (x^2 + y^2 + z^2)^{1/8}
\end{align*}
which has a singular gradient at the reentrant corner at the origin. Similar to the 2D case, we provide comparisons for various adaptivity parameters as well as for different choices for the discretization of the boundary data.

In Figure~\ref{fig:compare_orthogonal_sz}, we compare uniform and adaptive mesh refinement where the Dirichlet data are discretized by means of the $L^2$-orthogonal projection or the Scott-Zhang projection, respectively. The adaptivity parameter is varied between $1/4$ and $1/16$. We observe that either discretization $g_\ell$ of the Dirichlet data $g$ leads to the optimal convergence rate $\mathcal{O}(N^{-1/3})$ for all choices of $\theta$. Due to the generic singularity at the center, uniform refinement leads only to suboptimal convergence rate of $\mathcal{O}(N^{-2/9})$.

In Figure~\ref{fig:compare_sep}, we compare each contribution of the estimator separately for uniform and adaptive refinement with $\theta = 1/4$. For this comparison, we chose the $L^2$-orthogonal projection to discretize $g$. For adaptive refinement, we observe optimal order of convergence even for $\varrho_{\ell}(\EE_\ell^\Omega)$, $\varrho_{\ell}(\Omega)$, $\varrho_{\ell}(\EE_\ell^N)$, and $\oscD{\ell}$. Uniform refinement, on the other hand, leads to suboptimal convergence rate also for the individual contributions.

\begin{figure}
	\begin{center}
	\quad
		\begin{minipage}{.45\textwidth}
%
%
%
%

\begin{tikzpicture}[line join=round, x={(0.61,-0.35)},y={(-1.10cm,-0.17cm)},z={(0cm,1.15cm)},color=black,scale=1.8] 

	\coordinate (W1_1) at (0,1,0);
	\coordinate (W1_2) at (1,1,0);
	\coordinate (W1_3) at (1,0,0);
	\coordinate (W1_4) at (0,0,0);
	\coordinate (W1_5) at ($(W1_4)+(0,0,1)$);
	\coordinate (W1_6) at ($(W1_1)+(0,0,1)$);
	\coordinate (W1_7) at ($(W1_2)+(0,0,1)$);
	\coordinate (W1_8) at ($(W1_3)+(0,0,1)$);

	\coordinate (W2_1) at ($(W1_1)-(0,2,0)$);
	\coordinate (W2_2) at ($(W1_2)-(0,2,0)$);
	\coordinate (W2_3) at (W1_3);
	\coordinate (W2_4) at (W1_4);
	\coordinate (W2_5) at ($(W2_4)+(0,0,1)$);
	\coordinate (W2_6) at ($(W2_1)+(0,0,1)$);
	\coordinate (W2_7) at ($(W2_2)+(0,0,1)$);
	\coordinate (W2_8) at ($(W2_3)+(0,0,1)$);	

	\coordinate (W3_1) at (W1_1);
	\coordinate (W3_2) at ($(W1_2)-(2,0,0)$);
	\coordinate (W3_3) at ($(W1_3)-(2,0,0)$);
	\coordinate (W3_4) at (W1_4);
	\coordinate (W3_5) at ($(W3_4)+(0,0,1)$);
	\coordinate (W3_6) at ($(W3_1)+(0,0,1)$);
	\coordinate (W3_7) at ($(W3_2)+(0,0,1)$);
	\coordinate (W3_8) at ($(W3_3)+(0,0,1)$);	

	\coordinate (W4_1) at (W2_1);
	\coordinate (W4_2) at ($(W2_2)-(2,0,0)$);
	\coordinate (W4_3) at ($(W2_3)-(2,0,0)$);
	\coordinate (W4_4) at (W2_4);
	\coordinate (W4_5) at ($(W4_4)+(0,0,1)$);
	\coordinate (W4_6) at ($(W4_1)+(0,0,1)$);
	\coordinate (W4_7) at ($(W4_2)+(0,0,1)$);
	\coordinate (W4_8) at ($(W4_3)+(0,0,1)$);

	\coordinate (W5_1) at (0,1,2);
	\coordinate (W5_2) at (1,1,2);
	\coordinate (W5_3) at (1,0,2);
	\coordinate (W5_4) at (0,0,2);
	\coordinate (W5_5) at ($(W5_4)-(0,0,1)$);
	\coordinate (W5_6) at ($(W5_1)-(0,0,1)$);
	\coordinate (W5_7) at ($(W5_2)-(0,0,1)$);
	\coordinate (W5_8) at ($(W5_3)-(0,0,1)$);

	\coordinate (W6_1) at ($(W5_1)-(0,2,0)$);
	\coordinate (W6_2) at ($(W5_2)-(0,2,0)$);
	\coordinate (W6_3) at (W5_3);
	\coordinate (W6_4) at (W5_4);
	\coordinate (W6_5) at ($(W6_4)-(0,0,1)$);
	\coordinate (W6_6) at ($(W6_1)-(0,0,1)$);
	\coordinate (W6_7) at ($(W6_2)-(0,0,1)$);
	\coordinate (W6_8) at ($(W6_3)-(0,0,1)$);

	\coordinate (W7_1) at (W5_1);
	\coordinate (W7_2) at ($(W5_2)-(2,0,0)$);
	\coordinate (W7_3) at ($(W5_3)-(2,0,0)$);
	\coordinate (W7_4) at (W5_4);
	\coordinate (W7_5) at ($(W7_4)-(0,0,1)$);
	\coordinate (W7_6) at ($(W7_1)-(0,0,1)$);
	\coordinate (W7_7) at ($(W7_2)-(0,0,1)$);
	\coordinate (W7_8) at ($(W7_3)-(0,0,1)$);

	\coordinate (W8_1) at (W6_1);
	\coordinate (W8_2) at ($(W6_2)-(2,0,0)$);
	\coordinate (W8_3) at ($(W6_3)-(2,0,0)$);
	\coordinate (W8_4) at (W6_4);
	\coordinate (W8_5) at ($(W8_4)-(0,0,1)$);
	\coordinate (W8_6) at ($(W8_1)-(0,0,1)$);
	\coordinate (W8_7) at ($(W8_2)-(0,0,1)$);
	\coordinate (W8_8) at ($(W8_3)-(0,0,1)$);	

	\coordinate (mitte) at (W2_6); 
	\path node at ($(mitte)+(0,-0.8,+1.1)$) {$\Gamma_D$} edge[-latex,pos=0.4,in=10,out=185,thick] (mitte);
	
	\draw[dashed,thick] (W3_2) -- (W4_2) -- (W2_2);
	\draw[dashed,thick] (W4_2) -- (W8_2);
	

\fill[fill=red,opacity=0.65] (W2_2) -- (W4_2)  -- (W8_2) -- (W6_2); 
\fill[fill=blue,opacity=0.6] (W1_2) -- (W2_2)  -- (W6_2) -- (W8_2) -- (W7_2) -- (W3_2); 

\fill[fill=white,opacity=0.01] (W2_2) -- (W1_2)  -- (W1_7) -- (W1_8) -- (W1_8) -- (W6_3) -- (W6_2); 
\fill[fill=white,opacity=0.01] (W7_1) -- (W7_4) -- (W1_5) -- (W1_6);  

\fill[fill=white,opacity=0.3] (W7_2) -- (W8_2)  -- (W6_2) -- (W6_3) -- (W6_4) -- (W7_1) -- (W7_2); 
\fill[fill=white,opacity=0.3] (W1_7) -- (W1_8) -- (W1_5) -- (W1_6); 

\fill[fill=white,opacity=0.2] (W1_2) -- (W3_2)  -- (W7_2) -- (W7_1) -- (W1_6) -- (W1_7) -- (W1_2);  
\fill[fill=white,opacity=0.2] (W1_5) -- (W1_8) -- (W6_3) -- (W6_4);  

\draw (W1_2) -- (W2_2) -- (W6_2) -- (W8_2) -- (W7_2) -- (W3_2) -- cycle;
\draw (W1_2) -- (W1_7) -- (W1_6) -- (W7_1) -- (W7_4) -- (W1_5) -- (W1_8) -- (W6_3) -- (W6_2);
\draw (W1_6) -- (W1_5);
\draw (W1_7) -- (W1_8);
\draw (W6_3) -- (W6_4);

\draw (W7_1) -- (W7_2);

	\draw[loosely dashed] (W1_2) -- (W1_8);
	\draw[loosely dashed] (W1_1) -- (W1_7);
	\draw[loosely dashed] (W1_6) -- (W1_8);

	\draw[loosely dashed] (W1_1) -- (W1_6);
	\draw[loosely dashed] (W1_3) -- (W1_8);

	\draw[loosely dashed] (W2_2) -- (W2_8);
	\draw[loosely dashed] (W2_8) -- (W2_7);
	
	\draw[loosely dashed] (W6_8) -- (W6_2);
	\draw[loosely dashed] (W6_8) -- (W6_4);

	\draw[loosely dashed] (W7_5) -- (W7_1);

	\draw[loosely dashed] (W3_1) -- (W3_7);
	\draw[loosely dashed] (W3_6) -- (W3_7);
	
	\draw[loosely dashed] (W7_7) -- (W7_1) -- (W7_3) -- (W7_4) -- (W8_1);
	
	\draw[loosely dashed] (W7_3) -- (W8_1) -- (W6_3);

		\end{tikzpicture}

%
%
		\end{minipage}
		\quad
		\begin{minipage}{.45\textwidth}
			\includegraphics[width = .99\textwidth]{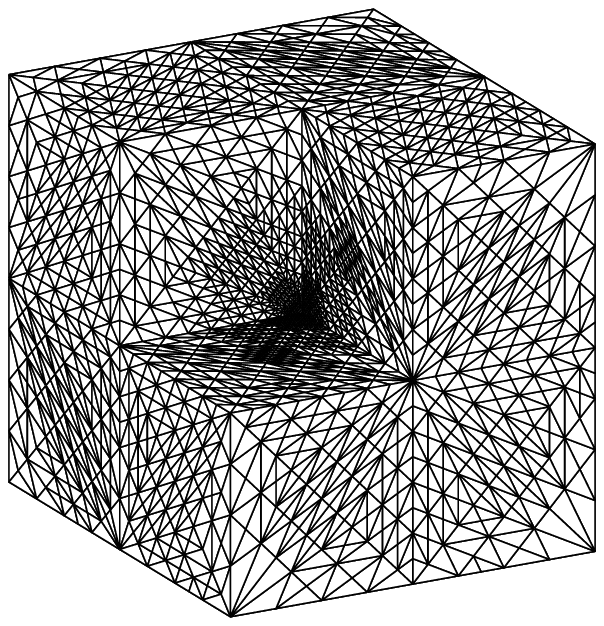}
		\end{minipage}
		
	\end{center}
	\caption{Fichera cube with boundary of the initial mesh $\TT_0$ and $\TT_{30}$ with $N=200.814$ for $\theta = 0.25$. The Dirichlet boundary $\Gamma_D = \{-1\} \times [-1,1]^2$ is marked red, whereas the blue parts denote the Neumann boundary $\Gamma\backslash \Gamma_D$.
	}
	\label{fig:meshes}
\end{figure}

The computational domain, with initial (surface) mesh $\TT_0$ as well as the adaptively generated mesh $\TT_{30}$ with $\#\TT_{30} = 200.814$ elements is finally shown in Figure~\ref{fig:meshes}. As expected, the refinement is basically concentrated around the singularity at the origin.

\bigskip

\noindent \textbf{Acknowledgements} 
The authors M.A., D.P., and M.F.\ are funded by the Austrian Science Fund (FWF) through grant P21732 \emph{``Adaptive Boundary Element Method''}. The author M.P.\ acknowledges support from the Viennese Science and Technology Fund (WWTF) under grant MA09-029 \emph{``Micromagnetic Simulation and Computational Design of Future Devices''}.

\def\new#1{#1}
\newcommand{\bibentry}[2][!]{\ifthenelse{\equal{#1}{!}}{\bibitem{#2}}{\bibitem[#1]{#2}}}


\end{document}